\documentclass[11pt,reqno]{amsart}

\usepackage{amsmath,amssymb,amscd,color}
\usepackage{dutchcal}
\usepackage{graphicx}
\usepackage{epsfig}

\usepackage{pdfpages}
\usepackage{kotex}
\allowdisplaybreaks

\usepackage{times}
\usepackage{tikz}

\numberwithin{equation}{section}
\newtheorem{theo}{Theorem}[section]
\newtheorem{defin}[theo]{Definition}
\newtheorem{prop}[theo]{Proposition}
\newtheorem{coro}[theo]{Corollary}
\newtheorem{lemm}[theo]{Lemma}
\newtheorem{rem}[theo]{Remark}

\newcommand{\Om}{\Omega}

\newcommand{\pa}{\partial}

\begin{document}

\title[Navier-Slip Fluid-Structure Interaction]
{Regularity and Decay for Navier-Slip Fluid-Structure Interaction in Exterior Domains
 }

\author[B. Jin]{Bum Ja Jin}
\address{Department of Mathematics Education, Mokpo National University
\newline\indent
Muan 534-729, South Korea }
\email{bumjajin@mnu.ac.kr}



\thanks{This paper is   supported by  the National Research Foundation of Korea(NRF) grant funded by the Korea government(MSIT)(RS-2026-2546904).}

\begin{abstract}
{

We study the linearized motion of a spherical 
 rigid body immersed in an incompressible viscous fluid occupying the whole space, under Navier slip boundary conditions with friction. While semigroup regularity and decay estimates are well understood for the no-slip case, such results have been limited to bounded domains in the slip setting. By exploiting the geometric regularity of the sphere, we establish  strong $L^q$-regularity, bounded analyticity, and sharp $ L^q-L^r$  estimates for the corresponding fluid-structure operator.
These results would  provide a foundation for constructing strong  $W^{2,1}_q$ solutions and Kato-type solutions to the corresponding nonlinear system.
}
\end{abstract}

\maketitle

\section{\bf Introduction}
\label{introduction}

We consider a homogeneous rigid body immersed in a viscous, incompressible fluid occupying the entire space $\mathbb{R}^3$. Let ${\bf v}$ and $\pi$ denote the velocity and pressure of the fluid, respectively. Let $B \subset \mathbb{R}^3$ represent the initial region occupied by the rigid body, and define the fluid domain as $\Omega := \mathbb{R}^3 \setminus \overline{B}$. At any time $t > 0$, the body occupies the region $B(t)$, and its center of mass is denoted by ${\bf h}(t)$. The translational velocity is $\boldsymbol{\ell}(t) = \dot{\bf h}(t)$ and the angular velocity is $\boldsymbol{\omega}(t)$.
 (Here and in what follows, the dot denotes the derivative with respect to time $t$,, i.e., $\dot{y} = \frac{dy}{dt}$.)
 Consequently, the time-dependent fluid domain is given by
\[
\Omega(t) = \mathbb{R}^3 \setminus \overline{B(t)}.
\]
For the sake of simplicity, we assume that the densities of both the fluid and the rigid body are equal to $1$ throughout this paper (i.e., $\rho_F = 1$ and $\rho_B = 1$). In particular, we consider the case where the rigid body is a ball centered at the origin. However, we note that the general case where the densities may differ is discussed in Remark \ref{rem:general_density}.

Under these assumptions, the motion of the fluid and the rigid body is governed by the following system of equations in $\Omega(t)$:
\begin{equation}
\label{n1}
\begin{aligned}
\frac{\partial {\bf v}}{\partial t}+{\bf v}\cdot \nabla {\bf v}
-\mu\Delta {\bf v}+\nabla{\pi}&={\bf f}_0\\
\operatorname{div}{\bf v}&=0
\end{aligned}\qquad
\mbox{ in } \Omega(t)\mbox{ for }t>0,
\end{equation}
coupled with an ODE
\begin{equation}
\label{n3}
\begin{aligned}
m_0 \frac{d{\boldsymbol{\ell}}}{dt}(t)&=-\int_{\partial \Omega(t)}{\mathbb S}({\bf v},\pi){\bf n}{dS}+m_0{\bf f}_1
\\ 
\frac{d({\mathbb J}\boldsymbol{\omega})}{dt}(t)&=-\int_{\partial \Omega(t)}(x-{\bf h}(t))\times {\mathbb S}({\bf v},\pi)
{\bf n} { dS_x}+{\mathbb J}{\bf f}_2
\end{aligned}\qquad
 \mbox{ for }t>0.
\end{equation}
\noindent Here, $m_0$ and $\mathbb{J}$ denote the total mass and the inertia tensor of the rigid body, respectively, and ${\bf n}$ denotes the unit outward normal vector on $\partial \Omega(t)$. The constant $\mu > 0$ represents the fluid viscosity, and the shear stress tensor is defined as
\[
    \mathbb{S}({\bf v}, \pi) = 2\mu \mathbb{D}({\bf v}) - \pi \mathbb{I}_3,
\]
\noindent where $\mathbb{I}_3$ is the $3 \times 3$ identity matrix. The rate-of-strain tensor $\mathbb{D}({\bf v}) = (D_{ij}({\bf v}))_{i,j=1}^3$ is given by
\[
    D_{ij}({\bf v}) = \frac{1}{2} \left( \frac{\partial v_j}{\partial x_i} + \frac{\partial v_i}{\partial x_j} \right).
\]
For the sake of simplicity, we assume that the rigid body is sufficiently smooth, with the boundary $\partial \Omega$ being of class $C^\infty$.

We prescribe the initial condition:
\begin{align}
\label{n5}
{\bf v}|_{t=0} = {\bf v}_0, \quad \boldsymbol{\ell}|_{t=0} = \boldsymbol{\ell}_0, \quad \boldsymbol{\omega}|_{t=0} = \boldsymbol{\omega}_0.
\end{align}
For $t > 0$, on $\partial \Omega(t)$, we impose the Navier-slip boundary condition:
\begin{equation} \label{n7} 
\begin{cases} 
({\bf v} - \boldsymbol{\ell}- \boldsymbol{\omega} \times (x - {\bf h}(t))) \cdot {\bf n} = 0, \\
 2\mu {\bf n} \times \mathbb{D}({\bf v}){\bf n} + \alpha {\bf n} \times ({\bf v} - \boldsymbol{\ell} - \boldsymbol{\omega} \times (x - {\bf h}(t))) ={\bf  0} 
\end{cases} \text{on } \partial \Omega(t).
\end{equation}
\noindent In these boundary conditions, $\alpha > 0$ is the friction coefficient. The first equation in \eqref{n7} represents the kinematic condition (no-penetration), while the second is the Navier-slip condition, implying that the tangential stress is proportional to the relative tangential velocity between the fluid and the rigid body. For any vector-valued function ${\bf a}$ on $\partial \Omega$, let $[{\bf a}]_\tau := {\bf a} - ({\bf a} \cdot {\bf n}){\bf n}$ denote its tangential component. It is worth noting that the slip condition in $\eqref{n7}_2$ can be equivalently expressed as
\begin{equation} \label{slip_tau_equiv}
2\mu [\mathbb{D}({\bf v}){\bf n}]_\tau + \alpha [{\bf v} - \boldsymbol{\ell} - \boldsymbol{\omega} \times x]_\tau ={\bf  0}.
\end{equation}
Finally, we impose the far-field condition:
\begin{align}
\label{n8}
{\bf v}(x,t) \to {\bf  0} \quad \text{as } |x| \to \infty \quad \text{for } t > 0.
\end{align}

In general, by the product rule and the kinematic identity $\dot{\mathbb{J}}(t)\boldsymbol{\omega} = \boldsymbol{\omega} \times (\mathbb{J}\boldsymbol{\omega})$, the time derivative of the angular momentum expands to $\frac{d}{dt}(\mathbb{J}\boldsymbol{\omega}) = \mathbb{J}\dot{\boldsymbol{\omega}} + \boldsymbol{\omega} \times (\mathbb{J}\boldsymbol{\omega})$. However, we assume from the outset that the initial rigid body $B$ is a homogeneous sphere centered at the origin. We denote by $\mathbb{J}_0$ the constant inertia tensor of the body $B$ at the rest state, defined by
\[
    \mathbb{J}_0 = \left( \int_{B} \rho_B (\delta_{ij} |x|^2 - x_i x_j) \, dx \right)_{ i,j =1}^3,
\]
\noindent where $\rho_B$ is the constant density of the rigid body. Due to the spherical symmetry, we simply have $\mathbb{J}_0 = c \mathbb{I}_3$ for some constant $c > 0$. Furthermore, because a sphere is rotationally invariant, its time-dependent inertia tensor $\mathbb{J}(t)$ does not change over time, meaning $\mathbb{J}(t) = \mathbb{J}_0 = c \mathbb{I}_3$. Because of this property, the gyroscopic torque vanishes identically, i.e., $\boldsymbol{\omega} \times (\mathbb{J}\boldsymbol{\omega}) = c(\boldsymbol{\omega} \times \boldsymbol{\omega}) = {\bf 0}$. Therefore, the angular momentum balance equation  $\eqref{n3}_2$ simplifies to a linear form:
\begin{equation} \label{n8_2_equiv}
    \mathbb{J}_0\dot{\boldsymbol{\omega}}(t) = -\int_{\partial \Omega(t)}( x-{\bf h}(t)) \times \mathbb{S}({\bf v},\pi) {\bf n}\, dS_x + \mathbb{J}_0 {\bf f}_2.
\end{equation}
A number of studies have investigated the fluid-rigid body interaction problem under Navier-slip boundary conditions in bounded domains. In \cite{gerard}, the existence of weak solutions up to collision has been demonstrated. For the uniqueness and continuity of the 2D weak solution in a bounded domain, we refer to \cite{bravin}. In \cite{wang}, both the existence and uniqueness of the 2D strong solution up to collision are established. Furthermore, \cite{necasova.amrouche} shows the global-in-time solvability of 3D strong solutions in the $L^q_t L^p_x$ framework when the initial data is sufficiently small. It is worth noting that in all these works, the fluid region is contained within a bounded container.

To study the problem in an exterior domain, a fundamental first step is to analyze the corresponding linearized system in a fixed spatial domain. Following the approach in \cite{takahashi,wang-xin,ervedoza2}, we use the change of variables $x \mapsto x + {\bf h}(t)$ to rewrite the system \eqref{n1}--\eqref{n8} in the fixed domain $\Omega := \mathbb{R}^3 \setminus \overline{B}$. We then focus on the linearized system around the rest state, which is the main object of our study:
%
\begin{align}
\label{e01}
&\begin{cases}
\begin{aligned}
 \frac{\partial {\bf v}}{\partial t}-\mu\Delta {\bf v}+\nabla{{\pi}}&={\bf f}_0\\
\qquad\operatorname{div}{\bf  v}&=0
\end{aligned}\qquad
\mbox{ in } \Omega\mbox{ for }t>0,\\
\begin{aligned}
&m_0\dot{ \boldsymbol{\ell}} (t)=-\int_{\partial \Omega}{\mathbb S}({\bf v},\pi){\bf n} \, { dS_x}+m_0 {\bf f}_1\\
&{\mathbb J}_0\dot{ \boldsymbol{\omega}}(t)=-\int_{\partial \Omega}x\times
 {\mathbb S}({\bf v},\pi){\bf n} \, { dS_x}+{\mathbb J}_0{\bf f}_2
 \end{aligned}
\qquad \mbox{ for }t>0,
 \end{cases}
 \\
&\tag{C1} {\bf v}|_{t=0} = {\bf v}_0, \quad \boldsymbol{\ell}|_{t=0} = \boldsymbol{\ell}_0, \quad \boldsymbol{\omega}|_{t=0} = \boldsymbol{\omega}_0, \\
&\tag{C2} 
\begin{cases}
({\bf v} - \boldsymbol{\ell} - \boldsymbol{\omega} \times x) \cdot {\bf n} =0 \\
2\mu [\mathbb{D}({\bf v}){\bf n}]_\tau + \alpha[{\bf v} -\boldsymbol{\ell} - \boldsymbol{\omega}\times x]_\tau={\bf 0}
\end{cases} \quad \text{on } \partial \Omega \text{ for } t > 0, \\
&\tag{C3} {\bf v}(x,t) \to {\bf 0} \quad \text{as } |x| \to \infty \quad \text{for } t > 0.
\end{align}

The semigroup associated with the linear system \eqref{e01}, subject to the conditions \textup{(C1)}–\textup{(C3)}, is referred to as the fluid-structure semigroup under Navier-slip boundary conditions, and its generator is the corresponding fluid-structure operator.

%

Previous studies have extensively investigated the operator under no-slip boundary conditions (see \cite{ervedoza2, maity-tucsnak, takahashi, wang-xin, Hishida2024}).
Under Navier-slip conditions, however, the operator has been studied primarily in bounded domains (see, for instance, \cite{necasova.amrouche}, and also \cite{Djebour2024, Eiter2022, EiterShibata2025, Kajiwara2022} for related settings).

Among these works, the results in \cite{necasova.amrouche} are of particular importance to our present study. Let $\mathcal{O}$ be a bounded domain in $\mathbb{R}^3$ with $\mathcal{B} \subset \mathcal{O}$. For the fluid-structure operator on the bounded fluid region $\mathcal{O} \setminus \mathcal{B}$, they proved that the operator generates an analytic semigroup with exponential stability (see \cite[Theorem 4.7]{necasova.amrouche}), and the system enjoys the maximal regularity property (see \cite[Theorem 4.9]{necasova.amrouche}). Moreover, there exists a unique strong solution ${\bf u} \in {\bf L}^q_t({\bf W}^{2,p}_x) \cap {\bf W}^{1,q}_t({\bf L}^p_x)$ to the nonlinear problem (see \cite[Theorem 2.1]{necasova.amrouche}).
Despite these operator-theoretic advancements, the analysis of the fluid-structure operator in an exterior domain ($\Omega = \mathbb{R}^3 \setminus \overline{B}$) remains largely open.


In this paper, we study the linear system \eqref{e01} in the exterior domain $\Omega$ associated with a bounded rigid body $B$. We establish strong $L^q$ regularity and prove that the associated fluid-structure operator generates a bounded analytic semigroup. Furthermore, we derive the $L^r$-$L^q$ decay estimates for the semigroup, which is a crucial step toward understanding the long-time behavior of the original nonlinear problem.

The rest of this paper is organized as follows. 
In Section~\ref{section.mains}, we introduce the general notation and state the main results. Section~\ref{section.fluidstructure} describes the fluid-structure operator, followed by basic computations in Section~\ref{preliminary}.

In Section~\ref{section.maximal.L2}, we establish the foundational $L^2$-theory. By studying the fluid-structure operator on the Hilbert space $\mathbb{X}^2$, we prove its self-adjointness and the resulting maximal $L^2$-regularity. This $L^2$-framework lays the groundwork for our subsequent analysis, as the self-adjointness directly yields sectoriality (i.e., the generation of a bounded analytic semigroup in $L^2$).
In Section~\ref{section.maximal.thm}, we shift our focus to the general $L^q$ framework and establish the maximal $L^q$-regularity on a finite time interval for $1 < q < \infty$. This result is achieved through a localization technique combining whole-space and bounded-domain results, along with a standard density argument.

In Section~\ref{section.analyticity.thm}, we establish the bounded analyticity (sectoriality) of the fluid-structure semigroup for $1 < q < \infty$. We begin with the $L^2$-sectoriality via self-adjointness, then establish the resolvent estimates necessary for $L^q$-sectoriality in the range $1 < q < 3/2$ via a contradiction argument, and finally extend this property to the full range through interpolation and duality.

Transitioning to the asymptotic behavior of the solution, Section~\ref{section.largetime.thm} is devoted to deriving the $L^r$-$L^q$ decay estimates for $1 < r \leq q < \infty$. Our proof splits into short-time and large-time estimates. The large-time behavior is analyzed via global-in-time energy inequalities, Korn-type inequalities under Navier-slip conditions, and the classical Solonnikov--Maremonti approach \cite{maremonti-solonnikov0}.
Finally, in Section~\ref{section.maximal.global.thm}, by utilizing these decay estimates and related tools, we conclude our main analysis by establishing the {uniform-in-time maximal $L^q$-regularity} for the range $1 < q < 3/2$. The endpoint case $q = \infty$ is left for future study.

{
{
\begin{rem}
The linear system \eqref{e01} and the conditions \textup{(C1)}--\textup{(C3)} are derived from the fluid-rigid body interaction problem with a spherical obstacle. However, the same linear system can be obtained from the linearization around the rest state for a rigid body of an arbitrary shape (see, for instance, the appendix of \cite{necasova.amrouche} for the detailed linearization process). 
Specifically, while the time-dependent inertia tensor $\mathbb{J}(t)$ and the gyroscopic torque $\boldsymbol{\omega} \times (\mathbb{J}\boldsymbol{\omega})$ do not vanish in the fully nonlinear system for an arbitrary shape, the torque is quadratic in $\boldsymbol{\omega}$ and $\mathbb{J}(t)$ is approximated by the constant $\mathbb{J}_0$; thus, both terms vanish upon linearization around the rest state.

Mathematically, this linearized system is simply defined on a fixed domain $\Omega$ exterior to a bounded region. Since the main goal of this paper is to study the properties of the linear system \eqref{e01} itself, the specific geometry of this inner bounded region does not matter. Consequently, the results presented in the subsequent sections, including Theorems \ref{maximal.L2}--\ref{largetime.thm}, hold true whether the excluded bounded domain is a sphere or any other arbitrary shape.
\end{rem}
}
}

\section{\bf Notation and Main Theorems} \label{section.mains}

In this section, we briefly introduce the notation and function spaces required for the statement of the main theorems. To maintain focus, we reserve the operator-specific notation for the fluid-structure interaction for Section~\ref{section.fluidstructure}.

Throughout this paper, $c$ denotes a generic positive constant whose value may change from line to line. Its dependence on specific parameters is indicated by $c(\cdot, \dots, \cdot)$. For simplicity, we may omit parameters that remain fixed throughout the context.

Functions are distinguished according to their algebraic nature: scalar-valued functions are written in lightface (e.g., $u$), vector-valued functions in boldface (e.g., ${\bf u}$), and tensor-valued functions in blackboard bold (e.g., $\mathbb{T}$). To ensure a clear distinction, we apply the same typographic convention to their respective function spaces. For instance, $L^q(D)$, ${\bf L}^q(D)$, and $\mathbb{L}^q(D)$ represent the Lebesgue spaces of scalar, vector, and tensor-valued functions, respectively. 
Additionally, blackboard bold symbols are also employed to denote composite product spaces representing the coupled fluid-structure system (e.g., $\mathbb{X}^q$).
Under this convention, the second-order spatial derivative $D^2 {\bf u}$ is naturally understood as an element of a tensor-valued space such as $\mathbb{L}^q(D)$ or $\mathbb{W}^{s,p}(D)$.

Regardless of these typographical distinctions, the associated norms are always written in the standard form $\|\cdot\|_{X}$ without boldface or other decorations. For instance, we simply write $\|{\bf u}\|_{L^q(D)}$ for ${\bf  u} \in {\bf  L}^q(D)$. In particular, when the domain is the entire space $\mathbb{R}^3$, we often omit the domain and write $\|u\|_{L^q}$ for the $L^q(\mathbb{R}^3)$-norm. For any point $x \in \mathbb{R}^3$, its Euclidean norm is denoted by $|x|$.

For a domain $D \subset \mathbb{R}^3$, we denote by $C_0^k(D)$ the space of scalar-valued $C^k$ functions with compact support in $D$ ($0 \le k \le \infty$). For $1 \le q \le \infty$, $L^q(D)$ denotes the usual Lebesgue space. For $s \ge 0$ and $1 \le q < \infty$, $W^{s,q}(D)$ denotes the classical Sobolev space when $s \in \mathbb{N}_0$, and the Sobolev--Slobodeckij space when $s \notin \mathbb{N}_0$ (which coincides with the Besov space $B^s_{q,q}(D)$ for sufficiently regular domains $D$). We denote by $W_0^{k,q}(D)$ ($k \in \mathbb{N}_0$) the closure of $C_0^\infty(D)$ in $W^{k,q}(D)$, and define boundary spaces such as $W^{s,q}(\partial\Omega)$ in the standard way via local charts.

We denote by ${\bf C}_{0,\sigma}^\infty(D)$ the space of smooth, divergence-free vector fields compactly supported in $D$. For $1 < q < \infty$, let ${\bf  L}^q_\sigma(D)$ be the space of vector fields in ${\bf L}^q(D)$ that are divergence-free in the sense of distributions. For the exterior domain $\Omega$, since the normal trace ${\bf v}\cdot{\bf n}$ is well-defined in $W^{-1/q,q}(\partial\Omega)$ for every ${\bf v} \in {\bf L}^q_\sigma(\Omega)$, we define
\[
   {\bf L}^q_{\sigma,\tau}(\Omega) := \{ {\bf v} \in {\bf L}^q_\sigma(\Omega) : {\bf v} \cdot {\bf n} = 0 \text{ on } \partial\Omega \},
\]
which coincides with the closure of ${\bf C}_{0,\sigma}^\infty(\Omega)$ in ${\bf L}^q(\Omega)$ (see \cite[Section  III.2]{galdi}).
%
%

For $1 < q < \infty$, the Stokes operator $A_{\alpha,q} : \mathcal{D}(A_{\alpha,q}) \subset {\bf L}^q_{\sigma,\tau}(\Omega) \to {\bf L}^q_{\sigma,\tau}(\Omega)$ is defined by $A_{\alpha,q} {\bf v} := -P \Delta {\bf v}$, where $P : {\bf L}^q(\Omega) \to {\bf L}^q_{\sigma,\tau}(\Omega)$ denotes the Helmholtz projection. Its domain is given by
\[
\mathcal{D}(A_{\alpha,q}) := \left\{ {\bf v} \in {\bf L}^q_{\sigma,\tau}(\Omega) \cap {\bf W}^{2,q}(\Omega) : 2\mu [\mathbb{D}({\bf v}){\bf n}]_\tau + \alpha [{\bf v}]_\tau = {\bf 0} \text{ on } \partial \Omega \right\}.
\]

For any $T > 0$, we define the space-time domain $Q_T := \Omega \times (0, T)$. Given an interpolation couple of Banach spaces $(X_1, X_2)$, we denote by $(X_1, X_2)_{a,q}$ ($0 < a < 1$, $1 \le q \le \infty$) the real interpolation space between $X_1$ and $X_2$. Moreover, we write ${\bf C}_0^2(\overline{\Omega})$ for the space of $C^2$ vector fields on $\overline{\Omega}$ with compact support in $\overline{\Omega}$ (not necessarily vanishing on $\partial\Omega$).

For later use, we say that a triple $({\bf v}_0, \boldsymbol{\ell}_0, \boldsymbol{\omega}_0) \in {\bf W}^{2-\frac{2}{q},q}(\Omega) \times \mathbb{R}^3 \times \mathbb{R}^3$ ($1 < q < \infty$) satisfies the compatibility condition if
\begin{equation} \label{n66}
\begin{aligned}
&{\bf v}_0 - \boldsymbol{\varphi}_0 \in ({\bf L}^q_{\sigma,\tau}(\Omega), \mathcal{D}(A_{\alpha,q}))_{1-\frac{1}{q},\,q} \\
&\text{for some } \boldsymbol{\varphi}_0 \in {\bf C}^2_0(\overline{\Omega}) \text{ such that } \\
&\begin{cases}
\mathrm{div}\,\boldsymbol{\varphi}_0 = 0 & \text{in } \Omega,\\[2mm]
(\boldsymbol{\varphi}_0 - \boldsymbol{\ell}_0 - \boldsymbol{\omega}_0 \times x)\cdot {\bf n} = 0 & \text{on } \partial \Omega,\\[2mm]
2\mu [\mathbb{D}(\boldsymbol{\varphi}_0){\bf n}]_\tau + \alpha [\boldsymbol{\varphi}_0 - \boldsymbol{\ell}_0 - \boldsymbol{\omega}_0 \times x]_\tau ={\bf 0} & \text{on } \partial \Omega.
\end{cases}
\end{aligned}
\end{equation}

For initial data with lower regularity, specifically for $({\bf v}_0, \boldsymbol{\ell}_0, \boldsymbol{\omega}_0) \in {\bf L}^q(\Omega) \times \mathbb{R}^3 \times \mathbb{R}^3$, the compatibility condition is understood in a reduced sense: we require ${\bf v}_0 - \boldsymbol{\varphi}_0 \in {\bf L}^q_{\sigma,\tau}(\Omega)$ for some $\boldsymbol{\varphi}_0 \in {\bf C}_0^2(\overline{\Omega})$ satisfying the first two conditions in \eqref{n66} (equivalently, ${\bf v}_0 \in {\bf L}^q_\sigma(\Omega)$ with $({\bf v}_0 - \boldsymbol{\ell}_0 - \boldsymbol{\omega}_0 \times x)\cdot{\bf n} = 0$ on $\partial\Omega$).

This definition of the compatibility condition is motivated by the functional framework established for bounded domains in \cite{necasova.amrouche}. Since we intend to adapt and employ several fundamental estimates from the bounded domain case to our exterior domain setting, we adopt a consistent notion of compatibility. This ensures a natural transition of results between different domain geometries and provides a unified basis for our regularity analysis.

To establish strong solutions, we next recall the notion of maximal $L^q$-regularity,
which will play a central role in the analysis of the fluid–structure operator.
\begin{defin}
Let $X$ be a Banach space.
Let $1<q<\infty$ and $0<T\leq \infty$. 
Consider the Cauchy problem 
\[ \frac{dy}{dt} + Ay = f, \quad t \in (0,T), \quad y(0)=0 \]
for a given $f \in L^q(0,T; X)$.
We say $A$ has maximal $L^q$ regularity in $X$ on $(0,T)$ if for every such $f$, there exists a unique solution $y$ satisfying $\frac{dy}{dt}, Ay \in L^q(0,T; X)$, and
\[
\left\| \frac{dy}{dt} \right\|_{L^q(0,T;X)} + \| Ay \|_{L^q(0,T;X)} \leq c\|f\|_{L^q(0,T;X)}
\]
for some constant $c>0$ independent of $f$.
\end{defin}

{
{
Our first main result establishes the foundational $L^2$-theory. By exploiting the self-adjointness of the fluid-structure operator on the Hilbert space ${\mathbb X}^2$ (whose precise definition is deferred to Section \ref{section.fluidstructure}), we obtain the maximal $L^2$-regularity for the linear system \eqref{e01}.

\begin{theo}[Self-adjointness and Maximal $L^2$-regularity]
\label{maximal.L2}
Let $\Omega$ be the exterior domain introduced in Section 1, and let $0<T<\infty$.
Suppose that the initial data $({\bf v}_0,\boldsymbol{\ell}_0,\boldsymbol{\omega}_0)\in {\bf W}^{1,2}(\Omega)\times {\mathbb R}^3\times {\mathbb R}^3$ satisfies the compatibility condition \eqref{n66}, and that the forcing terms $({\bf f}_0,{\bf f}_1,{\bf f}_2)\in {\bf L}^2(Q_T)\times {\bf L}^2(0,T)\times {\bf L}^2(0,T)$.  

Then, the associated fluid-structure operator is self-adjoint and positive definite, and there exists a unique strong solution $({\bf v},\pi,\boldsymbol{\ell},\boldsymbol{\omega})$ to the system \eqref{e01} on $(0,T)$, subject to the conditions \textup{(C1)}–\textup{(C3)}, satisfying the following regularity estimate:
\begin{align*}
&\|{\bf v}\|_{L^2(0,T;W^{2,2}(\Omega))}+\|\partial_t {\bf v}\|_{L^2(Q_T)}+ \|{\bf v}\|_{L^\infty(0,T;W^{1,2}(\Omega))}\\
&\hspace{20mm}
 +\|\nabla \pi\|_{L^2(Q_T)}+\|\boldsymbol{\ell}\|_{W^{1,2}(0,T)}+\|\boldsymbol{\omega}\|_{W^{1,2}(0,T)}\\
&\leq c(T)\big( \|{\bf v}_0\|_{W^{1,2}(\Omega)}{+}|\boldsymbol{\ell}_0|{+}|\boldsymbol{\omega}_0|
{+}\|{\bf f}_0\|_{L^2(Q_T)}+\|{\bf f}_1\|_{L^2(0,T)}
{+}\|{\bf f}_2\|_{L^2(0,T)}\big).
\end{align*}
\end{theo}

Second, we establish the maximal $L^q$-regularity in the space ${\mathbb X}^q$ for $1 < q < \infty$ using a localization technique and a density argument. The precise result is stated in the following theorem.
\begin{theo}[Maximal $L^q$-regularity]
\label{maximal.thm}

Consider the same geometric setting as in Theorem \ref{maximal.L2}. Let $1 < q < \infty$ and $0 < T < \infty$. 
Suppose that the initial data satisfies $({\bf v}_0, \boldsymbol{\ell}_0, \boldsymbol{\omega}_0) \in {\bf W}^{2-\frac{2}{q},q}(\Omega) \times \mathbb{R}^3 \times \mathbb{R}^3$ along with the appropriate compatibility conditions. Assume further that the external forces satisfy $({\bf f}_0, {\bf f}_1, {\bf f}_2) \in {\bf L}^q(Q_T) \times {\bf L}^q(0, T) \times {\bf L}^q(0, T)$.

Then, there exists a unique strong solution $({\bf v},\pi,\boldsymbol{\ell},\boldsymbol{\omega})$ to the system \eqref{e01}, subject to the conditions \textup{(C1)}–\textup{(C3)}, which satisfies the following estimate:
\begin{align*}
&\|\nabla^2{\bf v}\|_{L^q(Q_T)} + \|\partial_t {\bf v}\|_{L^q(Q_T)} 
+ \|{\bf v}\|_{L^\infty(0,T;W^{2-\frac{2}{q},q}(\Omega))} \\
&\hspace{20mm}+\|\nabla \pi\|_{L^q(Q_T)} + \|\boldsymbol{\ell}\|_{W^{1,q}(0,T)} + \|\boldsymbol{\omega}\|_{W^{1,q}(0,T)} 
\\
&\leq c(T) \big( \|{\bf v}_0\|_{W^{2-\frac{2}{q},q}(\Omega)}
{+}|\boldsymbol{\ell}_0| {+} |\boldsymbol{\omega}_0| {+} \|{\bf f}_0\|_{L^q(Q_T)}
 {+} \|{\bf f}_1\|_{L^q(0,T)} {+} \|{\bf f}_2\|_{L^q(0,T)} \big).
\end{align*}
Here, the positive constant $c(T)$ generally increases with $T$. However, for $1 < q < \frac{3}{2}$, the constant can be chosen independently of $T$, yielding a uniform-in-time estimate.
\end{theo}

\begin{rem}
Although Theorem \ref{maximal.L2} ($q=2$) is a special case of our general $L^q$-result (Theorem \ref{maximal.thm}), it is stated separately to emphasize the self-adjointness of the operator in the Hilbert space setting. This self-adjointness directly implies $L^2$-sectoriality, which serves as a crucial building block for our analysis. 
\end{rem}

The following theorem establishes that the fluid-structure operator is sectorial, and consequently generates a bounded analytic semigroup on the underlying space $\mathbb{X}^q$, which will be precisely defined in Section \ref{section.fluidstructure}.

\begin{theo}[Sectoriality and Analytic Semigroup]
\label{analyticity.thm}
Under the same geometric setting for $\Omega$ as above, let $1 < q < \infty$. 
The fluid-structure operator associated with the system \eqref{e01}, subject to
\textup{(C2)}–\textup{(C3)}, is sectorial with angle $\omega=0$, and consequently generates a bounded analytic semigroup on $\mathbb{X}^q$.

Furthermore, assume that the external forces vanish, i.e., $({\bf f}_0, {\bf f}_1, {\bf f}_2) = ({\bf 0}, {\bf 0}, {\bf 0})$. For any initial data $({\bf v}_0, \boldsymbol{\ell}_0, \boldsymbol{\omega}_0) \in {\bf L}^q_\sigma(\Omega) \times \mathbb{R}^3 \times \mathbb{R}^3$ satisfying the compatibility condition \eqref{n66}, the unique strong solution $({\bf v}, \pi, \boldsymbol{\ell}, \boldsymbol{\omega})$ to the system \eqref{e01} satisfies the following estimates for all $t > 0$:
\[
\|{\bf v}(t)\|_{L^q(\Omega)} + |\boldsymbol{\ell}(t)| + |\boldsymbol{\omega}(t)|
\le c\big(\|{\bf v}_0\|_{L^q(\Omega)} + |\boldsymbol{\ell}_0| + |\boldsymbol{\omega}_0|\big),
\]
\[
\|\partial_t {\bf v}(t)\|_{L^q(\Omega)}
+ |\dot{\boldsymbol{\ell}}(t)| + |\dot{\boldsymbol{\omega}}(t)|
\le ct^{-1}\big(\|{\bf v}_0\|_{L^q(\Omega)} + |\boldsymbol{\ell}_0| + |\boldsymbol{\omega}_0|\big).
\]
\end{theo}

\begin{rem}
In Theorem~\ref{maximal.thm}, we established maximal $L^q$-regularity on a finite time interval, which corresponds to the case of equal temporal and spatial exponents (i.e., $p = q$). Moreover, in Theorem~\ref{analyticity.thm}, we proved that the fluid--structure operator is sectorial with angle $\omega = 0$ (generating a bounded analytic semigroup). It is a well-known fact from abstract parabolic theory that if such a sectorial operator exhibits maximal regularity on a finite time interval for a single temporal exponent, the property automatically extends to all $1 < p < \infty$. This extension relies on vector-valued singular integral theory and real (Marcinkiewicz) interpolation, as established by Benedek, Calder\'on, and Panzone \cite{benedek1962} and applied to abstract parabolic equations by Cannarsa and Vespri \cite{cannarsa1986} (see also Amann \cite{amann1995}). Therefore, our results naturally cover the full maximal $L^p$--$L^q$ regularity framework even when $p \neq q$.

\end{rem}

}
}

Next, we establish the $L^r$-$L^q$ decay estimates for the fluid-structure semigroup.

\begin{theo}[$L^r$-$L^q$ Decay Estimates]
\label{largetime.thm}
Assume the same geometric setting as in Theorem \ref{maximal.L2}. Let $1 < r \le q < \infty$. Assume further that the external forces vanish and that the initial data $({\bf v}_0, \boldsymbol{\ell}_0, \boldsymbol{\omega}_0) \in {\bf L}^r_\sigma(\Omega) \times \mathbb{R}^3 \times \mathbb{R}^3$ satisfies the compatibility condition \eqref{n66}.

For brevity, we denote the norm of the initial data in $\mathbb{X}^r$ by
\[
\|({\bf v}_0,\boldsymbol{\ell}_0,\boldsymbol{\omega}_0)\|_{\mathbb{X}^r} := \|{\bf v}_0\|_{L^r(\Omega)} + m_0^{\frac{1}{r}} |\boldsymbol{\ell}_0| + |\mathbb{J}_0|^{\frac{1}{r}}|\boldsymbol{\omega}_0|.
\]

Then, the unique strong solution $({\bf v}, \pi, \boldsymbol{\ell}, \boldsymbol{\omega})$ to the system \eqref{e01}, subject to the conditions \textup{(C1)}–\textup{(C3)}, satisfies the following estimates for all $t > 0$:
\begin{align*}
\|{\bf v}(t)\|_{L^q(\Omega)} + t\|\partial_t {\bf v}(t)\|_{L^q(\Omega)} 
&\leq c t^{- \frac{3}{2} \left( \frac{1}{r} - \frac{1}{q} \right)}\|({\bf v}_0,\boldsymbol{\ell}_0,\boldsymbol{\omega}_0)\|_{\mathbb{X}^r}, \\
|\boldsymbol{\ell}(t)| + |\boldsymbol{\omega}(t)| + t |\dot{\boldsymbol{\ell}}(t)| + t |\dot{\boldsymbol{\omega}}(t)| 
&\leq c(m) (1+t)^{-\frac{1}{2}- \frac{3}{2} \left( \frac{1}{r} - \frac{1}{m} \right)}\|({\bf v}_0,\boldsymbol{\ell}_0,\boldsymbol{\omega}_0)\|_{\mathbb{X}^r}.
\end{align*}

Moreover, for the short-time scale ($t \le 1$), we have
\begin{align*}
t\|\nabla^2 {\bf v}(t)\|_{L^q(\Omega)} + t\|\nabla \pi(t)\|_{L^q(\Omega)} + t^{\frac{1}{2}}\|\nabla{\bf v}(t)\|_{L^q(\Omega)}
&\leq c t^{-\frac{3}{2}\left( \frac{1}{r}-\frac{1}{q} \right)}\|({\bf v}_0,\boldsymbol{\ell}_0,\boldsymbol{\omega}_0)\|_{\mathbb{X}^r},
\end{align*}
and for the large-time scale ($t \ge 1$), the following refined estimates hold:
\begin{align*}
\|\nabla{\bf v}(t)\|_{L^q(\Omega)} 
&\leq c(m)\left( t^{-\frac{1}{2} - \frac{3}{2}\left(\frac{1}{r} - \frac{1}{q}\right)} + t^{-\frac{3}{2r}+\frac{1}{m}} \right)\|({\bf v}_0,\boldsymbol{\ell}_0,\boldsymbol{\omega}_0)\|_{\mathbb{X}^r}, \\
\|\nabla^2 {\bf v}(t)\|_{L^q(\Omega)} + \|\nabla \pi(t)\|_{L^q(\Omega)} 
&\leq c(m)\left( t^{-1 - \frac{3}{2}\left(\frac{1}{r} - \frac{1}{q}\right)} + t^{-\frac{3}{2r}+\frac{1}{m}} \right) \|({\bf v}_0,\boldsymbol{\ell}_0,\boldsymbol{\omega}_0)\|_{\mathbb{X}^r}.
\end{align*}
Here, the constant $c(m)$ depends on $m$, which can be chosen as an arbitrarily large integer.
\end{theo}

\section{\bf Fluid-structure operator}
\label{section.fluidstructure}

In this section, we introduce the fluid-structure operator associated with the motion of a rigid body immersed in a viscous fluid. Unless stated otherwise, we work throughout the paper in an exterior domain $\Omega$, where the fluid satisfies the Navier-slip boundary condition at the interface $\partial \Omega$.

For $1<q<\infty$, we define the fluid-structure composite space
\[
\mathbb{X}^q := \left\{ ({\bf v},\boldsymbol{\ell},\boldsymbol{\omega}) \in {\bf L}^q_\sigma(\Omega)\times\mathbb{R}^3\times\mathbb{R}^3 : {\bf v}\cdot {\bf n} = (\boldsymbol{\ell} + \boldsymbol{\omega}\times x)\cdot {\bf n} \text{ on }\partial \Omega \right\}.
\]
This is a Banach space equipped with the norm
\[
\|({\bf v},\boldsymbol{\ell},\boldsymbol{\omega})\|_{\mathbb{X}^q} := \|{\bf v}\|_{L^q(\Omega)} + m_0^{\frac{1}{q}}|\boldsymbol{\ell}| + |{\mathbb J}_0|^{\frac{1}{q}}|\boldsymbol{\omega}|.
\]
Its dual space is identified with $\mathbb{X}^{q'}$ (where $q' := \frac{q}{q-1}$), and the duality pairing is given by
\[
\langle ({\bf v},\boldsymbol{\ell},\boldsymbol{\omega}), (\boldsymbol{\varphi},{\bf g},\boldsymbol{\theta}) \rangle := \int_{\Omega} {\bf v}\cdot \boldsymbol{\varphi} \, dx + m_0\boldsymbol{\ell}\cdot{\bf g} + ({\mathbb J}_0\boldsymbol{\omega})\cdot \boldsymbol{\theta}
\]
for any $({\bf v},\boldsymbol{\ell},\boldsymbol{\omega}) \in \mathbb{X}^q$ and $(\boldsymbol{\varphi},{\bf g},\boldsymbol{\theta}) \in \mathbb{X}^{q'}$.

To utilize whole-space techniques, we also define the extended space
\[
\tilde{\mathbb X}^q := \left\{ {\bf u} \in {\bf L}^q_\sigma(\mathbb{R}^3) : \mathbb{D}({\bf u}) = 0 \text{ in } B \right\},
\]
equipped with the standard norm $\|{\bf u}\|_{\tilde{\mathbb X}^q} := \|{\bf u}\|_{L^q(\mathbb{R}^3)}$. It is clear that the space of smooth, divergence-free functions that are rigid in $B$ is dense in $\tilde{\mathbb X}^q$.

For a given state ${\bf u} \in \tilde{\mathbb X}^q$, its fluid, translational, and angular velocity components are extracted as follows:
\[
{\bf v}_{\bf u} := {\bf u}|_\Omega, \quad \boldsymbol{\ell}_{\bf u} := m_0^{-1} \int_B {\bf u}\, dx, \quad \boldsymbol{\omega}_{\bf u} := \mathbb{J}_0^{-1} \int_B x \times {\bf u}\, dx.
\]
It follows that ${\bf u} \in \tilde{\mathbb X}^q$ if and only if $({\bf v}_{\bf u}, \boldsymbol{\ell}_{\bf u}, \boldsymbol{\omega}_{\bf u}) \in \mathbb{X}^q$. Furthermore, their norms are equivalent, i.e., $\|{\bf u}\|_{\tilde{\mathbb X}^q} \approx \|({\bf v}_{\bf u}, \boldsymbol{\ell}_{\bf u}, \boldsymbol{\omega}_{\bf u})\|_{\mathbb{X}^q}$ (see, e.g., \cite[Lemma 1.1]{temam1}). From now on, we simply write ${\bf u} \simeq ({\bf v}, \boldsymbol{\ell}, \boldsymbol{\omega})$ whenever $({\bf v}_{\bf u}, \boldsymbol{\ell}_{\bf u}, \boldsymbol{\omega}_{\bf u}) = ({\bf v}, \boldsymbol{\ell}, \boldsymbol{\omega})$.

The whole-space Lebesgue space ${\bf L}^q(\mathbb{R}^3)$ admits a Helmholtz-Weyl type decomposition adapted to the fluid-structure interaction: ${\bf L}^q(\mathbb{R}^3) = \tilde{\mathbb{X}}^q \oplus G^q_1 \oplus G^q_2,$ where
\[
G^q_1 := \left\{ {\bf u} \in {\bf L}^q(\mathbb{R}^3) \;\middle|\; {\bf u} = \nabla q_1 \text{ for some } q_1 \in L^1_{\mathrm{loc}}(\mathbb{R}^3) \right\},
\]
and
\begin{align*}
G^q_2 &:= \Big\{ {\bf u} \in {\bf L}^q_\sigma(\mathbb{R}^3) \;\Big|\; {\bf u} = \nabla q_2 \text{ in } \Omega \text{ for some } q_2 \in L^1_{\mathrm{loc}}(\mathbb{R}^3), \\
&\qquad\qquad {\bf u} = \boldsymbol{\varphi} \text{ in } B \text{ for some } \boldsymbol{\varphi} \in L^q(B) \text{ such that } \\
&\qquad\qquad \int_B \boldsymbol{\varphi}\, dx = -\int_{\partial \Omega} q_2 {\bf n}\, dS_x, \quad \int_B \boldsymbol{\varphi} \times x\, dx = -\int_{\partial \Omega} q_2 {\bf n} \times x\, dS_x \Big\}.
\end{align*}
This fundamental decomposition and the boundedness of the associated projection were established in \cite{wang-xin} (see also \cite{baejin2023, dashti_robinson, ervedoza2}). We state this result formally below.

\begin{prop}[Theorem 2.2 in \cite{wang-xin}]
\label{helmholtz}
Let $1 < q < \infty$. For any ${\bf u} \in {\bf L}^q(\mathbb{R}^3)$, there exists a unique triple $({\bf v}, \nabla q_1, {\bf w}) \in \tilde{\mathbb{X}}^q \times G^q_1 \times G^q_2$ such that ${\bf u} = {\bf v} + \nabla q_1 + {\bf w}$. Furthermore, there exists a constant $c > 0$ such that
\[
\|{\bf v}\|_{L^q(\mathbb{R}^3)} + \|\nabla q_1\|_{L^q(\mathbb{R}^3)} + \|{\bf w}\|_{L^q(\mathbb{R}^3)} \le c \|{\bf u}\|_{L^q(\mathbb{R}^3)}.
\]
Consequently, the associated Helmholtz projection $\mathbb{P}_q : {\bf L}^q(\mathbb{R}^3) \to \tilde{\mathbb{X}}^q$ defined by $\mathbb{P}_q {\bf u} := {\bf v}$ is a bounded linear operator, and its adjoint is given by $(\mathbb{P}_q)^* = \mathbb{P}_{q'}$.
\end{prop}

%
%
%
With this whole-space framework and the Helmholtz projection $\mathbb{P}_q$ at our disposal, we are now in a position to define the fluid-structure operator. The core idea is to regard the rigid body velocity inside $B$ as a natural continuation of the fluid velocity in $\Omega$. This allows us to construct a global operator on $\mathbb{R}^3$ and fully utilize the Helmholtz-Weyl decomposition.

Let $\mathcal{A}$ be the linear operator formally defined for sufficiently smooth vector fields $\boldsymbol{\phi}$ on $\mathbb{R}^3$ by
\begin{equation}
\label{linearA}
\mathcal{A}\boldsymbol{\phi} :=
\begin{cases}
-\mu \Delta \boldsymbol{\phi}, & \text{in } \Omega,\\[4pt]
\displaystyle 2\mu m_0^{-1}\int_{\partial \Omega}\mathbb{D}(\boldsymbol{\phi}){\bf n}\, dS_y
+ 2\mu\left(\mathbb{J}_0^{-1}\!\int_{\partial \Omega} y\times \mathbb{D}(\boldsymbol{\phi}){\bf n}\, dS_y\right)\times x,
& \text{in } B.
\end{cases}
\end{equation}

Using the Helmholtz projection $\mathbb{P}_q$ on ${\bf L}^q(\mathbb{R}^3)$, we rigorously define the fluid-structure operator as
\[
\mathbb{A}_{\alpha,q} := \mathbb{P}_q \mathcal{A},
\]
with its domain given by
\[
\begin{aligned}
\mathcal{D}(\mathbb{A}_{\alpha,q})
:= \Big\{ {\bf u}\in \tilde{\mathbb{X}}^q \;\Big|\;&
{\bf v}_{\bf u} \in {\bf W}^{2,q}(\Omega),\\
&2\mu[\mathbb{D}({\bf v}_{\bf u}){\bf n}]_\tau
+\alpha[{\bf v}_{\bf u} - \boldsymbol{\ell}_{\bf u} - \boldsymbol{\omega}_{\bf u}\times x]_\tau={\bf 0}
\text{ on }\partial \Omega
\Big\}.
\end{aligned}
\]

\begin{rem} \label{rem:general_density}
For simplicity, we have assumed the fluid and rigid body densities to be $1$. To treat the general case with distinct constant densities $\rho_F$ and $\rho_B$, we introduce the global piecewise constant density $\rho(x) := \rho_F \chi_{\Omega} + \rho_B \chi_{B}$. The duality pairing on ${\bf L}^q(\mathbb{R}^3)$ is then endowed with the weight $\rho$:
\[
\langle {\bf f}, {\bf g} \rangle_{\rho} := \int_{\mathbb{R}^3} \rho(x) {\bf f} \cdot {\bf g} \, dx.
\]
For ${\bf f} \in \mathbb{X}^q$ and ${\bf g} \in \mathbb{X}^{q'}$, this naturally reduces to the physical representation $\rho_F \int_{\Omega} {\bf v}_{\bf f} \cdot {\bf v}_{\bf g} \, dx + m \boldsymbol{\ell}_{\bf f} \cdot \boldsymbol{\ell}_{\bf g} + (\mathbb{J} \boldsymbol{\omega}_{\bf f}) \cdot \boldsymbol{\omega}_{\bf g}$, where $m$ and $\mathbb{J}$ are the actual mass and inertia tensor of the rigid body.

Accordingly, the linear operator $\mathcal{A}$ in \eqref{linearA} is modified by scaling the fluid part by $\rho_F^{-1}$. Furthermore, the Helmholtz-Weyl decomposition remains valid under this weighted setting, where the component spaces are adjusted as follows:
\begin{itemize}
    \item $G^q_1 := \left\{ {\bf u} \in {\bf L}^q(\mathbb{R}^3) \mid \rho {\bf u} = \nabla q_1 \text{ for some } q_1 \in L^1_{\mathrm{loc}}(\mathbb{R}^3) \right\}$,
    \item $G^q_2$ is generalized by incorporating the weights into the boundary integrals, specifically:
    \begin{align*}
    G^q_2 &:= \Big\{ {\bf u} \in {\bf L}^q_\sigma(\mathbb{R}^3) \;\Big|\; \rho_F {\bf u} = \nabla q_2 \text{ in } \Omega \text{ for some } q_2 \in L^1_{\mathrm{loc}}(\mathbb{R}^3), \\
    &\qquad\qquad {\bf u} = \boldsymbol{\varphi} \text{ in } B \text{ for some } \boldsymbol{\varphi} \in L^q(B) \text{ such that } \\
    &\qquad\qquad \rho_B \int_B \boldsymbol{\varphi}\, dx = -\int_{\partial \Omega} q_2 {\bf n}\, dS_x, \quad \rho_B \int_B \boldsymbol{\varphi} \times x\, dx = -\int_{\partial \Omega} q_2 {\bf n} \times x\, dS_x \Big\}.
    \end{align*}
\end{itemize}
Consequently, the associated projection $\mathbb{P}_q$ remains bounded and satisfies the adjoint property $\langle \mathbb{P}_q {\bf u}, \boldsymbol{\phi} \rangle_\rho = \langle {\bf u}, \mathbb{P}_{q'} \boldsymbol{\phi} \rangle_\rho$.
\end{rem}

Before introducing the mild formulation, we recall that Theorem~\ref{analyticity.thm} asserts that the operator $-\mathbb{A}_{\alpha,q}$ generates a bounded analytic (and hence strongly continuous) semigroup on $\mathbb{X}^q$. We denote this semigroup by $e^{-\mathbb{A}_{\alpha,q} t}$ for $t \ge 0$.
The mild solution is then defined by
\[
{\bf u}(t) := e^{-\mathbb{A}_{\alpha,q} t} {\bf u}_0 
        + \int_0^t e^{-\mathbb{A}_{\alpha,q} (t-s)} 
          \mathbb{P}_q {\bf f}(s)\, ds,
\]
where ${\bf u}_0 \simeq ({\bf v}_0, \boldsymbol{\ell}_0, \boldsymbol{\omega}_0)$ and ${\bf f} \simeq ({\bf f}_0, {\bf f}_1, {\bf f}_2)$, that is, ${\bf u}_0 = {\bf v}_0\chi_\Omega + (\boldsymbol{\ell}_0 + \boldsymbol{\omega}_0 \times x)\chi_B$ and ${\bf f} = {\bf f}_0\chi_\Omega + ({\bf f}_1 + {\bf f}_2 \times x)\chi_B$.

We introduce the space
\[
\widehat{W}^{1,q}(\Omega) := \left\{ f \in L^q_{\text{loc}}(\overline{\Omega}) : \nabla f \in L^q(\Omega) \right\}.
\]
For each $t \ge 0$, we write ${\bf u}(t) \simeq ({\bf v}(t), \boldsymbol{\ell}(t), \boldsymbol{\omega}(t))$. We will show that there exists a pressure function $\pi \in C([0, \infty); \widehat{W}^{1,q}(\Omega))$ such that the quadruple $({\bf v}, \pi, \boldsymbol{\ell}, \boldsymbol{\omega})$ satisfies the system of equations \eqref{e01} along with the conditions \textup{(C1)}–\textup{(C3)} (see \cite[Proposition~3.2]{maity-tucsnak} for further details on this methodology).


For $1 < q < \infty$, we define the real interpolation space
\begin{equation}
\label{besov}
\mathbb{B}^{2-\frac{2}{q}, q}_{\alpha} := \left( \tilde{\mathbb{X}}^q, \mathcal{D}(\mathbb{A}_{\alpha,q}) \right)_{1 - \frac{1}{q}, q}
\end{equation}
(in particular, $\mathbb{B}^{1,2}_{\alpha}$ when $q = 2$). For any initial data $({\bf v}_0, \boldsymbol{\ell}_0, \boldsymbol{\omega}_0) \in {\bf W}^{2-\frac{2}{q},q}(\Omega) \times \mathbb{R}^3 \times \mathbb{R}^3$ satisfying the compatibility condition \eqref{n66}, the corresponding element ${\bf u}_0 \simeq ({\bf v}_0, \boldsymbol{\ell}_0, \boldsymbol{\omega}_0)$ belongs to $\mathbb{B}^{2-\frac{2}{q}, q}_{\alpha}$. Indeed, for $\boldsymbol{\varphi}_0 \in {\bf C}_0^2(\overline{\Omega})$ as in \eqref{n66}, we have $(\boldsymbol{\varphi}_0, \boldsymbol{\ell}_0, \boldsymbol{\omega}_0) \in \mathcal{D}(\mathbb{A}_{\alpha,q}) \hookrightarrow \mathbb{B}^{2-\frac{2}{q},q}_{\alpha}$. Moreover, since the canonical embedding $\iota({\bf v}) := ({\bf v},{\bf 0}, {\bf 0})$ is bounded from ${\bf L}^q_{\sigma,\tau}(\Omega)$ into $\tilde{\mathbb{X}}^q$ and from $\mathcal{D}(A_{\alpha,q})$ into $\mathcal{D}(\mathbb{A}_{\alpha,q})$, real interpolation yields $\iota({\bf v}_0 - \boldsymbol{\varphi}_0) \in \mathbb{B}^{2-\frac{2}{q},q}_{\alpha}$, and hence ${\bf u}_0 = \iota({\bf v}_0 - \boldsymbol{\varphi}_0) + (\boldsymbol{\varphi}_0, \boldsymbol{\ell}_0, \boldsymbol{\omega}_0) \in \mathbb{B}^{2-\frac{2}{q},q}_{\alpha}$.

%
%

\begin{rem}
\label{rem.3.2}
The formulation of the interpolation space $\mathbb{B}^{2-\frac{2}{q}, q}_{\alpha}$ and the compatibility condition \eqref{n66} via rigid-motion lifting parallels the operator-theoretic framework for fluid--structure interaction with Dirichlet boundary conditions (cf.~\cite[Section 3.1]{maity-tucsnak}), adapted here to the exterior-domain Navier-slip Stokes operator $A_{\alpha,q}$ (cf.~\cite{shibata-slip}).
\end{rem}

\section{\bf Auxiliary Estimates 
}
\label{preliminary}



In this section, we introduce several key estimates that will be used throughout the paper.  
\begin{lemm}
\label{prop.pre2}
Let \( {\bf u} \in \mathcal{D}({\mathbb A}_{\alpha,q}) \) and \( \boldsymbol{\phi} \in \mathcal{D}({\mathbb A}_{\alpha,q'}) \), with
\[ {\bf u}\simeq ({\bf v},\boldsymbol{\ell},\boldsymbol{\omega})\mbox{ and } \boldsymbol{\phi}\simeq (\boldsymbol{\varphi}, {\bf g},\boldsymbol{\theta}).\]
Then,
\begin{align*}
\langle {\mathbb A}_{\alpha,q} {\bf u}, \boldsymbol{\phi} \rangle &= 2\mu \int_{\Omega} {\mathbb D}({\bf v}) : {\mathbb D}(\boldsymbol{\varphi}) \, dx \\
&\quad + \alpha \int_{\partial \Omega} [{\bf v} - \boldsymbol{\ell} -\boldsymbol{ \omega} \times x]_\tau \cdot   [\boldsymbol{\varphi} -{\bf  g} - \boldsymbol{\theta} \times x]_\tau \, dS_x.
\end{align*}
Hence,
\begin{align*}
\langle {\mathbb A}_{\alpha,q} {\bf u}, \boldsymbol{\phi} \rangle& = \langle {\bf u}, {\mathbb A}_{\alpha,q'} \boldsymbol{\phi } \rangle.
\end{align*}
In particular, for $ {\bf u}
\in D({\mathbb A}_{\alpha,2}) $ 
\begin{align*}\langle {\mathbb A}_{\alpha,2} {\bf u}, {\bf u }\rangle& = 2\mu \int_{\Omega} |{\mathbb D}({\bf v})|^2 \, dx  
+ \alpha \int_{\partial \Omega}|[{\bf v} -\boldsymbol{ \ell} -\boldsymbol{ \omega }\times x]_\tau |^2 \, dS_x.
\end{align*}

\end{lemm}

\begin{proof}
Since the dual of \({\mathbb P}_q\) is \({\mathbb P}_{q'}\), the orthogonal decomposition in Proposition \ref{helmholtz} implies
\begin{align*}
&\langle{\mathbb A}_q{\bf u}, \boldsymbol{\phi}\rangle =\langle{\mathcal A} {\bf u}, {\mathbb P}_{q'}\boldsymbol{\phi}\rangle 
=\langle{\mathcal A} {\bf u}, \boldsymbol{\phi}\rangle \\
&=-\mu\int_{\Omega} \Delta{\bf v}\cdot{\boldsymbol{\varphi}} \, dx +  g\cdot \Big(2\mu  \int_{\partial \Omega}{\mathbb D}({\bf v}){\bf n} \, dS_x\Big) +  \theta\cdot \Big(2\mu\int_{\partial \Omega}y\times {\mathbb D}({\bf v}){\bf n}\, dS_x\Big).
\end{align*}
By integration by parts, we have
\begin{align*}
-\mu\int_{\Omega} \Delta{\bf v}\cdot{\boldsymbol{\varphi}} \, dx &= 2\mu \int_{\Omega} {\mathbb D}({\bf v}):{\mathbb D}({\boldsymbol{\varphi}}) \, dx - 2\mu \int_{\partial \Omega} {\boldsymbol{\varphi}}\cdot {\mathbb D}({\bf v}){\bf n} \, dS_x.
\end{align*}
%

By the definition of $\mathcal{D}({\mathbb A}_{\alpha,q})$, for 
${\bf u}\simeq ({\bf v},\boldsymbol{\ell},\boldsymbol{\omega})\in \mathcal{D}({\mathbb A}_{\alpha,q})$ we have
\[
2\mu[ {\mathbb D}({\bf v}){\bf n}]_\tau
+\alpha[{\bf v} -\boldsymbol{ \ell} -\boldsymbol{ \omega}\times x]_\tau={\bf 0}
\quad\text{on }\partial \Omega,
\]
and hence
\begin{align*}
&-2\mu \int_{\partial \Omega}\boldsymbol{\varphi}\cdot {\mathbb D}({\bf v}){\bf n} \, dS_x\\
&=\alpha \int_{\partial \Omega}[{\boldsymbol{\varphi}}-{\bf g}-\boldsymbol{\theta}\times x]_\tau
 \cdot [{\bf v} -\boldsymbol{ \ell} -\boldsymbol{ \omega}\times x]_\tau \, dS_x\\
&\qquad + g\cdot \Big( -2\mu \int_{\partial \Omega} {\mathbb D}({\bf v}){\bf n} \, dS_x\Big) 
+\theta\cdot \Big(-2\mu \int_{\partial \Omega} {x}\times {\mathbb D}({\bf v}){\bf n} \, dS\Big).
\end{align*}
This leads to the identity in our Lemma.
\end{proof}

\begin{lemm}
\label{korn}
Let \( {\bf u} \in {{\bf W}^{1,2}({\Omega}) }\)
with \(\operatorname{div} {\bf u} = 0 \) in \( \Omega \)
 and 
 \(  {\bf u} \cdot {\bf n} = 0 \) on \( \partial \Omega \).
Then,
\[
\int_\Omega |\nabla {\bf u} |^2 \, dx = \int_{\partial \Omega} \kappa(x) |[ {\bf u}]_\tau|^2 \, dS_x+ 2 \int_\Omega |{\mathbb D}( {\bf u} )|^2 \, dx,
\]
where 
$\kappa(x)$ denotes the curvature of $\partial \Omega$, representing the scalar curvature for $n=2$ and the Weingarten endomorphism (shape operator) for $n=3$. 
\footnote{For the geometric definition of the Weingarten endomorphism and its properties, we refer to do Carmo \cite[Section 3-2]{doCarmo}, and for its specific formulation in Navier-slip fluid boundaries, see Bresch et al. \cite[Section 2.1]{bresch}.}
%
\end{lemm}

\begin{proof}
First, assume that ${\bf u} \in {\bf  C}_0^\infty(\overline{\Omega})$ with $\mathrm{div}\,{\bf u} = 0$ in $\Omega$ and ${\bf u} \cdot {\bf n} = 0$ on $\partial\Omega$. In \cite{bresch}, it is shown that such ${\bf u}$ satisfies
\[
[\mathbb{D}({\bf u}){\bf n}]_\tau = \frac{1}{2}\left[\frac{\partial {\bf u}}{\partial {\bf n}}\right]_\tau - \frac{1}{2}\kappa(x)[{\bf u}]_\tau \quad \text{on } \partial\Omega.
\]

Observe that
\[
-\int_\Omega \Delta {\bf u} \cdot {\bf u} \, dx = -\int_{\partial \Omega} {\bf u} \cdot \frac{\partial {\bf u}}{\partial {\bf n}} \, dS_x + \int_\Omega |\nabla{\bf u}|^2 \, dx.
\]
On the other hand, since $\mathrm{div}\ {\bf u} = 0$, we have $u_i \Delta u_i = u_i \partial_{x_j}(\partial_{x_j} u_i + \partial_{x_i} u_j)$, and hence
\[
-\int_\Omega \Delta {\bf u} \cdot {\bf u} \, dx = -2 \int_{\partial \Omega} {\bf u} \cdot \mathbb{D}({\bf u}){\bf n} \, dS_x + 2 \int_\Omega |\mathbb{D}({\bf u})|^2 \, dx.
\]
Combining these two identities with the boundary relation above and ${\bf u} \cdot {\bf n} = 0$ on $\partial\Omega$ leads to
\[
\int_\Omega |\nabla {\bf u}|^2 \, dx = \int_{\partial \Omega} (\kappa(x)[{\bf u}]_\tau) \cdot [{\bf u}]_\tau \, dS_x + 2 \int_\Omega |\mathbb{D}({\bf u})|^2 \, dx.
\]

Since the space of such functions in ${\bf C}_0^\infty(\overline{\Omega})$ is dense in $\{{\bf u} \in {\bf W}^{1,2}(\Omega) \mid \mathrm{div}\,{\bf u} = 0 \text{ in } \Omega, \, {\bf u} \cdot {\bf n} = 0 \text{ on } \partial \Omega\}$ (by a standard cut-off with Bogovski\u{\i} correction and \cite{amrouche.seloula,solonnikov.scadilov}), the same identity extends to all ${\bf u}$ in this subspace by the standard density argument.
\end{proof}


While Lemma \ref{korn} establishes an exact geometric identity for vector fields with strict slip boundary conditions, we often need a more flexible estimate when dealing with exterior domains and rigid body motions. The following lemma provides a generalized Korn-type inequality that bounds the full gradient by the symmetric gradient and a relaxed boundary slip condition.

\begin{lemm}[Korn's type inequality]
\label{korn_exterior}
Let $\Omega$ be an exterior domain.
Assume that ${\bf v}\in W^{1,2}(\Omega),\ \boldsymbol{\ell}, \boldsymbol{\omega}\in {\mathbb R}^3$ satisfying
 that ${\bf v}\cdot {\bf n}=(\boldsymbol{\ell}+\boldsymbol{\omega}\times x )\cdot {\bf n},$ and
\([ {\bf v}-\boldsymbol{\ell}-\boldsymbol{\omega}\times x  ]_\tau \in L^2(\partial \Omega)\).
 
Then, the following inequality holds:
\[
\|\nabla {\bf v}\|_{L^2(\Omega)}+|\boldsymbol{\ell}|+|\boldsymbol{\omega}|\leq c\Big(
\|{\mathbb D}({\bf v})\|_{L^2(\Omega)}+\| [ {\bf v}-\boldsymbol{\ell}-\boldsymbol{\omega}\times x ]_\tau\|_{L^2(\partial \Omega)}
\Big).
\]
\end{lemm}

\begin{proof}


Since the mapping $(\boldsymbol{\ell}, \boldsymbol{\omega}) \mapsto \int_{\partial \Omega} |\boldsymbol{\ell} + \boldsymbol{\omega} \times x|^2 dS$ defines a strictly positive definite quadratic form on $\mathbb{R}^6$, we readily observe the following equivalence for any $\boldsymbol{\ell}, \boldsymbol{\omega} \in \mathbb{R}^3$:
\begin{equation*}
\int_{\partial \Omega} |\boldsymbol{\ell} + \boldsymbol{\omega} \times x|^2 dS \simeq |\boldsymbol{\ell}|^2 + |\boldsymbol{\omega}|^2,
\end{equation*}
where the implicit constants depend only on the geometry of $\partial \Omega$.

Since 
\(
|\boldsymbol{\ell}+\boldsymbol{\omega}\times x |\leq |{\bf v}-\boldsymbol{\ell}-\boldsymbol{\omega}\times x|+|{\bf v}|
\), 
\[
\int_{\partial \Omega}|\boldsymbol{\ell}+\boldsymbol{\omega}\times x |^2dS\leq c\int_{\partial \Omega}|{\bf v}-\boldsymbol{\ell}-\boldsymbol{\omega}\times x|^2 dS+c\int_{\partial \Omega}|{\bf v}|^2 dS.
\]
By the condition $({\bf v}-\boldsymbol{\ell}-\boldsymbol{\omega}\times x )\cdot {\bf n}={\bf 0},$ 
\[
  \int_{\partial \Omega}|{\bf v}-\boldsymbol{\ell}-\boldsymbol{\omega}\times x|^2 dS=\int_{\partial \Omega}|[{\bf v}-\boldsymbol{\ell}-\boldsymbol{\omega}\times x]_\tau|^2 dS.
\]

By the trace theorem and Gagliardo-Nirenberg inequality,
\[
\int_{\partial \Omega}|{\bf v}|^2 dS\leq c\|{\bf v}\|_{W^{\frac{1}{2}}_2(\Omega_R)}^2
\leq c\|\nabla {\bf v}\|_{L^2(\Omega_R)}^2+c\|{\bf v}\|_{L^2(\Omega_R)}^2.
\]
Here $\Omega_R = \Omega \cap B_R$.

Shibata and Soga \cite{shibata-soga} showed that
\[
\| {\bf v}\|_{L^2(\Omega_R)}\leq c
\|\nabla {\bf v}\|_{L^2(\Omega)}.
\]

Combining all the above estimates, we obtain 
\[
|\boldsymbol{\ell}|^2+|\boldsymbol{\omega}|^2\leq \int_{\partial \Omega}|[{\bf v}-\boldsymbol{\ell}-\boldsymbol{\omega}\times x]_\tau|^2 dS+c\|\nabla {\bf v}\|_{L^2(\Omega)}^2.
\]
Taking the square root on both sides and using the fact that Shibata and Soga also showed
\[
\|\nabla {\bf v}\|_{L^2(\Omega)}\leq c
\|{\mathbb D}({\bf v})\|_{L^2(\Omega)},
\]
we immediately arrive at the desired linear bound. This completes the proof of our Lemma.
\end{proof}

\begin{lemm}
\label{strong.reg1}
Let \( {\bf u} \in \mathcal{D}( {\mathbb A}_{\alpha,2}) \), and let ${\bf v}={\bf u}|_\Omega$. Then, 
\[
\|\nabla {\bf  v}\|_{L^2(\Omega)}^2 \leq c \langle {\mathbb A}_{\alpha,2} {\bf u}, {\bf u} \rangle + c \|{\bf u}\|^2_{\tilde{\mathbb X}^2}.
\]
\end{lemm}

\begin{proof}
Let ${\bf u}\simeq ({\bf v},\boldsymbol{\ell},\boldsymbol{\omega})$, and let \({\bf v}_B := \boldsymbol{\ell}+ \boldsymbol{\omega}\times x \) and 
\( \tilde{\bf v}_B =-\frac{1}{2} \nabla \times \Big( \zeta(x) [ x\times \boldsymbol{\ell} + \boldsymbol{\omega}|x|^2 ]\Big) \) 
for a smooth cut-off function \( \zeta \) with \( \zeta = 1 \) near \( \partial \Omega \). 
Then, we have \( \text{div} \, \tilde{\bf v}_B = 0 \) in \( \Omega \) and \( \tilde{\bf v}_B = {\bf v}_B \) on \( \partial \Omega \). 
Hence, by Lemma \ref{korn}, we obtain the identity
\[
\|\nabla ({\bf v}- \tilde{\bf v}_B)\|_{L^2(\Omega)}^2 = 2 \|{\mathbb D}({\bf v}- \tilde{\bf v}_B)\|_{L^2(\Omega)}^2 + \kappa \|[{\bf v}- {\bf v}_B]_\tau\|_{L^2(\partial \Omega)}^2.
\]

Using the standard triangle inequality (e.g., $(a-b)^2 \le 2a^2 + 2b^2$) and the bound \( \|\nabla \tilde{\bf v}_B\|_{L^2(\Omega)}^2 \leq c(|\boldsymbol{\ell}|^2 + |\boldsymbol{\omega}|^2) \leq c\|{\bf u}\|^2_{\tilde{\mathbb X}^2} \), we can deduce that
\[
\|\nabla {\bf v}\|_{L^2(\Omega)}^2 \leq c \|{\mathbb D}({\bf v})\|_{L^2(\Omega)}^2 + c \|[{\bf v}- \boldsymbol{\ell} -\boldsymbol{ \omega} \times x]_\tau \|_{L^2(\partial \Omega)}^2 + c(|\boldsymbol{\ell}|^2 + |\boldsymbol{\omega}|^2).
\]
Recall that
\[
\langle {\mathbb A}_{\alpha,2} {\bf u}, {\bf u} \rangle = 2\mu \int_\Omega |{\mathbb D}({\bf v})|^2 \, dx + \alpha \int_{\partial \Omega} |[{\bf v}-\boldsymbol{ \ell}- \boldsymbol{\omega} \times x]_\tau |^2 \, dS.
\]
Combining these bounds yields the inequality stated in Lemma \ref{strong.reg1}.
\end{proof}

In order to establish the following lemma, let us consider the boundary value problem of the Stokes equations in an exterior domain $\mathcal{O}$ under Navier slip conditions:
%
\begin{align}
\label{stokes}
\begin{cases}
-\mu \Delta {\bf u}+ \nabla{  p} ={\bf f} \qquad \text{in } \mathcal{O}, \\
\operatorname{div}{\bf u} = 0 \qquad \text{in } \mathcal{O}, \\
{\bf u} \cdot {\bf n} = 0, \quad  2\mu[{\mathbb D}({\bf u}){\bf n}]_\tau +\alpha [ {\bf u}]_\tau={\bf 0} \qquad
 \text{on } \partial \mathcal{O}, \\
\lim_{|x| \to \infty} {\bf u}(x) = {\bf 0}.
\end{cases}
\end{align}

\begin{lemm}
\label{prop.ogawa}
Let $\Omega$ be an exterior domain such that $\Omega^c \subset B_{R_0}$ for some $R_0 > 0$. For any $R > 3R_0$, let $\Omega_R = \Omega \cap B_R$.
Let ${\bf f} \in W^{m,r}(\Omega)$ for $m\in {\mathbb N}\cup \{0\}$.
Let \( ({\bf u}, p) \) be a solution to \eqref{stokes} such that \( D^{m+2} {\bf u}, \nabla^{m+1} p \in L^r_{\text{loc}}(\Omega) \).

\begin{itemize}
  \item[(i)] If \( 1 < r < \frac{3}{2} \), then the following estimate holds:
  \[
  \|\nabla^2{\bf u}\|_{W^{m,r}(\Omega)} + \|\nabla p\|_{W^{m,r}(\Omega)} \leq c \|{\bf f} \|_{W^{m,r}(\Omega)}.
  \]

  \item[(ii)] Let \( 1 < r < \infty \). If \({\bf u} \in L^q(\Omega_R) \) for some \( q \in [r, \infty] \), then
  \[
  \|\nabla^2 {\bf u}\|_{W^{m,r}(\Omega)} + \|\nabla p\|_{W^{m,r}(\Omega)} \leq c \|{\bf f} \|_{W^{m,r}(\Omega)} + c \|{\bf u}\|_{L^q(\Omega_R)}.
  \]
\end{itemize}
\end{lemm}

\begin{proof}
The base case $m=0$, along with a summary of known results for various domains, is provided in Appendix~\ref{appendix.prop.ogawa}.  
For higher-order regularity ($m \ge 1$), the result follows via a standard inductive argument. Since the Navier-slip boundary condition satisfies the Lopatinskii--Shapiro complementarity condition and the Stokes system is elliptic in the sense of Agmon--Douglas--Nirenberg, higher-order estimates are obtained by differentiating the equations and applying the base estimate iteratively, provided $\partial\Omega$ is sufficiently smooth (of class $C^{m+2}$).
\end{proof}

\begin{rem}
By Lemma \ref{poincare}, the local $L^q$-norm $\|{\bf u}\|_{L^q(\Omega_R)}$ in \textup{(ii)} can be further bounded by the gradient norm $\|\nabla {\bf  u}\|_{L^q(\Omega_R)}$. 
We note that the generic constant $c > 0$ in these estimates depends on the radius $R$. However, since $R > 3R_0$ is chosen once and fixed throughout the analysis, this dependence poses no difficulty for our subsequent uniform estimates.

Furthermore, while the regularity theory for Stokes equations in bounded domains or with Dirichlet boundary conditions is classical (see, e.g., \cite{galdi, kozono-ogawa, maremonti-solonnikov}), prior works for Navier-slip conditions in exterior domains have predominantly utilized weighted Sobolev spaces to handle decay properties at infinity (see \cite{amrouche0, anis}). In contrast, our estimate in Lemma~\ref{prop.ogawa} is formulated entirely within standard Sobolev spaces, where the local $L^q$-term naturally emerges to capture the interplay between the slip boundary condition and the unbounded domain geometry.
\end{rem}

\begin{lemm}
\label{strong.reg2}
Let ${\bf u} \in \mathcal{D}({\mathbb A}_{\alpha,q})$, and let ${\bf v}={\bf u}|_\Omega$.
Then, for any $m \in \mathbb{N} $, it holds that 
\begin{align}
\label{qn2}
\|{\bf v}\|_{W^{2m,q}(\Omega)}
&\le c\Big( \|{\mathbb A}_{\alpha,q}^{m} {\bf u}\|_{\tilde{\mathbb X}^q}
          + \|{\bf u}\|_{\tilde{\mathbb X}^q} \Big),
\quad \text{if } {\bf u} \in \mathcal{D}({\mathbb A}^{m}_{\alpha,q}).
\end{align}
\end{lemm}

\begin{proof}
Let  ${\bf u}\simeq ({\bf v},\boldsymbol{\ell},\boldsymbol{\omega})$, and let ${\bf F}= {\mathbb A}_{\alpha,q} {\bf u}\big|_\Omega$.

Since ${\mathbb A}_{\alpha,q} = \mathbb{P}\mathcal{A}:\mathcal{D}({\mathbb A}_{\alpha,q})\to\tilde{\mathbb X}^q$,
there exists a pressure $p$ such that
\begin{align*}
-\mu \Delta {\bf v} + \nabla p =  {\bf F},
\quad \operatorname{div} {\bf v} &= 0 \quad \text{in } \Omega,\\
({\bf v} -\boldsymbol{\ell} - \boldsymbol{\omega} \times x)\cdot {\bf n}&=0 \quad \text{on } \partial \Omega,\\
2\mu[{\mathbb D}({\bf v}){\bf n}]_\tau
+ \alpha[{\bf v} - \boldsymbol{\ell} - \boldsymbol{\omega} \times x]_\tau &= {\bf 0}
\quad \text{on } \partial \Omega.
\end{align*}
Let $\tilde{\bf v}_B$ be the auxiliary field constructed in the proof of
Lemma~\ref{strong.reg1}. Applying Lemma~\ref{prop.ogawa} to
${\bf v}- \tilde{\bf v}_B$ yields
\begin{align*}
\|\nabla^2({\bf v}  - \tilde{\bf v}_B)\|_{L^q(\Omega)}
&\leq c\|{\bf F}\|_{L^q(\Omega)}
     +c\|\Delta\tilde{\bf v}_B\|_{L^q(\Omega_R)}
     +c\|{\bf v} - \tilde{\bf v}_B\|_{L^q(\Omega_R)}.
\end{align*}
Hence
\begin{align*}
\|\nabla^2 {\bf v}\|_{L^q(\Omega)}
&\leq c\|{\bf F}\|_{L^q(\Omega)}
     +c|\boldsymbol{\ell}| + c|\boldsymbol{\omega}| + c\|{\bf v}\|_{L^q(\Omega_R)} \\
&\leq c\|{\mathbb A}_{\alpha,q}{\bf u}\|_{\tilde{\mathbb X}^q}
     +c\|{\bf u}\|_{\tilde{\mathbb X}^q}.
\end{align*}
This proves the estimate for $m=1$.

We now prove the general case by induction on $m$. 
Let ${\bf u}\in \mathcal{D}({\mathbb A}_{\alpha,q}^{m+1})$. Then ${\mathbb A}_{\alpha,q} {\bf u}\in D({\mathbb A}_{\alpha,q}^m)$,
so we may apply \eqref{qn2} with ${\bf u}$ replaced by ${\mathbb A}_{\alpha,q} {\bf u}$:
\begin{equation}
\label{qn3}
\|{\mathbb A}_{\alpha,q} {\bf u}|_{\Omega}\|_{W^{2m,q}(\Omega)}
\leq c\Big(\|{\mathbb A}_{\alpha,q}^{m+1} {\bf u}\|_{\tilde{\mathbb X}^q}
          +\|{\mathbb A}_{\alpha,q}{\bf u}\|_{\tilde{\mathbb X}^q}\Big).
\end{equation}
Also, according to Lemma \ref{prop.ogawa} it holds that
\begin{align}
\label{qn4}
\|{\bf v}\|_{W^{2m+2,q}(\Omega)}\leq c\Big(\|{\mathbb A}_{\alpha,q}{\bf u}|_{\Omega}\|_{W^{2m,q}(\Omega)}+\| {\bf v}  \|_{ L^q(\Omega_R)}  \Big).
\end{align}

Since ${\mathbb A}_{\alpha,q}{\bf u}={\mathbb P}_q{\mathcal A}{\bf u}$, 
\[
\|{\mathbb A}_{\alpha,q}{\bf u}\|_{\tilde{\mathbb X}^q}\leq c\|{\mathcal A}{\bf u}\|_{L^q({\mathbb R}^3)} 
\le c\|  {\bf v}\|_{W^{2,q}(\Omega)}.
\]

Now we use a standard interpolation inequality of Gagliardo–Nirenberg type
to obtain, for any $\varepsilon>0$,
\begin{equation}
\label{eq:interp}
\|{\bf v}\|_{W^{2,q}(\Omega)}\leq \varepsilon\|\nabla^{2m+2} {\bf v}\|_{L^q(\Omega)}
   + C_\varepsilon \|{\bf v}\|_{L^q(\Omega)},
\end{equation}
and choosing $\varepsilon>0$ sufficiently small, we finally arrive at the $m+1$ step:
\[
\|\nabla^{2m+2} {\bf v}\|_{L^q(\Omega)}
\leq c\Big(\|{\mathbb A}_{\alpha,q}^{m+1} {\bf u}\|_{\tilde{\mathbb X}^q}
          +\|{\bf u}\|_{\tilde{\mathbb X}^q}\Big).
\]

Thus the estimate holds for $m+1$, and by induction \eqref{qn2} holds for all
$m\in\mathbb{N}$.

\end{proof}

\begin{rem}
The higher-order regularity estimate established in Lemma~\ref{strong.reg2} is directly utilized in the derivation of the short-time estimates in Section~\ref{smallt}.
\end{rem}

We prepare the following estimate of the local pressure for later use in the proof of Lemma \ref{prop.theorem.strong.Lp} and Section \ref{gradient.decay}.

\begin{lemm}
\label{lem:pressure_accel_decomp}
Let $\Omega^c \subset B_{R/4}$, and let $({\bf v}, \pi, \boldsymbol{\ell}, \boldsymbol{\omega})$ be a sufficiently smooth solution to the system \eqref{e01}. Then, the following estimate holds uniformly in time for almost every $t > 0$:
\begin{equation}
\label{est:pressure_accel}
\|\pi(t)\|_{L^q(\Omega_{2R})} 
\leq c\big(\|\nabla {\bf v}(t)\|_{W^{s,q}(\Omega_{2R})} + \|{\bf f}_0(t)\|_{L^q(\Omega)} + |{\bf f}_1(t)| + |{\bf f}_2(t)| \big), \quad \text{for } \frac{1}{q} < s < 1,
\end{equation}
where the constant $c > 0$ is independent of time $t$.
\end{lemm}

\begin{proof}
Observe that for any fixed time $t$, the pressure $\pi$ satisfies the following Neumann problem in the exterior domain $\Omega$:
\begin{align*}
\Delta \pi &= \text{div} {\bf f}_0 \quad \text{in } \Omega, \\
\frac{\partial \pi}{\partial n} &= \mu \Delta {\bf v} \cdot \boldsymbol{n} + {\bf f}_0 \cdot \boldsymbol{n} - \partial_t {\bf v} \cdot \boldsymbol{n} \quad \text{on } \partial \Omega.
\end{align*}
To rigorously estimate the pressure, we decompose $\pi = \pi_1 + \pi_2 + \pi_3$, where each $\pi_i$ solves the corresponding Neumann problem:
\begin{align*}
\Delta \pi_1 &= 0 \quad \text{in } \Omega, \quad
&\frac{\partial \pi_1}{\partial n} &= \mu \Delta {\bf v} \cdot \boldsymbol{n} \quad \text{on } \partial \Omega, \\
\Delta \pi_2 &= \text{div} {\bf f}_0 \quad \text{in } \Omega, \quad
&\frac{\partial \pi_2}{\partial n} &= {\bf f}_0 \cdot \boldsymbol{n} \quad \text{on } \partial \Omega, \\
\Delta \pi_3 &= 0 \quad \text{in } \Omega, \quad
&\frac{\partial \pi_3}{\partial n} &= -(\dot{\boldsymbol{\ell}}+\dot{\boldsymbol{\omega}}\times x)\cdot \boldsymbol{n} \quad \text{on } \partial \Omega,
\end{align*}
where we used the Navier slip boundary condition $({\bf v}-\boldsymbol{\ell}-\boldsymbol{\omega}\times x)\cdot \boldsymbol{n}=0$ on $\partial \Omega$ to explicitly simplify the normal derivative for $\pi_3$.

To uniquely determine the solutions $\pi_1$ and $\pi_3$ in the exterior domain, we impose the natural vanishing condition at infinity ($\pi_i(x) \to 0$ as $|x| \to \infty$). Under this uniqueness condition, the proof of Theorem 8.1 in \cite{solonnikov} 
 ensures that the local norm is bounded by the boundary data:
\begin{align}
\label{pi1_local}
\|\pi_1\|_{L^q(\Omega_{2R})} &\leq c \|\nabla {\bf v}\|_{W^{s_0,q}(\partial \Omega)}, \quad 0 < s_0 < 1 - \frac{1}{q}, \\
\label{pi1_local1}
\|\nabla \pi_1\|_{L^q(\Omega_{2R})} &\leq c \|\nabla {\bf v}\|_{W^{1-\frac{1}{q},q}(\partial \Omega)}. 
\end{align}
Interpolating \eqref{pi1_local} and \eqref{pi1_local1}, we have
\begin{equation}
\label{pi1_interp}
\|\pi_1\|_{W^{a,q}(\Omega_R)} \leq c \|\nabla {\bf v}\|_{W^{a(1-\frac{1}{q}),q}(\partial \Omega)}, \quad \frac{1}{q}<a<1.
\end{equation}

On the other hand, the standard elliptic $L^q$-theory for the weak Neumann problem in $\Omega$ yields that the global gradient is bounded:
\begin{align}
\label{p2_grad}
\|\nabla \pi_2\|_{L^q(\Omega)} \leq c\|{\bf f}_0\|_{L^q(\Omega)}.
\end{align}
Since $\pi_2$ does not necessarily vanish at infinity, we uniquely fix it by imposing the zero-mean condition $\int_{\Omega_{2R}} \pi_2 \, dx = 0$. By the Poincar\'e inequality on $\Omega_{2R}$, we obtain bounds for the local bulk and boundary norms of $\pi_2$:
\begin{align}
\label{pi2_local}
\|\pi_2\|_{L^q(\Omega_{2R})} \leq c \|\nabla \pi_2\|_{L^q(\Omega_{2R})} \leq c\|{\bf f}_0\|_{L^q(\Omega)}.
\end{align}

By utilizing the vanishing condition at infinity for $\pi_3$, we can express it in terms of the rigid body acceleration $A = (\dot{\boldsymbol{\ell}}, \dot{\boldsymbol{\omega}})^T \in \mathbb{R}^6$. To this end, we introduce the classical Kirchhoff potentials $\Phi_i$ ($i=1,\dots,6$), which are harmonic functions in $\Omega$ satisfying $\frac{\partial \Phi_i}{\partial n} = {\bf n}_i$ for $i=1,2,3$ and $\frac{\partial \Phi_i}{\partial n} = ({\bf e}_{i-3} \times x) \cdot {\bf n}$ for $i=4,5,6$ on $\partial \Omega$, along with the decay condition $\Phi_i(x) \to 0$ as $|x| \to \infty$. Using these potentials, we can explicitly represent $\pi_3$ as:
\[
\pi_3(x,t) = - \sum_{j=1}^6 A_j(t) \Phi_j(x).
\]
As is standard in fluid-structure interaction (see, e.g., Galdi \cite[Section 2]{galdi2002}), substituting this representation into the hydrodynamic force and torque integrals yields exactly the added mass term:
\begin{align*}
\begin{pmatrix}
-\int_{\partial \Omega} \pi_3 {\bf n} \, dS_x \\
-\int_{\partial \Omega} x \times \pi_3 {\bf n} \, dS_x
\end{pmatrix}
= - \mathcal{M}_{\text{added}} A,
\end{align*}
where $\mathcal{M}_{\text{added}}$ is the $6 \times 6$ symmetric, strictly positive definite added mass matrix with entries $m_{ij} = \int_{\Omega} \nabla \Phi_i \cdot \nabla \Phi_j \, dx$.
Moving this added mass term to the left-hand side of the Newton-Euler equations, we obtain:
\begin{align*} 
&\left[ \begin{pmatrix}
m_0 {\mathbb  I} & 0 \\
0 & \mathbb{J}_0
\end{pmatrix} 
+ \mathcal{M}_{\text{added}} \right] A
\\
&=
\begin{pmatrix}
-\int_{\partial \Omega} (\pi_1 + \pi_2) {\bf n} \, dS_x + 2\mu\int_{\partial \Omega} \mathbb{D}({\bf v}){\bf n} \, dS_x + m_0 {\bf f}_1 \\
-\int_{\partial \Omega} x \times (\pi_1 + \pi_2) {\bf n} \, dS_x + 2\mu\int_{\partial \Omega} x \times \mathbb{D}({\bf v}){\bf n} \, dS_x + \mathbb{J}_0 {\bf f}_2
\end{pmatrix}.
\end{align*}
Since the total mass matrix is invertible, we directly deduce:
\begin{align}
\label{pi3_local}
\|\pi_3\|_{L^q(\Omega_{2R})} + \|\nabla \pi_3\|_{L^q(\Omega_{2R})} 
&\leq c|A| \notag \\
&\leq c \left( \|\pi_1\|_{L^q(\partial \Omega)} + \|\pi_2\|_{L^q(\partial \Omega)} + \|\nabla{\bf v}\|_{L^q(\partial \Omega)} \right) + c(|{\bf f}_1| + |{\bf f}_2|).
\end{align}

Finally, we bound the boundary norms of $\pi_2$ and $\pi_1$. Using the zero-mean property of $\pi_2$ on $\Omega_{2R}$, we have:
\begin{align}
\label{p2_trace}
\|\pi_2\|_{L^q(\partial \Omega)} \leq c \|\pi_2\|_{W^{1,q}(\Omega_{2R})} \leq c \|\nabla \pi_2\|_{L^q(\Omega_{2R})} \leq c\|{\bf f}_0\|_{L^q(\Omega)}.
\end{align}
For $\pi_1$, fixing an exponent $a$ such that $\frac{1}{q} < a < 1$, the standard trace theorem and \eqref{pi1_interp} yield
\begin{align}
\label{pi5_local}
\|\pi_1\|_{L^q(\partial \Omega)} 
&\leq c\|\pi_1\|_{W^{a,q}(\Omega_R)} 
\leq c \|\nabla {\bf v}\|_{W^{a(1-\frac{1}{q}),q}(\partial \Omega)} \notag \\
&\leq c \|\nabla {\bf v}\|_{W^{s,q}(\Omega_{2R})}, \quad \text{where } s = a\left(1-\frac{1}{q}\right) + \frac{1}{q}.
\end{align}
Note that since $\frac{1}{q} < a < 1$, it follows immediately that $\frac{1}{q} < s < 1$. 

Combining the local estimates established above yields the desired estimate \eqref{est:pressure_accel}.
\end{proof}

The following lemma, which will be used several times in this paper, relies on the well-known properties of the Bogovskii operator proved in Galdi \cite[Lemma III.3.1]{galdi}:
\begin{lemm}
\label{lem.bogovskii_cutoff}
Let $\Omega$ be an exterior domain such that $\Omega^c \subset B_{R_1/2}$ for some $R_1 > 0$. 
For any $R_2 > R_1$, we denote $\Omega_{R_2} = \Omega \cap B_{R_2}$. 
Let $\zeta \in C^\infty(\mathbb{R}^3)$ be a smooth function such that its gradient $\nabla \zeta$ is compactly supported in $B_{R_2} \setminus B_{R_1}$. 
Assume that ${\bf v}$ is a time-dependent, divergence-free vector field in $\Omega$ (i.e., $\operatorname{div} {\bf v} = 0$) satisfying the compatibility condition:
\[
\int_{B_{R_2} \setminus B_{R_1}} {\bf v} \cdot \nabla \zeta \, dx = 0.
\]
Then, there exists an auxiliary vector field ${\bf w}$ defined on $\mathbb{R}^3$ such that:
\begin{enumerate}
    \item[(i)] $\operatorname{div} {\bf w} = {\bf v} \cdot \nabla \zeta$ in ${\mathbb R}^3$;
    \item[(ii)] ${\bf w}$ is compactly supported in $B_{R_2} \setminus B_{R_1}$, so that ${\bf w} = {\bf 0}$ in a small neighborhood of $\partial\Omega$ and outside $B_{R_2}$;
    \item[(iii)] For $q \in (1, \infty)$, the following estimates hold:
    \begin{align} 
    \label{grad1}
    \|\nabla {\bf w}\|_{L^q(\Omega)}  &\leq c\|{\bf v}\|_{L^q(\Omega_{R_2})}, \nonumber \\
    \|\Delta {\bf w}\|_{L^q(\Omega)}  &\leq c\|{\bf v}\|_{L^q(\Omega_{R_2})} + c\|\nabla{\bf v}\|_{L^q(\Omega_{R_2})}, \nonumber \\
  \|\partial_t {\bf w}\|_{L^q(\Omega)} &\leq c\|\partial_t {\bf v}\|_{W^{-1,q}(\Omega_{R_2})}.
    \end{align}
\end{enumerate}
\end{lemm}

\begin{proof}
By the assumed compatibility condition, we can apply the standard Bogovskii operator $\mathbb{B}$ (see Lemma III.3.1 of Galdi \cite{galdi}) on the annulus $B_{R_2} \setminus B_{R_1}$ to obtain a vector field $\mathbb{B}({\bf v} \cdot \nabla \zeta) \in W^{1,q}_0(B_{R_2} \setminus B_{R_1})$ satisfying $\operatorname{div} (\mathbb{B}({\bf v} \cdot \nabla \zeta)) = {\bf v} \cdot \nabla \zeta$.

We then define ${\bf w}$ as the zero extension of $\mathbb{B}({\bf v} \cdot \nabla \zeta)$ to the whole space $\mathbb{R}^3$. Because the original Bogovskii function vanishes identically on the inner and outer boundaries of the annulus ($\partial B_{R_1}$ and $\partial B_{R_2}$), its zero extension across these boundaries is smooth and strictly supported inside $B_{R_2} \setminus B_{R_1}$. Consequently, ${\bf w}$ is compactly supported, vanishes identically in a small neighborhood of $\partial\Omega$, and satisfies the spatial estimates in \eqref{grad1} through the boundedness of the Bogovskii operator. 

Finally, since $\zeta$ is independent of time, we have $\partial_t {\bf w} = \mathbb{B}(\partial_t {\bf v} \cdot \nabla \zeta)$. 
By duality, since $\nabla \zeta$ is smooth and compactly supported in $B_{R_2} \setminus B_{R_1}$, we readily obtain the bound 
$\|\partial_t {\bf v} \cdot \nabla \zeta\|_{W^{-1,q}(B_{R_2} \setminus B_{R_1})} \leq c\|\partial_t {\bf v}\|_{W^{-1,q}(\Omega_{R_2})}$.
Note that on a bounded domain, any functional in $W^{-1,q}_0$ can be represented as the divergence of an $L^q$ vector field. Thus, the standard estimate for the Bogovskii operator (see, e.g., Galdi \cite[Theorem III.3.1]{galdi}) naturally extends to the negative norm boundedness $\mathbb{B}: W^{-1,q}_0(B_{R_2} \setminus B_{R_1}) \to L^q(B_{R_2} \setminus B_{R_1})$. 
Consequently, applying this boundedness directly yields
\begin{align*}
\|\partial_t {\bf w}\|_{L^q(\Omega)} &= \|\mathbb{B}(\partial_t {\bf v} \cdot \nabla \zeta)\|_{L^q(B_{R_2} \setminus B_{R_1})} \\
&\leq c \|\partial_t {\bf v} \cdot \nabla \zeta\|_{W^{-1,q}(B_{R_2} \setminus B_{R_1})} \leq c\|\partial_t {\bf v}\|_{W^{-1,q}(\Omega_{R_2})}.
\end{align*}
\end{proof}

\section{\bf Proof of Theorem \ref{maximal.L2}}
\label{section.maximal.L2}

For notational convenience in this section, we set
\[
\tilde{\mathbb X} := \tilde{\mathbb X}^2, \qquad
{\mathbb P} := {\mathbb P}_2, \qquad
{\mathbb A}_\alpha := {\mathbb A}_{\alpha,2}.
\]

The following lemma is of fundamental importance: it establishes the self-adjointness of the fluid-structure operator $\mathbb{A}_\alpha$ and, consequently, plays a key role in proving the strong solvability in $W^{2,1}_2(Q_T)$.
\begin{lemm}
\label{lemma.appendix2}
The operator \({\mathbb A}_\alpha : \mathcal{D}({\mathbb A}_\alpha) \subset \tilde{\mathbb X} \to \tilde{\mathbb X}\) is surjective; that is,
\[
\mathrm{Ran}(I + {\mathbb A}_\alpha) = \tilde{\mathbb X},
\]
and satisfies the estimate
\begin{equation}
\label{surjective}
\|{\bf u}\|_{\mathcal{D}({\mathbb A}_\alpha)} \leq \|(I + {\mathbb A}_\alpha)u\|_{\tilde{\mathbb X}} \quad \text{for all } {\bf u} \in \mathcal{D}({\mathbb A}_\alpha).
\end{equation}

\end{lemm}

\begin{proof}


Before going on, we define the function spaces 
\( \tilde{\mathbb Y} \) as follows:
\[
\tilde {\mathbb Y} := \{ {\bf u} \in \tilde {\mathbb X}: \ {\bf u}|_\Omega \in W^{1,2}(\Omega) \},
\]
and define the bilinear form $a : \tilde{ \mathbb Y} \times \tilde{\mathbb Y} \to \mathbb{R}$ by
\begin{align*}
a({\bf u}, \boldsymbol{\phi})
 &=  \int_\Omega {\bf v} \cdot \boldsymbol{\varphi} \,dx + m_0\, \boldsymbol{\ell} \cdot {\bf g} + {\mathbb J}_0\, \boldsymbol{\omega }\cdot \boldsymbol{\theta}+2\mu \int_\Omega {\mathbb D}({\bf v}) : {\mathbb D}(\boldsymbol{\varphi}) \,dx
 \\
&\quad + \alpha \int_{\partial \Omega}[{\bf v} - \boldsymbol{\ell} - \boldsymbol{\omega} \times x]_\tau
 \cdot [\boldsymbol{\varphi} - {\bf g} - \boldsymbol{\theta} \times x]_\tau\,dS_x,
\end{align*}
where ${\bf u} \simeq ({\bf v}, \boldsymbol{\ell},  \boldsymbol{\omega})$ and $\boldsymbol{\phi }\simeq (\boldsymbol{\varphi}, {\bf g}, \boldsymbol{\theta})$.

This bilinear form is bounded on the Hilbert space $\tilde{\mathbb Y}$,
\[
|a({\bf u}, \boldsymbol{\phi})| \leq c \|{\bf u}\|_{\tilde{ \mathbb Y}} \|\boldsymbol{\phi}\|_{\tilde{ \mathbb Y}}.
\]
Moreover, it is coercive:
\[
a({\bf u}, {\bf u})\geq c \|{\bf u}\|_{\tilde{ \mathbb Y }}^2.
\]


Let \({\bf  f } \in \tilde{\mathbb X}\) with ${\bf f}\simeq ({\bf f}_0,{\bf f}_1,{\bf f}_2)$. Then, for any \(\boldsymbol{\phi} \simeq (\boldsymbol{\varphi}, {\bf g}, \boldsymbol{\theta})\),
\[
\langle{\bf  f}, \boldsymbol{\phi} \rangle = \int_\Omega {\bf f}_0(x) \boldsymbol{\varphi}(x) \,dx + m_0 {\bf f}_1 \cdot {\bf g} + \mathbb{J}_0 {\bf f}_2 \cdot \boldsymbol{\theta} 
\leq c \|{\bf f}\|_{\tilde{\mathbb X}} \|\boldsymbol{\phi}\|_{\tilde{\mathbb X}}\leq c \|{\bf f}\|_{\tilde{\mathbb X}} \|\boldsymbol{\phi}\|_{\tilde{\mathbb Y}}.
\]
Hence, \({\bf f} \in \tilde{\mathbb{Y}}'\), where $\tilde{\mathbb{Y}}'$ is the dual space of \(\tilde{\mathbb{Y}}\).

By the Lax–Milgram theorem, there exists a unique \( {\bf u} \in \tilde{\mathbb Y} \) such that
\begin{equation}
\label{lax}
a({\bf u}, \boldsymbol{\phi}) = \langle{\bf  f},\boldsymbol{ \phi} \rangle, \quad \forall \boldsymbol{\phi} \in \tilde{\mathbb{Y}}.
\end{equation}
This variational solution ${\bf u}$ is precisely the weak solution to the operator equation $(I + \mathbb{A}_\alpha){\bf u} = {\bf f}$. To conclude that ${\bf u}$ coincides with the strong definition of $\mathbb{A}_\alpha := \mathbb{P}\mathcal{A}$ introduced in Section \ref{section.fluidstructure}, it suffices to show that ${\bf u}$ belongs to the domain $\mathcal{D}(\mathbb{A}_\alpha)$. 

In what follows, we demonstrate that ${\bf u}$ possesses the required $W^{2,2}$-regularity and satisfies the prescribed boundary conditions, thereby ensuring that the identity $(I + \mathbb{A}_\alpha){\bf u} = {\bf f}$ holds in $\tilde{\mathbb{X}}$.

Since \( {\bf u} \in \tilde{\mathbb{Y}} \), we have:
 \( \mathrm{div}\, {\bf u} = 0 \) in \( {\mathbb R}^3 \),
 \( \mathbb{D}({\bf u}) = 0 \) in \( B \),
 \( ({\bf u} -\boldsymbol{\ell} - \boldsymbol{\omega} \times x) \cdot {\bf n}=0 \)
 on \( \partial \Omega \).
Moreover, the condition \( \mathbb{D}({\bf u}) = 0 \) in \( B \), together with the definitions
\[
\boldsymbol{\ell }= \frac{1}{m_0} \int_B {\bf u}\,dx, \qquad \boldsymbol{\omega }= \mathbb{J}_0^{-1} \int_B x \times {\bf u}\,dx,
\]
implies that
\[
{\bf u}(x) = \boldsymbol{\ell} + \boldsymbol{\omega} \times x \quad \text{in } B.
\]

Taking $\boldsymbol{\phi} = \boldsymbol{u}$ in the weak formulation \eqref{lax}, we obtain
\[
\|{\bf u}\|_{{\mathbb X}}^2 + 2\mu \int_\Omega |{\mathbb D}({\bf v})|^2\,dx + \alpha  \int_{\partial \Omega} |[{\bf v} -\boldsymbol{ \ell} -\boldsymbol{ \omega }\times x]_\tau|^2\,dS_x
\leq c \|{\bf f}\|_{\tilde{\mathbb X}} \|{\bf u}\|_{\tilde {\mathbb X}}.
\]
This implies the estimate
\[
\|{\bf u}\|_{{\mathbb X}}^2 +  \int_\Omega |{\mathbb D}({\bf v})|^2 dx +  \int_{\partial \Omega} |[{\bf v} -\boldsymbol{ \ell} -\boldsymbol{ \omega }\times x]_\tau|^2\,dS_x
\leq c \|{\bf f}\|_{\tilde{\mathbb X}}^2.
\]
Furthermore, by Lemma \ref{strong.reg1}, we have
\[
\|{\bf u}\|_{\tilde{\mathbb{Y}}}^2\approx \|{\bf u}\|_{\tilde{\mathbb X}}^2 +  \int_\Omega |{\mathbb D}({\bf v})|^2\,dx +  \int_{\partial \Omega} |[{\bf v} -\boldsymbol{ \ell} -\boldsymbol{ \omega }\times x]_\tau|^2\,dS_x.\]
Hence
\[
\|{\bf u}\|_{\tilde{\mathbb{Y}}}  \leq c \|{\bf f}\|_{\tilde{\mathbb X}}.
\]

Let  $\boldsymbol{\varphi} \in {\bf C}^\infty_{0,\sigma}(\Omega)$.
  Define
  $\boldsymbol{\phi }=\begin{cases} \boldsymbol{\varphi}\mbox{ in }\Omega,\\
                       0\mbox{ in }B.
\end{cases}$
Substituting this test function into the weak formulation \eqref{lax} yields the identity  
\[
(-\mu \Delta {\bf v} + {\bf v}- {\bf f}_0, \boldsymbol{\varphi}) = 0 .
\]

By a well-known result due to Temam~\cite{temam}, there exists a function \( \pi \in \left( C^\infty_{0}(\Omega) \right)' \) such that
\[
-\mu \Delta {\bf v} + {\bf v}- {\bf f}_0= -\nabla \pi \quad \text{in } \Omega.
\]
Integrating by parts, the weak formulation \eqref{lax} can be rewritten as
\begin{align}
\label{b1}
\begin{array}{l}
    2\mu\int_{\partial \Omega} \boldsymbol{\varphi} \cdot \mathbb{D}({\bf v}){\bf  n} \, dS_x
    - \int_{\partial \Omega} \pi \boldsymbol{\varphi} \cdot {\bf n} \, dS_x
    \\[3mm]
    + \alpha \int_{\partial \Omega} 
        [{\bf v} -\boldsymbol{ \ell} -\boldsymbol{ \omega }\times x]_\tau
        \cdot [\boldsymbol{\varphi} -{\bf  g }- \boldsymbol{\theta} \times x] _\tau\, dS_x\\[3mm]
    + m_0 \boldsymbol{\ell }\cdot {\bf g }+ \mathbb{J}_0 \boldsymbol{\omega } \cdot \boldsymbol{\theta}
    = m_0{\bf  f}_1 \cdot{\bf  g} + \mathbb{J}_0{\bf f}_2 \cdot \boldsymbol{\theta}
    \end{array}
\end{align}
for any 
\(
\boldsymbol{\phi} \in {\bf C}^\infty_{0,\sigma}({\mathbb R}^3) \cap \tilde{\mathbb{Y}}\mbox{ with }\boldsymbol{\phi}\simeq (\boldsymbol{\varphi},{\bf g},\boldsymbol{\theta}).
\)

In particular, choosing
\(
\boldsymbol{\phi} = -\frac{1}{2} \nabla \times \left( \zeta(x) \left[ x \times {\bf g} + \boldsymbol{\theta} |x|^2 \right] \right)
\)
with \(\zeta \in C^\infty_0(\mathbb{R}^3)\) such that \(\zeta = 1\) on \(\overline{B}\),  
equation \eqref{b1} reduces to the following identity:
\begin{align*}
   & g \cdot \left( 
        \int_{\partial \Omega}2\mu \mathbb{D}({\bf v}) {\bf n} \, dS_x
        + m_0 \boldsymbol{\ell} - m_0 {\bf f}_1 
        - \int_{\partial \Omega} \pi {\bf n}dS_x
    \right) \\
    &+ \theta \cdot \left( 
        2\mu\int_{\partial \Omega} x \times \mathbb{D}({\bf v}){\bf  n} \, dS_x
        - \mathbb{J}_0 \boldsymbol{\omega} - \mathbb{J}_0 {\bf f}_2 
        - \int_{\partial \Omega} x \times \pi  {\bf n} \, dS_x
    \right) = 0.
\end{align*}
From the arbitrariness of \( {\bf g}, \boldsymbol{\theta} \in \mathbb{R}^3 \), we conclude:
\begin{align}
\label{b2}
    \begin{array}{l}
    m_0 \boldsymbol{\ell} + \int_{\partial \Omega} \pi {\bf n} \, dS_x
    = - 2\mu\int_{\partial \Omega} \mathbb{D}({\bf v}) {\bf n} \, dS_x+ m_0 {\bf f}_1, \\[3mm]
    \mathbb{J}_0\boldsymbol{\omega} + \int_{\partial \Omega} x \times \pi {\bf n} \, dS_x
    = -2\mu \int_{\partial \Omega} x \times \mathbb{D}({\bf v}) {\bf n} \, dS_x+ \mathbb{J}_0 {\bf f}_2.
    \end{array}
\end{align}


Again, applying the identity \eqref{b2} to the expression \eqref{b1}, we obtain
\begin{align*}
    \int_{\partial \Omega} 
    [\boldsymbol{\varphi} - {\bf g }- \boldsymbol{\theta} \times x]_\tau
    \cdot 
    \left(  2\mu[ \mathbb{D}({\bf v}){\bf n}]_\tau + \alpha [{\bf v} -\boldsymbol{ \ell} -\boldsymbol{ \omega }\times x]_\tau \right)
    \, dS_x= 0
\end{align*}
for any \( \boldsymbol{\phi} \in {\bf C}^\infty_{0,\sigma}({\mathbb R}^3) \cap \tilde{\mathbb{Y}} \) with \( \boldsymbol{\phi} \simeq (\boldsymbol{\varphi}, {\bf g}, \boldsymbol{\theta}) \).
This leads to the conclusion:
\[
2\mu[  \mathbb{D}({\bf v}){\bf n} ]_\tau+ \alpha [{\bf v} -\boldsymbol{ \ell} -\boldsymbol{ \omega }\times x]_\tau ={\bf  0}\mbox{ on }\partial \Omega.
\]

For the smooth cut-off function  $\zeta $ with $\zeta=1$  near $\partial \Omega$
define $\tilde{\bf v}_B=-\frac{1}{2}\nabla \times \Big( \zeta(x) [x\times \boldsymbol{\ell}+\boldsymbol{\omega}|x|^2]\Big)$
 and $\tilde{\bf v}={\bf  v}-\tilde{\bf v}_B $.
Applying  Lemma \ref{prop.ogawa} to the  solution of the system
\begin{align*}
&\begin{array}{rl}
- \mu \Delta\tilde{\bf v} + \nabla \pi &= {\bf f}_0-{\bf v}+\mu \Delta \tilde{\bf u}_B, \\
\mathrm{div}\,  \tilde{\bf v} &= 0
\end{array} \qquad \text{in } \Omega, \\
&
\tilde{\bf v} \cdot {\bf n} = 0, \ 
2\mu [ {\mathbb D}(\tilde{\bf v}){\bf n}]_\tau+\alpha[ \tilde{\bf v} ]_\tau={\bf 0}
\qquad \text{on } \partial \Omega,\\
&\lim_{|x|\rightarrow \infty}\tilde{\bf v}(x)={\bf 0},
\end{align*}
we deduce that \( {\bf v} \in {\bf W}^{2,2}(\Omega) \), and
\[
\|{\bf v}\|_{W^{2,2}(\Omega)} \leq c \|{\bf f}\|_{\tilde{\mathbb X}} + c \|{\bf u}\|_{\tilde{\mathbb X}}.
\]
This implies that \( {\bf u} \in \mathcal{D}({\mathbb A}_\alpha) \), and furthermore,
\[
\|{\bf u}\|_{\mathcal{ D}({\mathbb A}_\alpha ) } \leq c \|(I + {\mathbb A}_\alpha){\bf u}\|_{L^2(\mathbb{R}^n)}.
\]
Therefore,
the operator \( I + {\mathbb A}_\alpha  :\mathcal{ D}({\mathbb A}_\alpha ) \to {\mathbb X} \) is surjective,  and  satisfies the estimate \eqref{surjective}.
\end{proof}

%

According to Lemma~\ref{prop.pre2}, the operator ${\mathbb A}_\alpha$ is symmetric and positive definite—that is, $-{\mathbb A}_\alpha$ is accretive.
Furthermore, by Theorem 5.19 in \cite{weidmann}, the operator
\[-{\mathbb A}_\alpha: \mathcal{D}({\mathbb A}_\alpha) \subset \tilde{\mathbb X} \rightarrow \tilde{\mathbb X}\mbox{
is self-adjoint},\]
 owing to its symmetry, positive definiteness, and the fact that
$\mathrm{Ran}(I + {\mathbb A}_\alpha) = \tilde{\mathbb X}.$

%

Now, suppose that 
\[
({\bf f}_0, {\bf  f}_1, {\bf f}_2) \in {\bf L}^2(Q_T) \times {\bf L}^2(0,T) \times {\bf L}^2(0,T), \quad 
({\bf v}_0, \boldsymbol{\ell}_0, \boldsymbol{\omega}_0) \in {\bf W}^{1,2}(\Omega) \times {\mathbb R}^3 \times {\mathbb R}^3
\]
satisfying the compatibility condition. 

Let ${\bf f}={\bf f}_0\chi_\Omega+({\bf f}_1+{\bf f}_2\times x)\chi_B$ and  ${\bf u}_0={\bf v}_0\chi_\Omega+(\boldsymbol{\ell}_0+\boldsymbol{\omega}_0\times x)\chi_B$ so that  
\[
{\bf f} \simeq ({\bf f}_0, {\bf f}_1, {\bf f}_2), \quad {\bf u}_0 \simeq ({\bf v}_0, \boldsymbol{\ell}_0, \boldsymbol{\omega}_0).
\]
Then, we have
\[
{\bf f} \in {\bf L}^2(\mathbb{R}^n\times (0,T)), \quad 
{\bf u}_0 \in \mathbb{B}^{1,2}_\alpha.
\]

Since $\mathbb{A}_\alpha$ is self-adjoint and accretive on the Hilbert space $\tilde{\mathbb{X}}^2$, the operator $-\mathbb{A}_\alpha$ generates a bounded analytic semigroup on $\tilde{\mathbb{X}}^2$. According to \cite[Part~II, Chapter~1, Theorem~3.1]{bensoussan}, the equation
\begin{align}
\label{t1}
\frac{d{\bf u}}{dt} + \mathbb{A}_\alpha {\bf u} = \mathbb{P}{\bf f}, \quad {\bf u}(0) ={\bf  u}_0
\end{align}
admits a unique solution 
${\bf u} \in L^2(0,T;\mathcal{D}(\mathbb{A}_\alpha)) \cap C([0,T]; \mathbb{B}_\alpha^{1,2}) \cap W^{1,2}(0,T; \tilde{\mathbb{X}}^2)$.
Recall from \eqref{besov} in Section~\ref{section.fluidstructure} that $\mathbb{B}_\alpha^{1,2}$ is the real interpolation space $\mathbb{B}^{2-\frac{2}{q}, q}_{\alpha}$ evaluated at $q=2$.
Moreover, this solution satisfies the estimate
\begin{align*}
\|{\bf u}\|_{L^2(0,T; \mathcal{D}(\mathbb{A}_\alpha))} 
&+ \|{\bf u}\|_{L^\infty(0,T; \mathbb{B}^{1,2}_\alpha)}
 + \|{\bf u}\|_{W^{1,2}(0,T; \tilde{\mathbb{X}}^2)}\\
 & \leq C\left( \|{\bf u}_0\|_{\mathbb{B}^{1,2}_\alpha} + \|\mathbb{P}{\bf f}\|_{L^2(0,T; \tilde{\mathbb{X}}^2)} \right).
\end{align*}

%
%
%
%

Furthermore, combining the previous computations with Lemma~\ref{prop.ogawa}, we obtain the estimate
\begin{align*}
\|{\bf v}\|_{L^2(0,T; W^{2,2}(\Omega))} + \|\nabla \pi\|_{L^2(0,T; L^2(\Omega))} 
\leq C\left( \|{\bf u}_0\|_{\mathbb{B}^{1,2}_\alpha} + \|\mathbb{P}{\bf f}\|_{L^2(0,T; \tilde{\mathbb{X}})} \right).
\end{align*}

Let \( {\bf u}\in \tilde {\mathbb X}  $ with ${\bf u}\simeq ({\bf v}, \boldsymbol{\ell},  \boldsymbol{\omega}) \) and \(\boldsymbol{\phi} \in \tilde {\mathbb X} \) with \( \boldsymbol{\phi} \simeq (\boldsymbol{\varphi}, {\bf g}, \boldsymbol{\theta}) \).
Taking the inner product of \eqref{t1} with \( \boldsymbol{\phi} \in \tilde{\mathbb X} \), we obtain
\begin{align}
\notag
\label{abc} 
0& = \langle \frac{d}{dt}{\bf u} + {\mathbb A} {\bf u} - {\mathbb P}{\bf f}, \boldsymbol{\phi}\rangle= \int_\Omega \left( \partial_t {\bf v} - \mu\Delta {\bf v} - {\bf f} \right) \cdot \boldsymbol{\varphi} \, dx + m_0 \dot{\boldsymbol{\ell}} \cdot {\bf g} + {\mathbb J}_0 \dot{\boldsymbol{\omega}} \cdot \boldsymbol{\boldsymbol{\theta}} dx\\
&+ {\bf g }\cdot \int_{\partial \Omega} 2\mu {\mathbb D}({\bf v}){\bf n} \, dS_x+ \boldsymbol{\theta} \cdot \int_{\partial \Omega} x \times {\mathbb D}({\bf v}) \, {\bf n} \, dS_x.
\end{align}

For \(\boldsymbol{\varphi} \in {\bf C}^\infty_{0,\sigma}(\Omega)\), we define the test function \(\boldsymbol{\phi}\) as its zero extension to \(\mathbb{R}^3\). Then, we obtain
\[
\int_\Omega \left( \partial_t {\bf v} - \mu\Delta {\bf v} - {\bf f}_0 \right) \cdot \boldsymbol{\varphi} \, dx = 0.
\]

According to Proposition 1.2 in \cite{temam}, there exists
 \( \pi \in L^2(0,T; \hat{W}^{1,2}(\Omega)) \) such that
\[
\partial_t {\bf v}(t) - \mu \Delta{\bf  v}(t) + \nabla \pi(t) = {\bf f}_0(t) \quad \text{in } \Omega.
\]

This yields
\[
\int_\Omega \left(\partial_t {\bf v} - \mu \Delta {\bf v} -{\bf  f}_0\right) \cdot \boldsymbol{\varphi} \, dx = -\int_{\partial \Omega} \pi {\bf n} \cdot \boldsymbol{\varphi} \, dS_x. 
\] 
Thus, \eqref{abc} leads to the identity
\begin{align*}
m_0 \dot{\boldsymbol{\ell}}{\cdot} {\bf g }+ \mathbb{J}_0 \dot{ \boldsymbol{\omega}} {\cdot} \boldsymbol{\theta}
 +{\bf  g} {\cdot} \int_{\partial \Omega} \left(2\mu \mathbb{D}({\bf v}) {-} \pi \mathbb{I}_3\right) {\bf n} \, dS_x
  +\boldsymbol{ \theta} {\cdot} \int_{\partial \Omega} x \times
  (2\mu \mathbb{D}({\bf v}) {-}\pi \mathbb{I}_3) {\bf n} \, dS_x= 0.
\end{align*}
From the above identity we deduce that
\begin{align*}
m_0 \dot{\boldsymbol{\ell}}& = -\int_{\partial \Omega} (2\mu \mathbb{D}(\boldsymbol{\bf v}) - \pi \mathbb{I}_3) {\bf n} \, dS_x,\\
\mathbb{J}_0 \dot{\boldsymbol{\omega}}& = -\int_{\partial \Omega} x \times
 (2\mu \mathbb{D}({\bf v}) - \pi \mathbb{I}_3) {\bf n} \, dS_x.
\end{align*}

This  completes the proof of Theorem \ref{maximal.L2}.

\section{Proof of Theorem \ref{maximal.thm}: Maximal regularity on a finite time interval}
\label{section.maximal.thm}

This section is devoted to the first part of Theorem \ref{maximal.thm}. Specifically, we establish the maximal $L^q$-regularity of the solution on a finite time interval $(0,T)$ for the general case $1 < q < \infty$. 

{
{

The following lemma provides an \textit{a priori} estimate essential for the proof of the maximal $L^q$ regularity.

\begin{lemm}[A priori estimate]
\label{prop.theorem.strong.Lp}
Let the data 
$$({\bf f}_0, {\bf f}_1, {\bf f}_2) \in {\bf L}^q(Q_T) \times {\bf L}^q(0,T) \times {\bf L}^q(0,T),$$
$$({\bf v}_0, \boldsymbol{\ell}_0, \boldsymbol{\omega}_0) \in {\bf W}^{2 - \frac{2}{q}, q}(\Omega) \times \mathbb{R}^3 \times \mathbb{R}^3, \quad 1 < q < \infty,$$
be sufficiently smooth and satisfy the compatibility conditions \textup{(C1)}--\textup{(C3)}. 

Let \( ({\bf v}, \pi, \boldsymbol{\ell},  \boldsymbol{\omega}) \) be a correspondingly smooth (or sufficiently regular) solution to the system \eqref{e01}. Then, there exists a constant $c > 0$, independent of the solution, such that the following \textit{a priori} estimate holds:
\begin{align}
\notag
&\int^T_0\big(  \|\partial_t{\bf v}(t) \|_{L^q(\Omega)}^q+\|\nabla^2{\bf v}(t)\|_{L^q(\Omega)}^q
+\|\nabla \pi\|_{L^q(\Omega)}^q\\
&\hspace{20mm}+|\boldsymbol{\ell}| ^q+|\boldsymbol{\omega}|^q+|\dot{\boldsymbol{\ell}}|^q+|\dot{\boldsymbol{\omega}}|^q \big) dt
\notag\\
\label{w8}
&\notag \leq c
\Big(   \int^T_0\|{\bf v}\|_{L^q(\Omega_{2R})}^q+\|{\bf f}_0(t)\|_{L^q(\Omega)}^q+|{\bf f}_1(t)|^q+|{\bf f}_2(t)|^q dt\Big)\\
&\hspace{20mm}
 +c\|{\bf v}_0\|_{W^{2-\frac{2}{q},q}(\Omega)}^q+c|\boldsymbol{\ell}_0|^q+c|\boldsymbol{\omega}_0|^q.
\end{align}
Here, $c$ is independent of $T$.
\end{lemm}
}
}

\begin{proof}

Let $\Omega^c \subset B_{\frac{R}{4}}$. 
Take a smooth cut-off function $\xi \in C^\infty_0({\mathbb R}^3)$ such that $\xi = 1$ on $B_{\frac{R}{2}}$ and $\xi = 0$ outside $B_R$.

Denoting by ${\bf n}$ the unit outward normal vector to $\Omega$ (pointing into $\Omega^c$), we have $\operatorname{div}(\xi{\bf v}) = {\bf v}\cdot\nabla\xi$. By the divergence theorem on $\Omega_R$,
\begin{align*}
\int_{B_R \setminus B_{R/2}} {\bf v} \cdot \nabla \xi \, dx
&= \int_{\Omega_R} \operatorname{div}(\xi {\bf v}) \, dx
 = \int_{\partial\Omega} {\bf v} \cdot {\bf n} \, dS_x \\
&= \int_{\partial\Omega} (\boldsymbol{\ell} + \boldsymbol{\omega} \times x) \cdot {\bf n} \, dS_x
 = -\int_{\Omega^c} \operatorname{div}(\boldsymbol{\ell} + \boldsymbol{\omega} \times x) \, dx = 0,
\end{align*}
since the outward normal to $\Omega^c$ is $-{\bf n}$ and $\operatorname{div}(\boldsymbol{\ell} + \boldsymbol{\omega}\times x) = 0$.

According to Lemma \ref{lem.bogovskii_cutoff}, there exists an auxiliary vector field ${\bf w}_1$ compactly supported in $B_R \setminus B_{R/2}$ such that
\[
\operatorname{div} {\bf w}_1 = {\bf v} \cdot \nabla \xi \quad \text{in } {\mathbb R}^3,
\]
and for any $q \in (1, \infty)$, it satisfies the following estimates:
\begin{align}
\label{w21}
\|\nabla {\bf w}_1\|_{L^q(\Omega)} &\leq c\|{\bf v}\|_{L^q(\Omega_R)}, \\
\label{w31}\|\Delta {\bf w}_1\|_{L^q(\Omega)} &\leq c\|{\bf v}\|_{L^q(\Omega_R)} + c\|\nabla{\bf v}\|_{L^q(\Omega_R)}, \\
\label{w41}\|\partial_t {\bf w}_1\|_{L^q(\Omega)} &\leq c\|\partial_t {\bf v}\|_{W^{-1,q}(\Omega_R)}.
\end{align}
To estimate the time derivative in the negative Sobolev space, we test the equation against a function \( \boldsymbol{\phi} \in {\bf W}^{1, q'}_0(\Omega_R) \), where \( 1/q + 1/q' = 1 \). Using integration by parts and H\"older's inequality, we readily obtain
\begin{equation} \label{eq:time_deriv_estimate}
\| \partial_t {\bf v} \|_{W^{-1,q}(\Omega_R)} \le C \left( \|\nabla {\bf v}\|_{L^q(\Omega_R)} + \|\pi\|_{L^q(\Omega_R)} + \|{\bf f}_0\|_{L^q(\Omega_R)} \right),
\end{equation}
where the constant \( C > 0 \) depends only on \( q \) and \( \Omega_R \).

Now, we set the localized velocity and pressure fields as
\[
{\bf v}'' = {\bf v}(1 - \xi) + {\bf w}_1, \qquad \pi'' = \pi(1 - \xi).
\]
Then, $({\bf v}'', \pi'')$ satisfies the whole-space Stokes system:
\begin{align*}
\begin{aligned}
&\partial_t {\bf v}'' - \mu \Delta {\bf v}'' + \nabla \pi'' = {\bf f}_0'' \quad \text{in } {\mathbb R}^3 \times (0, T), \\
&\operatorname{div} {\bf v}'' = 0 \quad \text{in } {\mathbb R}^3 \times (0, T), \\
&{\bf v}'' \to {\bf 0} \quad \text{as } |x| \to \infty, \ t > 0, \\
&{\bf v}''|_{t=0} = (1 - \xi) {\bf v}_0 + {\bf w}_1(0) := {\bf v}_0'' \quad \text{in } {\mathbb R}^3,
\end{aligned}
\end{align*}
where the effective right-hand side source term is given by
\begin{align*}
{\bf f}_0'' &= (1 - \xi) {\bf f}_0 + \mu \left( 2 \nabla{\bf v} \cdot \nabla \xi + {\bf v} \Delta \xi \right) - \pi \nabla \xi + \partial_t{\bf w}_1 - \mu \Delta {\bf w}_1.
\end{align*}

Based on the well-known estimate for the Stokes operator defined in the whole space, we obtain the following estimate:
\begin{align}\label{w1}
\notag
\int_0^T \| \partial_t {\bf v}'' \|_{L^q({\mathbb R}^3)}^q
 &+ \|\nabla^2{\bf v}''\|_{L^q({\mathbb R}^3)}^q + \| \nabla \pi''\|_{L^q({\mathbb R}^3)}^q dt \\
&\leq c \int_0^T \| {\bf f}_0''\|_{L^q({\mathbb R}^3)}^q dt + c \| {\bf v}''_0 \|_{W^{2 - \frac{2}{q},q}({\mathbb R}^3)}^q.
\end{align}

Note that by \eqref{w41} and \eqref{eq:time_deriv_estimate}, the time derivative term in \( {\bf f}_0'' \) is bounded entirely by spatial norms: \( \| \partial_t {\bf w}_1 \|_{L^q(\Omega)} \le c ( \|\nabla {\bf v}\|_{L^q(\Omega_R)} + \|\pi\|_{L^q(\Omega_R)} + \|{\bf f}_0\|_{L^q(\Omega_R)} ) \). Collecting the estimates \eqref{w21}-\eqref{w1}, we arrive at
\begin{align}
\label{w100}
&\int^T_0 \Big( \|\partial_t {\bf v}(t)\|_{L^q(B_R^c)}^q
 + \|\nabla^2{\bf  v}(t)\|_{L^q(B_R^c)}^q + \|\nabla \pi\|_{L^q(B_R^c)}^q \Big) dt \\
\notag&\leq c\Big( \int^T_0  \|{\bf v}\|_{W^{1,q}(\Omega_{R})}^q
 + \|\pi\|_{L^{q}(\Omega_{R})}^q +\|{\bf f}_0(t)\|_{L^q(\Omega)}^q dt
  + \|{\bf v}_0\|_{W^{2-\frac{2}{q},q}(\Omega)}^q \Big).
\end{align}

Now, we take a cut-off function $\zeta \in C^\infty_0({\mathbb R}^3)$ such that $\zeta = 1$ on $B_R$ and $\zeta = 0$ outside $B_{2R}$, so that $\operatorname{supp}(\nabla\zeta) \subset B_{2R} \setminus \overline{B_R}$. By the exact same divergence theorem argument as above, $\int_{B_{2R}\setminus B_R} {\bf v}\cdot\nabla\zeta\,dx = 0$. Applying Lemma~\ref{lem.bogovskii_cutoff} and repeating the $W^{-1,q}$ estimate \eqref{eq:time_deriv_estimate} on $\Omega_{2R}$, we obtain an auxiliary vector field ${\bf w}_2$ compactly supported in $B_{2R} \setminus B_R$ such that $\operatorname{div}{\bf w}_2 = {\bf v}\cdot\nabla\zeta$ in $\mathbb{R}^3$ and
\begin{equation}
\label{w22}
\|\nabla {\bf w}_2\|_{L^q(\Omega)} + \|\Delta {\bf w}_2\|_{L^q(\Omega)} + \|\partial_t {\bf w}_2\|_{L^q(\Omega)}
\leq c\Big(\|{\bf v}\|_{W^{1,q}(\Omega_{2R})} + \|\pi\|_{L^q(\Omega_{2R})} + \|{\bf f}_0\|_{L^q(\Omega_{2R})}\Big).
\end{equation}

Set \( {\bf v}' ={\bf  v}\zeta - {\bf w}_2 \) and \( \pi' = \pi \zeta \).
Thus, the localized velocity field ${\bf v}'$ strictly satisfies the homogeneous Navier-slip boundary condition on $\partial \Omega$ without any residual stress terms. Then, \( ({\bf v}', \pi') \) satisfies the linearized fluid interaction problem in $\Omega_{2R}$:
\begin{align*}
\begin{cases}
\begin{aligned}
&\partial_t {\bf v}' - \mu \Delta {\bf v}' + \nabla \pi' = {\bf f}'_0\\
&\operatorname{div} {\bf v}' = 0 
\end{aligned}\qquad
\mbox{ in }\Omega_{2R},\\
{\bf v}' = 0\qquad \mbox{ on }\partial B_{2R}, \\
\begin{aligned}
&({\bf v}' -\boldsymbol{\ell} - \boldsymbol{\omega} \times x) \cdot {\bf n}=0\\
&2\mu[\mathbb{D}({\bf v}'){\bf n}]_\tau+\alpha [{\bf v}' - \boldsymbol{\ell} - \boldsymbol{\omega} \times x]_\tau={\bf 0}
\end{aligned}\qquad \text{ on } \partial \Omega, \\
\begin{aligned}
&m_0\dot{\boldsymbol{\ell}} = -\int_{\partial \Omega} \mathbb{S}({\bf v}', \pi'){\bf n} dS_x+ m_0 {\bf f}_1\\
&\mathbb{J}_0\dot{\boldsymbol{\omega} }= -\int_{\partial \Omega} x \times \mathbb{S}({\bf v}', \pi'){\bf n} dS_x
+ \mathbb{J}_0{\bf  f}_2
\end{aligned}\qquad\mbox{ for }t>0,\\
{\bf v}'|_{t=0} =  {\bf v}_0', \quad \boldsymbol{\ell}|_{t=0} = \boldsymbol{\ell}_0, \quad \boldsymbol{\omega}|_{t=0} = \boldsymbol{\omega}_0\mbox{ on }\Omega_{2R},
\end{cases}
\end{align*}

where ${\bf v}_0':=\zeta {\bf v}_0 - {\bf w}_2(0)$, and
\[
{\bf f}'_0 := \zeta{\bf  f}_0 - \mu \left( 2\nabla {\bf v} \cdot \nabla \zeta + {\bf v}\Delta \zeta \right) + \pi \nabla \zeta - \left( \partial_t {\bf w}_2 - \mu \Delta {\bf w}_2 \right).
\]

According to the result in \cite[Theorem 4.9]{necasova.amrouche}, \( ({\bf v}', \pi') \) satisfies  the following  \( L^q_t(L^q_x) \) 
regularity estimates:
\begin{align}
\label{w4}
\notag & \int_0^T 
\|\partial_t {\bf v}'\|_{L^q(\Omega_{2R})}^q + \|\nabla^2 {\bf v}'\|_{L^q(\Omega_{2R})}^q + \|\nabla \pi'\|_{L^q(\Omega_{2R})}^q \\
&\qquad\qquad\qquad+ |\boldsymbol{\ell}| ^q+ |\dot{\boldsymbol{\ell}}|^q + |\boldsymbol{\omega}|^q + |\dot{\boldsymbol{\omega}}|^q \, dt \\
\notag& \leq c \int_0^T \|{\bf f}'_0\|_{L^q(\Omega_{2R})}^q + |{\bf f}_1(t)|^q + |{\bf f}_2(t)|^q \, dt + 
\|{\bf v}_0\|_{W^{2-\frac{2}{q},q}(\Omega_{2R})}^q + |\boldsymbol{\ell}_0|^q + |\boldsymbol{\omega}_0|^q.
\end{align}

In view of the estimates \eqref{w22}-\eqref{w4}, we deduce the following bound:
\begin{align}
\label{w200}
\notag &\int^T_0 \left( 
\|\partial_t {\bf v}\|_{L^q(\Om_{R})}^q + \|\nabla^2{\bf  v}\|_{L^q(\Omega_{R})}^q + \|\nabla \pi\|_{L^q(\Omega_{R})}^q + |\boldsymbol{\ell}| ^q +|\dot{\boldsymbol{\ell}}|^q+ |\boldsymbol{\omega}|^q+|\dot{\boldsymbol{\omega}}|^q \right) dt\\
\notag &\leq c \int^T_0 \left(  \|{\bf v}\|_{W^{1,q}(\Omega_{2R})}^q + \|\pi\|_{L^q(\Omega_{2R})}^q +\|{\bf f}_0\|_{L^q(\Omega)}^q + |{\bf f}_1(t)|^q + |{\bf f}_2(t)|^q \right) dt\\
&\qquad +c \|{\bf v}_0\|_{W^{2-\frac{2}{q},q}(\Omega)}^q +c |\boldsymbol{\ell}_0|^q +c |\boldsymbol{\omega}_0|^q.
\end{align}

Combining \eqref{w100} and \eqref{w200}, we obtain the estimate
\begin{align}
\label{eq.theore.strong.Lp}
\notag & \int_0^T \Big( \|\partial_t {\bf v}(t)\|_{L^q(\Omega)}^q + \|\nabla^2{\bf v}(t)\|_{L^q(\Omega)}^q + \|\nabla \pi(t)\|_{L^q(\Omega)}^q \\
\notag&\qquad\qquad\qquad
+ |\boldsymbol{\ell}(t)|^q + |\boldsymbol{\omega}(t)|^q + |\dot{\boldsymbol{\ell}}(t)|^q + |\dot{\boldsymbol{\omega}}(t)|^q \Big) dt \\
\notag &\leq c \int_0^T \Big( \|{\bf v}(t)\|_{W^{1,q}(\Omega_{2R})}^q + \|\pi(t)\|_{L^q(\Omega_{2R})}^q 
+ \|{\bf f}_0(t)\|_{L^q(\Omega)}^q + |{\bf f}_1(t)|^q + |{\bf f}_2(t)|^q \Big) dt \\
&\quad \qquad \qquad+ c \|{\bf v}_0\|_{W^{2 - \frac{2}{q}, q}(\Omega)}^q + c |\boldsymbol{\ell}_0|^q + c |\boldsymbol{\omega}_0|^q.
\end{align}

Our next goal is to estimate the local terms $\|{\bf v}\|_{W^{1,q}(\Omega_{2R})}$ and $\|\pi\|_{L^q(\Omega_{2R})}$ on the right-hand side. 
First, applying Lemma \ref{lem:pressure_accel_decomp}, the local pressure satisfies the following estimate: 
\begin{align}
\label{pi_local_t}
\|\pi(t)\|_{L^q(\Omega_{2R})} 
\leq c\big(\|\nabla {\bf v}(t)\|_{W^{s,q}(\Omega_{2R})} + \|{\bf f}_0(t)\|_{L^q(\Omega)} + |{\bf f}_1(t)| + |{\bf f}_2(t)| \big),
\end{align}
where $ \frac{1}{q} < s < 1$.

Furthermore, by the standard interpolation inequality on bounded domains, for any sufficiently small $\epsilon > 0$, we have:
\begin{equation}
\label{pi0_local}
\|{\bf v}(t)\|_{W^{1,q}(\Omega_{2R})} + \|\nabla {\bf v}(t)\|_{W^{s,q}(\Omega_{2R})} \leq \epsilon \|\nabla^2 {\bf v}(t)\|_{L^q(\Omega)} + C_\epsilon \|{\bf v}(t)\|_{L^q(\Omega_{2R})}.
\end{equation}

Substituting these local estimates into \eqref{eq.theore.strong.Lp} and choosing $\epsilon > 0$ sufficiently small, we absorb the higher-order terms to obtain:
\begin{align*}
&\int^T_0 \Big( \|\partial_t{\bf v}(t)\|_{L^q(\Omega)}^q + \|\nabla^2{\bf v}(t)\|_{L^q(\Omega)}^q + \|\nabla \pi(t)\|_{L^q(\Omega)}^q \\
&\qquad\quad + |\boldsymbol{\ell}(t)|^q + |\boldsymbol{\omega}(t)|^q + |\dot{\boldsymbol{\ell}}(t)|^q + |\dot{\boldsymbol{\omega}}(t)|^q \Big) \, dt \\
&\leq c \int^T_0 \Big( \|{\bf v}(t)\|_{L^q(\Omega_{2R})}^q + \|{\bf f}_0(t)\|_{L^q(\Omega)}^q + |{\bf f}_1(t)|^q + |{\bf f}_2(t)|^q \Big) \, dt \\
&\quad + c\|{\bf v}_0\|_{W^{2-\frac{2}{q},q}(\Omega)}^q + c|\boldsymbol{\ell}_0|^q + c|\boldsymbol{\omega}_0|^q.
\end{align*}
This completes the proof of Lemma \ref{prop.theorem.strong.Lp}.

\end{proof}

We now begin the proof of Theorem~\ref{maximal.thm}. By a standard density argument, we may assume without loss of generality that the initial data satisfy ${\bf v}_0 \in {\bf C}^\infty_{0,\sigma}(\bar{\Omega})$ together with the compatibility conditions \eqref{n66}, and that the external forces $({\bf f}_0, {\bf f}_1, {\bf f}_2)$ are smooth and compactly supported.

By utilizing Theorem~\ref{maximal.L2} along with a bootstrap argument based on the higher-order estimates (such as Lemma~\ref{prop.ogawa}), there exists a unique smooth solution satisfying
\[
({\bf v}, \boldsymbol{\ell}, \boldsymbol{\omega}) \in {\bf C}^\infty(\bar{\Omega} \times (0,T)) \times {\bf C}^\infty(0,T) \times {\bf C}^\infty(0,T),
\qquad \pi \in C^\infty(\bar{\Omega} \times (0,T)).
\]
Let ${\bf u} \simeq ({\bf v}, \boldsymbol{\ell}, \boldsymbol{\omega})$.

%

To derive the regularity estimates in the \(L^q\) setting, we begin with the estimate \eqref{w8} from Lemma~\ref{prop.theorem.strong.Lp}.

Since \( {\bf v}(t) = {\bf v}_0 + \int_0^t \partial_s {\bf v}(s) \, ds \), applying H\"older's inequality with respect to the time variable yields
\begin{align*}
\|{\bf v}(t)\|_{L^q(\Omega)} 
&\leq \|{\bf v}_0\|_{L^q(\Omega)} + \int_0^t \|\partial_s {\bf v}(s)\|_{L^q(\Omega)} \, ds \\
&\leq \|{\bf v}_0\|_{L^q(\Omega)} + t^{1-\frac{1}{q}} \left( \int_0^t \|\partial_s {\bf v}(s)\|^q_{L^q(\Omega)} \, ds \right)^{\frac{1}{q}}.
\end{align*}

Taking the $q$-th power, we note that the time derivative integral $\int_0^t \|\partial_s {\bf v}(s)\|^q_{L^q(\Omega)} \, ds$ is explicitly bounded by the left-hand side of \eqref{w8}. By substituting the right-hand side of \eqref{w8} into this H\"older bound, we eliminate the time derivative and establish a closed Gronwall-type inequality for $\|{\bf v}(t)\|_{L^q(\Omega)}^q$:
\begin{align*}
\|{\bf v}(t)\|_{L^q(\Omega)}^q
&\leq c(t) \int_0^t \|{\bf v}(s)\|_{L^q(\Omega_{2R})}^q ds
 + c(t) \Big( \|{\bf v}_0\|_{W^{2 - \frac{2}{q},q}(\Omega)}^q
  + |\boldsymbol{\ell}_0|^q + |\boldsymbol{\omega}_0|^q \Big) \\
& \hspace{10mm} + c(t) \int_0^t \|{\bf f}_0(s)\|_{L^q(\Omega)}^q  + |{\bf f}_1(s)|^q + |{\bf f}_2(s)|^q ds,
\end{align*}
where the coefficient $c(t)$ is a strictly increasing function of time $t$.

By solving this Gronwall inequality, we deduce the explicit integral bound for ${\bf v}$, which is sufficient for our maximal $L^q$-regularity purpose:
\begin{align*}
\int_0^T \|{\bf v}(t)\|_{L^q(\Omega)}^q dt
 &\leq c(T) \Big( \int_0^T  \|{\bf f}_0(t)\|_{L^q(\Omega)}^q  + |{\bf f}_1(t)|^q + |{\bf f}_2(t)|^q dt \\
&\qquad+\|{\bf v}_0\|_{W^{2 - \frac{2}{q},q}(\Omega)}^q + |\boldsymbol{\ell}_0|^q + |\boldsymbol{\omega}_0|^q \Big).
\end{align*}

Finally, we substitute this bounded $L^q$-norm of ${\bf v}$ back into the right-hand side of \eqref{w8}. This is a standard \textit{a priori} estimate procedure, not a circular argument; it allows us to bound the remaining terms on the left-hand side of \eqref{w8} purely by the given data, yielding the desired maximal $L^q$-regularity:
\begin{align*}
\notag
&\int_0^T \|\partial_t {\bf v}(t)\|_{L^q(\Omega)}^q + \|\nabla^2{\bf v}(t)\|_{L^q(\Omega)}^q
 + \|\nabla \pi(t)\|_{L^q(\Omega)}^q +|\boldsymbol{\ell}(t)|^q+|\dot{\boldsymbol{\ell}}(t)|^q+|\boldsymbol{\omega}(t)|^q+|\dot{\boldsymbol{\omega}}(t)|^q dt \\
& \leq c(T)\Big(\int_0^T \|{\bf f}_0(t)\|_{L^q(\Omega)}^q + |{\bf f}_1(t)|^q + |{\bf f}_2(t)|^q dt + \|{\bf v}_0\|_{W^{2 - \frac{2}{q},q}(\Omega)}^q
 + |\boldsymbol{\ell}_0|^q + |\boldsymbol{\omega}_0|^q \Big).
\end{align*}

\section{\bf Proof of Theorem \ref{analyticity.thm}}
\label{section.analyticity.thm}  

Theorems~\ref{maximal.L2} and \ref{maximal.thm} imply that the fluid-structure operator 
possesses maximal $L^q$-regularity for all $1 < q < \infty$ on a fixed time interval. 
By the foundational results of Dore (see, for instance, \cite{dore1993} or related maximal regularity theory), 
this maximal regularity property guarantees that the operator is sectorial with some $\omega > 0$, and thus 
generates an analytic semigroup in $\tilde{\mathbb{X}}^q$ for all $1 < q < \infty$.

Our main goal in the subsequent analysis, however, is to establish that the semigroup is 
actually a bounded analytic semigroup for all $1 < q < \infty$ by proving sectoriality with $\omega = 0$. 

Denote the sector \(\Sigma_\theta\) by 
\[
\Sigma_\theta = \{\lambda \in \mathbb{C} \setminus \{0\} : |\arg(\lambda)| < \theta\}, \quad \text{for } \theta > 0.
\]

\begin{defin} 
Let $X$ be a complex Banach space, and let $\mathcal{L}(X)$ denote the space of bounded linear operators on $X$. Consider a densely defined, closed, linear operator $A : \mathcal{D}(A) \subset X \rightarrow X$.

We say that $A$ is a sectorial operator at $\omega$ if there exist constants $\omega \in \mathbb{R}$, $\theta \in \left( \frac{\pi}{2}, \pi \right)$, and $M > 0$ such that:
\begin{align*} 
(i) \quad &\rho(A) \supset \omega + \Sigma_\theta, \\
(ii) \quad &\|(\lambda I - A)^{-1}\|_{\mathcal{L}(X)} \leq \frac{M}{|\lambda - \omega|}, \quad \lambda \in \omega + \Sigma_\theta. 
\end{align*}
When $\omega = 0$, the operator is simply referred to as a sectorial operator. 
A semigroup generated by a sectorial operator $A$ is denoted by $e^{At}$. 
\end{defin}

\subsection{Sectoriality in  $\tilde{\mathbb X}^2$ }


According to Theorem 13.31 in \cite{rudin}, the self-adjointness of the operator $-{\mathbb A}_{\alpha,2}$ on $\tilde{\mathbb X}^2$ implies that its spectrum is purely real. Furthermore, its accretivity (or non-negativity) ensures that the spectrum is contained in the nonpositive real axis:
\begin{equation} \label{coro.L2bound}
\sigma(-{\mathbb A}_{\alpha,2}) \subset (-\infty, 0], \quad \text{or equivalently,} \quad \mathbb{C} \setminus (-\infty, 0] \subset \rho(-{\mathbb A}_{\alpha,2}).
\end{equation}
In particular, it is straightforward to verify the following resolvent estimate: \begin{equation} \label{L2bound2} \| \lambda(\lambda I + \mathbb{A}_{\alpha,2})^{-1} \|_{\mathcal{L}(\tilde{\mathbb X}^2)} \leq 1 \quad \text{for all } \lambda \in \mathbb{C} \setminus {0} \text{ with } \operatorname{Re}(\lambda) \geq 0, \end{equation}
where $\mathcal{L}(\tilde{\mathbb{X}}^2)$ denotes the space of bounded linear operators on $\tilde{\mathbb{X}}^2$.
%

According to Proposition 2.1.11 in Lunardi \cite[p. 43]{lunardi}, the spectral properties in \eqref{coro.L2bound} and the resolvent estimate in \eqref{L2bound2} together imply that the operator $-{\mathbb A}_{\alpha,2}$ is sectorial with any angle $\theta \in \left( \frac{\pi}{2}, \pi \right)$.

\subsection{Sectoriality in $\tilde{\mathbb X}^q$  for $1<q<\frac{3}{2}$}

Based on the approach in \cite{Takahashi2003, takahashi}, the abstract resolvent equation 
\[
(\lambda I + \mathbb{A}_{\alpha,q}){\bf U} = \mathbb{P}_{q}{\bf F}
\]
(where ${\bf U} = {\bf V}\chi_\Omega + (\boldsymbol{L} +\boldsymbol{\varOmega}\times x)\chi_B$ and ${\bf F} = {\bf F}_0\chi_\Omega + ({\bf F}_1 + {\bf F}_2\times x)\chi_B$) 
is equivalently rewritten in our Navier-slip setting as the following system for some pressure $\Pi$:
\begin{align}
\label{Resolvent1}
\begin{cases}
\begin{aligned}
 \lambda{\bf V}-\mu\Delta {\bf V}+\nabla\Pi&={\bf F}_0, \\
 \operatorname{div}{\bf V}&=0
\end{aligned} 
& \text{in } \Omega, \\[6pt]
\begin{aligned}
 m_0\lambda\boldsymbol{L}&=-\int_{\partial \Omega}{\mathbb S}({\bf V},\Pi){\bf n} \, dS_x + {\bf F}_1, \\
 {\mathbb J}_0\lambda\boldsymbol{\Omega}&=-\int_{\partial \Omega} x \times {\mathbb S}({\bf V},\Pi){\bf n} \, dS_x + {\bf F}_2
\end{aligned} 
& \\[6pt]
\begin{aligned}
 ({\bf V}-\boldsymbol{L}-\boldsymbol{\varOmega}\times x)\cdot{\bf n}&=0, \\
 2\mu [ {\mathbb D}({\bf V}){\bf n}]_\tau + \alpha[{\bf V}-\boldsymbol{L}-\boldsymbol{\varOmega}\times x]_\tau&={\bf 0}
\end{aligned} 
& \text{on } \partial \Omega, \\[6pt]
 \lim_{|x|\to\infty}{\bf V} = 0. &
\end{cases}
\end{align}

The following uniqueness result will be useful for our proof.

\begin{lemm}
\label{lemma.uniqueness}
Assume that $\mathrm{Re}(\lambda)\geq 0$, $\lambda\neq 0$.
Let $({\bf v}, \pi, \boldsymbol{\ell}, \boldsymbol{\omega})$ be a solution to the homogeneous system \eqref{Resolvent1} (i.e., with ${\bf F}_0 = {\bf 0}$, ${\bf F}_1 = {\bf 0}$, and ${\bf F}_2 = {\bf 0}$).
If ${\bf v} \in {\bf L}^s(\Omega)$ for some $1 \leq s <\infty$, and $\lambda {\bf v}, \nabla^2 {\bf v}, \nabla \pi \in {\bf L}^q(\Omega)$ for some $1 \leq q \leq \infty$, then 
\[
({\bf v}, \boldsymbol{\ell}, \boldsymbol{\omega}) \equiv ({\bf 0}, {\bf 0}, {\bf 0}).
\]
\end{lemm}

\begin{proof}
See Appendix~\ref{appendix.uniqueness}.

\end{proof}

As stated at the beginning of this section, the fluid-structure operator $\mathbb{A}_{\alpha,q}$ is already known to be sectorial with some shift $\omega_0 > 0$ for all $1 < q < \infty$. According to Proposition 2.1.11 in Lunardi \cite[p. 43]{lunardi}, to remove this shift and establish that $-{\mathbb{A}}_{\alpha,q}$ genuinely generates a bounded analytic semigroup, it suffices to show that the resolvent problem satisfies a uniform a priori estimate on a punctured neighborhood of the origin in the closed right half-plane:
\[
\lambda \in \{ z \in \mathbb{C} \mid \operatorname{Re} z \ge 0, 0 < |z| \le \gamma \}.
\]
Once this low-frequency bound is established, it naturally combines with the high-frequency bound to guarantee the existence of an angle $\theta \in \left( \frac{\pi}{2}, \pi \right)$ and a constant $M>0$ such that
\begin{equation} \label{theo.sectorial_estimate}
    \left\| (\lambda I + {\mathbb{A}}_{\alpha,q})^{-1} \right\|_{\mathcal{L}(\tilde{\mathbb{X}}^q)} \leq \frac{M}{|\lambda|} \quad \text{for all } \lambda \in \Sigma_\theta,
\end{equation}
where $\Sigma_{\theta} = \{ \lambda \in \mathbb{C} : \lambda \neq 0, |\arg \lambda| < \theta \}$. 
Consequently, the remainder of this proof is solely devoted to establishing this exact uniform estimate for $\lambda$ near the origin.

Let $\lambda \in \mathbb{C} \setminus \{0\}$ with $0 < |\lambda| \leq \gamma$ and $\operatorname{Re} \lambda \geq 0$. By a standard density argument, we may assume that ${\bf F}_0 \in {\bf C}^\infty_0(\overline{\Omega})$ and ${\bf F}_1, {\bf F}_2 \in \mathbb{C}^3$. Then, there exists a solution $({\bf V}, \Pi, \boldsymbol{L}, \boldsymbol{\varOmega})$ with ${\bf V} \in {\bf C}^\infty(\Omega)$, $\Pi \in C^\infty(\Omega)$, and $\boldsymbol{L}, \boldsymbol{\varOmega} \in \mathbb{C}^3$ satisfying the resolvent problem \eqref{Resolvent1} corresponding to the data $({\bf F}_0, {\bf F}_1, {\bf F}_2)$.

\begin{lemm}
\label{lemma.resolvent}
Let $\lambda \in \mathbb{C}$ satisfy $\operatorname{Re} \lambda \geq 0$ and $0 < |\lambda| \leq \gamma$. Under the same assumptions on the data as above, the solution $({\bf V}, \Pi, \boldsymbol{L}, \boldsymbol{\varOmega})$ satisfies the following uniform a priori estimate:
\begin{equation}
\label{Resolvent2}
\begin{aligned}
    & |\lambda| \Big( \|{\bf V}\|_{L^q(\Omega)} + |\boldsymbol{L}| + |\boldsymbol{\varOmega}| \Big) + \|\nabla^2{\bf V}\|_{L^q(\Omega)} + \|\nabla \Pi\|_{L^q(\Omega)} + |\boldsymbol{L}| + |\boldsymbol{\varOmega}| \\
    &\leq C \Big( \|{\bf V}\|_{L^q(\Omega_{2R})} + \|{\bf F}_0\|_{L^q(\Omega)} + |{\bf F}_1| + |{\bf F}_2| \Big),
\end{aligned}
\end{equation}
where the constant $C > 0$ is independent of $\lambda$.
\end{lemm}

\begin{proof}
To rigorously utilize Lemma \ref{prop.theorem.strong.Lp} with zero initial conditions, we define a test function 
\[
    {\bf v}(t) = \phi(t) e^{\lambda t} {\bf V}, \quad \pi(t) = \phi(t) e^{\lambda t} \Pi, \quad \boldsymbol{\ell}(t) = \phi(t) e^{\lambda t} \boldsymbol{L}, \quad \boldsymbol{\omega}(t) = \phi(t) e^{\lambda t} \boldsymbol{\Omega}
\]
for $t \geq 0$, where $\phi \in C^\infty([0, \infty))$ is a smooth, fixed cut-off function such that $\phi(t) = 0$ for $t \in [0, 1]$ and $\phi(t) = 1$ for $t \geq 2$. Consequently, we have ${\bf v}(0) = {\bf 0}, \boldsymbol{\ell}(0) = {\bf 0}, \text{ and } \boldsymbol{\omega}(0) = {\bf 0}$.

Direct calculation yields
\begin{align*}
&\begin{cases}
\begin{aligned}
 \frac{\partial {\bf v}}{\partial t}-\mu\Delta {\bf v}+\nabla{{\pi}}&=\phi(t) e^{\lambda t} {\bf F}_0+\phi'(t)e^{\lambda t} {\bf V}\\
\qquad\operatorname{div}{\bf  v}&=0
\end{aligned}\qquad
&\mbox{ in } \Omega\mbox{ for }t>0,\\
\begin{aligned}
&m_0\dot{ \boldsymbol{\ell}} (t)=-\int_{\partial \Omega}{\mathbb S}({\bf v},\pi){\bf n} \, { dS_x}+m_0\phi(t) e^{\lambda t} {\bf F}_1+\phi'(t)e^{\lambda t}\boldsymbol{L}\\
&{\mathbb J}_0\dot{ \boldsymbol{\omega}}(t)=-\int_{\partial \Omega}x\times
 {\mathbb S}({\bf v},\pi){\bf n} \, { dS_x}+{\mathbb J}_0\phi(t) e^{\lambda t} {\bf F}_2+\phi'(t)e^{\lambda t}\boldsymbol{\varOmega}
 \end{aligned}
\qquad &\mbox{ for }t>0,
 \end{cases}
 \\
& {\bf v}|_{t=0} = {\bf 0}, \quad \boldsymbol{\ell}|_{t=0} = {\bf 0}, \quad \boldsymbol{\omega}|_{t=0} = {\bf 0}, \\
&
\begin{cases}
({\bf v} - \boldsymbol{\ell} - \boldsymbol{\omega} \times x) \cdot {\bf n} =0 \\
2\mu [\mathbb{D}({\bf v}){\bf n}]_\tau + \alpha[{\bf v} -\boldsymbol{\ell} - \boldsymbol{\omega}\times x]_\tau={\bf 0}
\end{cases} \quad \text{on } \partial \Omega \text{ for } t > 0, \\
& {\bf v}(x,t) \to {\bf 0} \quad \text{as } |x| \to \infty \quad \text{for } t > 0.
\end{align*}

According to Lemma \ref{prop.theorem.strong.Lp}, for any $T > 2$, we have the following a priori estimate on $(0,T)$ with a constant $C > 0$ independent of $T$:
\begin{align}
\notag
\int^T_0 &\Big( \|\partial_t{\bf v}(t) \|_{L^q(\Omega)}^q + \|\nabla^2{\bf v}(t)\|_{L^q(\Omega)}^q + \|\nabla \pi(t)\|_{L^q(\Omega)}^q \\
&\hspace{10mm} + |\boldsymbol{\ell}(t)|^q + |\boldsymbol{\omega}(t)|^q + |\dot{\boldsymbol{\ell}}(t)|^q + |\dot{\boldsymbol{\omega}}(t)|^q \Big) dt \notag \\
\label{w88}
&\leq C \int^T_0 \Big( \|{\bf v}(t)\|_{L^q(\Omega_{2R})}^q + \|\phi(t) e^{\lambda t} {\bf F}_0\|_{L^q(\Omega)}^q + |\phi(t) e^{\lambda t} {\bf F}_1|^q + |\phi(t) e^{\lambda t} {\bf F}_2|^q \notag \\
&\hspace{15mm} + \|\phi'(t)e^{\lambda t} {\bf V}\|_{L^q(\Omega)}^q + |\phi'(t)e^{\lambda t}\boldsymbol{L}|^q + |\phi'(t)e^{\lambda t}\boldsymbol{\varOmega}|^q \Big) dt.
\end{align}

Observe that for the spatial derivatives and rigid body motions, we have
\begin{align*}
\int^T_0 &\Big( \|\nabla^2{\bf v}(t)\|_{L^q(\Omega)}^q + \|\nabla \pi(t)\|_{L^q(\Omega)}^q + |\boldsymbol{\ell}(t)|^q + |\boldsymbol{\omega}(t)|^q \Big) dt \\
&= \Big( \|\nabla^2{\bf V}\|_{L^q(\Omega)}^q + \|\nabla \Pi\|_{L^q(\Omega)}^q + |\boldsymbol{L}|^q + |\boldsymbol{\varOmega}|^q \Big) \int^T_0 |\phi(t)|^q e^{q \operatorname{Re}(\lambda) t} dt.
\end{align*}

Similarly, we can explicitly factor out the temporal integral from the source terms and the local norm:
\begin{align*}
\int^T_0 &\Big( \|{\bf v}(t)\|_{L^q(\Omega_{2R})}^q + \|\phi(t) e^{\lambda t} {\bf F}_0\|_{L^q(\Omega)}^q + |\phi(t) e^{\lambda t} {\bf F}_1|^q + |\phi(t) e^{\lambda t} {\bf F}_2|^q \Big) dt \\
&= \Big( \|{\bf V}\|_{L^q(\Omega_{2R})}^q + \|{\bf F}_0\|_{L^q(\Omega)}^q + |{\bf F}_1|^q + |{\bf F}_2|^q \Big) \int^T_0 |\phi(t)|^q e^{q \operatorname{Re}(\lambda) t} dt.
\end{align*}

Furthermore, since $\partial_t {\bf v} = \lambda \phi e^{\lambda t} {\bf V} + \phi' e^{\lambda t} {\bf V}$ (and similarly for $\dot{\boldsymbol{\ell}}, \dot{\boldsymbol{\omega}}$), an elementary inequality $(a+b)^q \leq 2^{q-1}(a^q + b^q)$ yields the lower bound for the time derivatives:
\begin{align*}
\int^T_0 &\Big( \|\partial_t{\bf v}(t) \|_{L^q(\Omega)}^q + |\dot{\boldsymbol{\ell}}(t)|^q + |\dot{\boldsymbol{\omega}}(t)|^q \Big) dt \\
&\geq c_q |\lambda|^q \Big( \|{\bf V}\|_{L^q(\Omega)}^q + |\boldsymbol{L}|^q + |\boldsymbol{\varOmega}|^q \Big) \int^T_0 |\phi(t)|^q e^{q \operatorname{Re}(\lambda) t} dt \\
&\quad - \int^T_0 \Big( \|\phi'(t)e^{\lambda t} {\bf V}\|_{L^q(\Omega)}^q + |\phi'(t)e^{\lambda t}\boldsymbol{L}|^q + |\phi'(t)e^{\lambda t}\boldsymbol{\varOmega}|^q \Big) dt,
\end{align*}
where $c_q = 2^{1-q}$. Finally, the commutator terms strictly satisfy
\begin{align*}
\int^T_0 &\Big( \|\phi'(t)e^{\lambda t} {\bf V}\|_{L^q(\Omega)}^q + |\phi'(t)e^{\lambda t}\boldsymbol{L}|^q + |\phi'(t)e^{\lambda t}\boldsymbol{\varOmega}|^q \Big) dt \\
&= \Big( \|{\bf V}\|_{L^q(\Omega)}^q + |\boldsymbol{L}|^q + |\boldsymbol{\varOmega}|^q \Big) \int^T_0 |\phi'(t)|^q e^{q \operatorname{Re}(\lambda) t} dt.
\end{align*}

Let $I_T = \int_0^T |\phi(t)|^q e^{q \operatorname{Re}(\lambda) t} \, dt$. 
Since $\operatorname{Re}(\lambda) \geq 0$, we have $e^{q \operatorname{Re}(\lambda) t} \geq 1$. Thus, for any $T > 2$,
\[
    I_T \geq \int_2^T 1 \, dt = T-2,
\]
which clearly implies $I_T \to \infty$ as $T \to \infty$. On the other hand, the integral involving $\phi'$ is evaluated solely on the compact interval $[1, 2]$, meaning it remains a finite constant independent of $T$.

Dividing both sides of the inequality by $I_T$ and passing to the limit as $T \to \infty$, the error terms involving $\phi'$ vanish entirely. This directly gives \eqref{Resolvent2}.
\end{proof}

To conclude the proof, we must show that for the range $1 < q < 3/2$, the local norm on the right-hand side of \eqref{Resolvent2} can be completely absorbed, yielding the exact uniform estimate:
\begin{equation}
\label{Resolvent3}
\begin{aligned}
    & |\lambda| \Big( \|{\bf V}\|_{L^q(\Omega)} + |\boldsymbol{L}| + |\boldsymbol{\varOmega}| \Big) + \|\nabla^2{\bf V}\|_{L^q(\Omega)} + \|\nabla \Pi\|_{L^q(\Omega)} + |\boldsymbol{L}| + |\boldsymbol{\varOmega}| \\
    &\leq C \Big( \|{\bf F}_0\|_{L^q(\Omega)} + |{\bf F}_1| + |{\bf F}_2| \Big).
\end{aligned}
\end{equation}

Assume, for contradiction, that the uniform estimate \eqref{Resolvent3} fails. Then there exist a sequence of parameters $\lambda_m \in \{ z \in \mathbb{C} \mid \operatorname{Re} z \ge 0, 0 < |z| \le \gamma \}$, a sequence of data $({\bf F}_{0,m}, {\bf F}_{1,m}, {\bf F}_{2,m})$ satisfying
\[
 \|{\bf F}_{0,m}\|_{L^q(\Omega)} + |{\bf F}_{1,m}| + |{\bf F}_{2,m}| \to 0 \quad \text{as } m \to \infty,
\]
and corresponding solutions $({\bf V}_m, \Pi_m, \boldsymbol{L}_m, \boldsymbol{\varOmega}_m)$ to the system \eqref{Resolvent1} normalized such that
\begin{align*}
    & |\lambda_m| \Big( \|{\bf V}_m\|_{L^q(\Omega)} + |\boldsymbol{L}_m| + |\boldsymbol{\varOmega}_m| \Big)\\
    & + \|\nabla^2{\bf V}_m\|_{L^q(\Omega)} + \|\nabla \Pi_m\|_{L^q(\Omega)} + |\boldsymbol{L}_m| + |\boldsymbol{\varOmega}_m| = 1.
\end{align*}

By Lemma \ref{lemma.resolvent}, this normalization and the vanishing source terms yield the lower bound for the local norm:
\begin{equation} \label{eq:lower_bound_m}
1 \leq C \|{\bf V}_m\|_{L^q(\Omega_{2R})} + o(1) \quad \text{as } m \to \infty.
\end{equation}

Since the sequence $\lambda_m$ is bounded, we may extract a subsequence such that $\lambda_m \to \lambda_0$ with $\operatorname{Re} \lambda_0 \ge 0$. Using the Banach--Alaoglu theorem and the Rellich--Kondrachov compactness theorem, we can extract a weakly converging subsequence (still denoted by $m$) such that $({\bf V}_m, \Pi_m, \boldsymbol{L}_m, \boldsymbol{\varOmega}_m) \rightharpoonup ({\bf V}, \Pi, \boldsymbol{L}, \boldsymbol{\varOmega})$ weakly in their respective spaces. The strong convergence holds locally in space, i.e., ${\bf V}_m \to {\bf V}$ strongly in $L^q(\Omega_{2R})$. 

The limit satisfies the homogeneous system \eqref{Resolvent1} associated with $\lambda_0$. By weak lower semi-continuity, we have
\[
|\lambda_0| \Big( \|{\bf V}\|_{L^q(\Omega)} + |\boldsymbol{L}| + |\boldsymbol{\varOmega}| \Big) + \|\nabla^2{\bf V}\|_{L^q(\Omega)} + \|\nabla \Pi\|_{L^q(\Omega)} + |\boldsymbol{L}| + |\boldsymbol{\varOmega}| \leq 1.
\]
Furthermore, by the Sobolev embedding $W^{2,q}(\Omega) \hookrightarrow L^{q^*}(\Omega)$ for $q^* = \frac{3q}{3 - 2q}$ (which is strictly valid for $1 < q < 3/2$), it follows that ${\bf V} \in L^{q^*}(\Omega)$.
Consequently, by Lemma \ref{lemma.uniqueness}, the homogeneous system under this integrability condition admits only the trivial solution, meaning $({\bf V}, \boldsymbol{L}, \boldsymbol{\varOmega}) \equiv ({\bf 0}, {\bf 0}, {\bf 0})$.
However, the strong local convergence and \eqref{eq:lower_bound_m} yield
\[
 \|{\bf V}\|_{L^q(\Omega_{2R})} = \lim_{m \to \infty}  \|{\bf V}_m\|_{L^q(\Omega_{2R})} \geq \frac{1}{C} > 0,
\]
which strictly contradicts ${\bf V} \equiv {\bf 0}$.

\subsection{ Sectoriality for general $ q \in (1, \infty)$}

Recall that the operator \( -\mathbb{A}_{\alpha,2} \) is sectorial with angle $\theta$, 
 for any  \( \theta \in \left( \frac{\pi}{2}, \pi \right) \)  satisfying
\[
|\lambda| \left\| (\lambda I + \mathbb{A}_{\alpha,2})^{-1} {\bf f} \right\|_{\tilde{\mathbb{X}}_2} \leq C_2 \|{\bf f}\|_{\tilde{\mathbb{X}}_2} \quad \text{for all } \lambda \in \Sigma_\theta,
\]
 and  the operator \( -{\mathbb A}_{\alpha,q_0} \) is also sectorial with  $\theta$ for some \( \theta_0 \in \left( \frac{\pi}{2}, \pi \right) \) 
satisfying
\[
|\lambda| \left\| (\lambda I + {\mathbb A}_{\alpha,q_0})^{-1} {\bf f} \right\|_{\tilde{\mathbb X}_{q_0}} \leq C_{q_0} \|{\bf f}\|_{\tilde{\mathbb X}_{q_0}} \quad \text{for all } \lambda \in \Sigma_{\theta_0}.
\]



{\bf  Case: $\frac{3}{2}\leq q<2$.}

Choose \( q_0 \in \left(1, \frac{3}{2} \right) \).  
Then, there exists \( s \in (0,1) \) such that
\[
\frac{1}{q} = \frac{1-s}{q_0} + \frac{s}{2}.
\]

%
%

By the interpolation theorem, the operator \( -{\mathbb A}_{\alpha,q} \) is also sectorial with  the angle \(  \theta_0 \), 
and satisfies
\[
|\lambda| \left\| (\lambda I + {\mathbb A}_{\alpha,q})^{-1}{\bf  f} \right\|_{\tilde{\mathbb X}_q} \leq C_{q_0}^{1-s} C_2^s \|{\bf f}\|_{\tilde{\mathbb X}_q}\quad \text{for all } \lambda \in \Sigma_{\theta_0}.
\]

{\bf Case:  \( 2 < q < \infty \).}

 In this case, we have \( 1 < q' < 2 \), where \( q' \) is the Hölder conjugate of \( q \).  
Using the duality argument, we observe that
\[
\langle (\lambda I + {\mathbb A}_{\alpha,q})^{-1}{\bf  f}, \boldsymbol{\phi} \rangle = \langle{\bf  f}, (\lambda I + {\mathbb A}_{\alpha,q'})^{-1} \boldsymbol{\phi} \rangle,
\]
which implies the estimate
\[
|\lambda| \left\| (\lambda I + {\mathbb A}_{\alpha,q})^{-1}{\bf  f} \right\|_{\tilde{\mathbb X}_q} \leq C_q \|{\bf f} \|_{\tilde{\mathbb X}_q}\quad \text{for all } \lambda \in \Sigma_{\theta_0}.
\]

\subsection{Bounded Analyticity}

From the result established in the previous section, the operator \( -\mathbb{A}_{\alpha,q} \) is sectorial with angle \( \omega = 0 \) on \( \tilde{\mathbb{X}}^q \) for \( 1 < q < \infty \).  

By applying Proposition 2.1.1 in \cite{lunardi}, it follows that \( -\mathbb{A}_{\alpha,q} \) generates a   bounded analytic \( C_0 \)-semigroup  $\{e^{-\mathbb{A}_{\alpha,q}t}\}_{t \geq 0}$ on \( \tilde{\mathbb{X}}^q \).  
Moreover, the following estimates hold  for all ${\bf f} \in \tilde{\mathbb{X}}^q$ and $t > 0$:
\begin{align*}
\left\| e^{-{\mathbb A}_{\alpha,q}t}{\bf  f} \right\|_{\tilde{\mathbb{X}}_q} 
& \leq c \|{\bf f}\|_{\tilde{\mathbb{X}}_q}, \\
\left\| {\mathbb A}_{\alpha,q} e^{-{\mathbb A}_{\alpha,q}t}{\bf  f} \right\|_{\tilde{\mathbb{X}}_q} 
& \leq c t^{-1} \|{\bf f}\|_{\tilde{\mathbb{X}}_q},
\end{align*}
for all \( {\bf f} \in \tilde{\mathbb{X}}_q \) and \( t > 0 \), where \( c > 0 \) is a constant independent of \( t \).

This concludes the proof of Theorem \ref{analyticity.thm}.

\section{\bf Proof of Theorem \ref{largetime.thm}
}

\label{section.largetime.thm}

Let \( {\bf u}_0 \simeq ({\bf v}_0, \boldsymbol{\ell}_0, \boldsymbol{\omega}_0) \in \tilde{\mathbb{X}}^r \), and define the solution \( {\bf u}(t) := e^{-t \mathbb{A}_{\alpha,q}} {\bf u}_0 \simeq ({\bf v}, \boldsymbol{\ell}, \boldsymbol{\omega}) \). 
Let \( \pi \) be the associated pressure such that the quadruple \( ({\bf v}, \boldsymbol{\ell},  \boldsymbol{\omega}, \pi) \) satisfies the homogeneous system \eqref{e01} (i.e., \( {\bf f}_0, {\bf f}_1, {\bf f}_2 \equiv 0 \)) under the conditions \textup{(C1)}–\textup{(C3)}. 
Recall the norm equivalence:
\[
\|{\bf u}\|_{\tilde{\mathbb{X}}^r} \approx \|{\bf u}\|_{L^r(\mathbb{R}^3)} \approx \|{\bf v}\|_{L^r(\Omega)} + |\boldsymbol{\ell}|  + |\boldsymbol{\omega}|.
\]

\subsection{$L^r$–$L^q$ Estimate for short time $0 < t \leq T$.}
\label{smallt}

%

Let $1 < r \leq q < \infty$ and $0 < t \leq T$. Here and in what follows, $c(T)$ denotes a generic positive constant depending on $T$, which may change from line to line. For any $\alpha \geq 0$, since $1 \leq (1+T)^\alpha (1+t)^{-\alpha}$ on $(0, T]$, we can always include the factor $(1+t)^{-\alpha}$ in our upper bounds by simply absorbing $(1+T)^\alpha$ into $c(T)$. We systematically apply this convention throughout this subsection to unify the notation with the large-time estimates. 

First, taking $\alpha = \frac{3}{2}\big(\frac{1}{r}-\frac{1}{q}\big)$, we obtain
\begin{equation}
\label{loc1}
|\boldsymbol{\ell}(t)| + |\boldsymbol{\omega}(t)| \leq \|{\bf u}(t)\|_{\tilde{\mathbb{X}}^r} 
\leq c(T) \|{\bf u}_0\|_{\tilde{\mathbb{X}}^r} 
\leq c(T)(1 + t)^{-\frac{3}{2}\left(\frac{1}{r} - \frac{1}{q}\right)} \|{\bf u}_0\|_{\tilde{\mathbb{X}}^r}.
\end{equation}

Since $|\dot{\boldsymbol{\ell}}| + |\dot{\boldsymbol{\omega}}| \leq \|\partial_t {\bf u}(t)\|_{\tilde{\mathbb{X}}^r}$ and $ \|\partial_t {\bf u}(t)\|_{\tilde{\mathbb{X}}^r} \leq c t^{-1} \|{\bf u}(t)\|_{\tilde{\mathbb{X}}^r} $, \eqref{loc1} implies that
\begin{align}
\label{loc2}
|\dot{\boldsymbol{\ell}}| + |\dot{\boldsymbol{\omega}}| 
\leq c(T) t^{-1} (1 + t)^{-\frac{3}{2}\left(\frac{1}{r} - \frac{1}{q}\right)} \|{\bf u}_0\|_{\tilde{\mathbb{X}}^r}.
\end{align}

By the Gagliardo--Nirenberg interpolation inequality for exterior domains \cite{maremonti-crispo},
\[
\|{\bf v}(t)\|_{L^q(\Omega)} \leq c \|\nabla^{2m} {\bf v}(t)\|_{L^r(\Omega)}^\theta \|{\bf v}(t)\|_{L^r(\Omega)}^{1 - \theta},
\]
where $\theta$ satisfies
\[
\frac{1}{q} = \theta\left(\frac{1}{r} - \frac{2m}{3}\right) + (1 - \theta)\frac{1}{r}, \quad \text{provided } r > \frac{3}{m}.
\]
This algebraic relation explicitly yields $\theta = \frac{3}{2m}\left(\frac{1}{r} - \frac{1}{q}\right)$. 
From Lemma \ref{strong.reg2}, we have the regularity estimate
\[
\|\nabla^{2m} {\bf v}(t)\|_{L^r(\Omega)} \leq c(t^{-m} + 1)\|{\bf u}_0\|_{\tilde{\mathbb{X}}^r}.
\]
Combining this with \eqref{loc1}, the exponent of $t$ evaluates precisely to $-m\theta = -\frac{3}{2}\left(\frac{1}{r} - \frac{1}{q}\right)$, which leads to
\begin{equation}
\label{loc3}
\|{\bf v}(t)\|_{L^q(\Omega)} \leq c(T)t^{-\frac{3}{2}\left(\frac{1}{r} - \frac{1}{q}\right)}\|{\bf u}_0\|_{\tilde{\mathbb{X}}^r}.
\end{equation}

From \eqref{loc1} and \eqref{loc3}, the time derivative bound follows naturally:
\begin{align}
\label{loc4}
\notag \|\partial_t {\bf v}(t)\|_{L^q(\Omega)} 
&\leq \|\partial_t {\bf u}(t)\|_{\tilde{\mathbb{X}}^q} 
\leq c t^{-1} \|{\bf u}(t)\|_{\tilde{\mathbb{X}}^q} 
= c t^{-1} \Big( \|{\bf v}(t)\|_{L^q(\Omega)} + |\boldsymbol{\ell}(t)| + |\boldsymbol{\omega}(t)| \Big) \\
&\leq c(T) t^{-1 - \frac{3}{2}\left( \frac{1}{r} - \frac{1}{q} \right)} \|{\bf u}_0\|_{\tilde{\mathbb{X}}^r}.
\end{align}

Applying Lemma \ref{prop.ogawa} together with the estimates \eqref{loc3} and \eqref{loc4}, we obtain the bounds for the second derivatives and the pressure. Note that on the short-time interval $t \in (0, T]$, the singular term $t^{-1}$ dominates, yielding:
\begin{align}
\label{loc5}
\notag \|\nabla^2 {\bf v}(t)\|_{L^q(\Omega)} + \|\nabla \pi(t)\|_{L^q(\Omega)} 
&\leq c \|\partial_t {\bf v}(t)\|_{L^q(\Omega)} + c \|{\bf v}(t)\|_{L^q(\Omega_R)} \\
&\leq c(T) t^{- 1 - \frac{3}{2}\left( \frac{1}{r} - \frac{1}{q} \right)} \|{\bf u}_0\|_{\tilde{\mathbb{X}}^r}.
\end{align}

Finally, interpolating between \eqref{loc3} and \eqref{loc5} via the Gagliardo--Nirenberg inequality yields the exact decay for the gradient:
\begin{align}
\label{loc6}
\|\nabla {\bf v}(t)\|_{L^q(\Omega)} 
&\leq c \|\nabla^2 {\bf v}(t)\|_{L^q(\Omega)}^{1/2} \|{\bf v}(t)\|_{L^q(\Omega)}^{1/2} \notag \\
&\leq c(T) t^{- \frac{1}{2} - \frac{3}{2}\left( \frac{1}{r} - \frac{1}{q} \right)} \|{\bf u}_0\|_{\tilde{\mathbb{X}}^r}.
\end{align}

\subsection{$L^r$–$L^q$ Estimate for large time $t > 2$}
\label{larget}
%
%


Having established the local-in-time bounds on the bounded interval $(0, T]$, we now focus on the genuine asymptotic decay of the solution for large time $t > 2$.

 \subsubsection{\bf $L^r$-$L^2$ Estimates for the Velocity Field ${\bf u}$.}

We first establish the following  decay estimate:
\begin{lemm}
\label{L2L2}
There exists a constant $c > 0$ such that
\begin{equation}
\|\nabla {\bf v}(t)\|_{L^2(\Omega)} + |\boldsymbol{\ell}(t)| + |\boldsymbol{\omega}(t)| \leq c t^{-\frac{1}{2}} \|{\bf u}_0\|_{L^2(\mathbb{R}^3)}.
\end{equation}
\end{lemm}

\begin{proof}
Let ${\bf u}(t) = e^{-t\mathbb{A}_{\alpha,2}} {\bf u}_0$. By the properties of the analytic semigroup, we observe that
\begin{align*}
\langle \mathbb{A}_{\alpha,2} {\bf u}, {\bf u} \rangle &\leq \|\mathbb{A}_{\alpha,2} {\bf u}\| \|{\bf u}\| \\
&= \|\mathbb{A}_{\alpha,2} e^{-t\mathbb{A}_{\alpha,2}} {\bf u}_0\| \|e^{-t\mathbb{A}_{\alpha,2}} {\bf u }_0\| \leq c t^{-1} \|{\bf u}_0\|_{L^2}^2.
\end{align*}
Recalling the identity from Lemma \ref{prop.pre2}:
\begin{equation*}
\langle \mathbb{A}_{\alpha,2} {\bf u}, {\bf u} \rangle = \mu \|\mathbb{D}({\bf v})\|_{L^2(\Omega)}^2 + \alpha \|[{\bf v}-\boldsymbol{\ell}(t)-\boldsymbol{\omega}(t)\times x]_\tau\|_{L^2(\partial\Omega)}^2,
\end{equation*}
we deduce that
\begin{equation}
\label{analiticity_L2}
\|\mathbb{D}({\bf v}(t))\|_{L^2(\Omega)} + \|[{\bf v}-\boldsymbol{\ell}(t)-\boldsymbol{\omega}(t)\times x]_\tau\|_{L^2(\partial\Omega)} \leq c t^{-\frac{1}{2}} \|{\bf u}_0\|_{L^2(\mathbb{R}^3)}.
\end{equation}
Finally, applying the Korn-type inequality (Lemma \ref{korn_exterior}) to \eqref{analiticity_L2}, we obtain the desired decay estimate in Lemma \ref{L2L2}.
\end{proof}

Combining Lemma \ref{L2L2} with standard interpolation properties leads to the following decay rate estimates for the velocity field in $L^q$ spaces.

\begin{coro}
\label{coro.LqL2}
Let $({\bf v},\boldsymbol{\ell}, \boldsymbol{\omega})$ and $({\bf v}_0, \boldsymbol{\ell}_0, \boldsymbol{\omega}_0)$ be the same as those defined in Lemma \ref{analiticity_L2}.
Then for $t>1$, we have 
$$
\|{\bf v}(t)\|_{L^q(\Omega)}\leq 
\begin{cases}
ct^{-\frac{3}{2}\left(\frac{1}{2}-\frac{1}{q}\right)}\|{\bf u}_0\|_{L^2({\mathbb R}^3)}, & 2\leq q\leq 6, \\
ct^{-\frac{1}{2}}\|{\bf u}_0\|_{L^2({\mathbb R}^3)}, & 6<q\leq \infty.
\end{cases}
$$
\end{coro}

\begin{proof}
According to Lemma \ref{prop.ogawa}, we have the following strong regularity estimate:
\begin{align*}
\|\nabla^2{\bf v}\|_{L^2(\Omega)}+\|\nabla \pi\|_{L^2(\Omega)}&\leq c\|\partial_t{\bf v}\|_{L^2(\Omega)}+c\|{\bf v}\|_{L^2(\Omega_R)}+c|\boldsymbol{\ell}|+c|\boldsymbol{\omega}|\\
&\leq ct^{-\frac{1}{2}}\|{\bf u}_0\|_{L^2({\mathbb R}^3)}, \quad t>1.
\end{align*}
Here, we used the estimate by Shibata and Soga \cite{shibata-soga} that $\|{\bf v}\|_{L^2(\Omega_R)}\leq c\|\nabla{\bf v}\|_{L^2(\Omega)}$.

By the interpolation property, we have
$$
\|{\bf v}\|_{L^q(\Omega)}\leq c\|\nabla {\bf v}\|_{L^2(\Omega)}^\theta \|{\bf v}\|_{L^2(\Omega)}^{1-\theta}, \quad \theta=3\left(\frac{1}{2}-\frac{1}{q}\right), \quad 2\leq q\leq 6.
$$
For $6<q\leq \infty$, we directly apply the Gagliardo--Nirenberg interpolation inequality to obtain
\begin{align*}
\|{\bf v}\|_{L^q(\Omega)} &\leq c\|\nabla^2{\bf v}\|_{L^2(\Omega)}^{1-\theta}\|\nabla {\bf v}\|_{L^2(\Omega)}^\theta, \quad \theta=3\left(\frac{1}{q}+\frac{1}{6}\right).
\end{align*}
Applying the bounds from Lemma \ref{L2L2} and the strong regularity estimate yields
$$
\|{\bf v}\|_{L^q(\Omega)} \leq c \big(t^{-\frac{1}{2}}\big)^{1-\theta} \big(t^{-\frac{1}{2}}\big)^\theta \|{\bf u}_0\|_{L^2({\mathbb R}^3)} = c t^{-\frac{1}{2}} \|{\bf u}_0\|_{L^2({\mathbb R}^3)}.
$$
This gives the desired decay rates and completes the proof.
\end{proof}

By a duality argument, we also conclude the following:

\begin{coro}
\label{coro.L2Lr}
Let $({\bf v},\boldsymbol{\ell}, \boldsymbol{\omega})$ and $({\bf v}_0, \boldsymbol{\ell}_0, \boldsymbol{\omega}_0)$ be the same as those defined in Corollary \ref{coro.LqL2}.
Then for $t>1$, we have
\begin{equation}
\label{L2Lr1}
\|{\bf v}(t)\|_{L^2(\Omega)}\leq 
\begin{cases}
ct^{-\frac{3}{2}\left(\frac{1}{r}-\frac{1}{2}\right)}\|{\bf u}_0\|_{L^r({\mathbb R}^3)}, & \frac{6}{5}\leq r\leq 2, \\
ct^{-\frac{1}{2}}\|{\bf u}_0\|_{L^r({\mathbb R}^3)}, & 1<r\leq \frac{6}{5}.
\end{cases}
\end{equation}
\end{coro}

We now wish to extend the estimate $\eqref{L2Lr1}_1$ to the case $1 < r \leq 2$:

\begin{lemm}
\label{lemm.LqL2}
Let $1 < r \leq 2$. Then
\begin{equation}
\label{lemm.LqL20}
\|e^{-t{\mathbb A}_\alpha}{\bf u}_0\|_{\tilde{\mathbb X}^2} \leq ct^{-\frac{3}{2}\left(\frac{1}{r}-\frac{1}{2}\right)}\|{\bf u}_0\|_{\tilde{\mathbb X}^r}.
\end{equation}
\end{lemm}

\begin{proof}
By a standard density argument, we may assume without loss of generality that the initial data and the solution are sufficiently smooth to justify the following energy arguments. 

Taking the $L^2$ inner product of \eqref{e01} with ${\bf v}$ and utilizing the boundary conditions, we obtain the energy identity:
\begin{align}
 \label{gronwall}
 \frac{1}{2}\frac{d}{dt}\mathcal{E}(t) + 2\mu \|\mathbb{D}({\bf v}(t))\|_{L^2(\Omega)}^2 + \alpha \|[{\bf v}-\boldsymbol{\ell}-\boldsymbol{\omega}\times x]_\tau\|_{L^2(\partial \Omega)}^2 = 0,
\end{align}
where $\mathcal{E}(t) := \|{\bf v}(t)\|_{L^2(\Omega)}^2 + m|\boldsymbol{\ell}(t)|^2 + |{\mathbb J}_0 |\boldsymbol{\omega}(t)|^2$.

By the Korn type inequality in Lemma \ref{korn_exterior}, there exists a constant $c > 0$ such that
\[
c(\|\nabla {\bf v}\|_{L^2(\Omega)}+|\boldsymbol{\ell}|+|\boldsymbol{\omega}|)\leq \|\mathbb{D}({\bf v}(t))\|_{L^2(\Omega)}+ \alpha \|[{\bf v}-\boldsymbol{\ell}-\boldsymbol{\omega}\times x]_\tau\|_{L^2(\partial \Omega)}.
\]
By the Gagliardo-Nirenberg interpolation inequality in 3D, we have
\[
\|{\bf v}\|_{L^2(\Omega)} \leq c\|\nabla {\bf v}\|_{L^2(\Omega)}^\theta \|{\bf v}\|_{L^r(\Omega)}^{1-\theta}, \quad \text{where } \theta = \frac{3(2-r)}{6-r}, \quad 1 < r \leq 2.
\]
Since $\|{\bf v}\|_{L^r(\Omega)} \leq c\|{\bf u}_0\|_{\tilde{\mathbb X}^r}$ and the rigid body velocities satisfy $(|\boldsymbol{\ell}|+|\boldsymbol{\omega}|) \leq c\|{\bf u}_0\|_{\tilde{\mathbb X}^r}$, combined with the Korn-type inequality in Lemma \ref{korn_exterior}, we obtain:
\[
\mathcal{E}(t)^{\frac{1}{2}} \approx \|{\bf v}\|_{L^2(\Omega)} + |\boldsymbol{\ell}| + |\boldsymbol{\omega}| \leq c\Big(\|\mathbb{D}({\bf v}(t))\|_{L^2(\Omega)} + \alpha \|[{\bf v}-\boldsymbol{\ell}-\boldsymbol{\omega}\times x]_\tau\|_{L^2(\partial \Omega)}\Big)^\theta \|{\bf u}_0\|_{\tilde{\mathbb X}^r}^{1-\theta}.
\]
Squaring both sides, the total dissipation term is bounded from below by the energy:
\[
\|\mathbb{D}({\bf v}(t))\|_{L^2(\Omega)}^2 + \alpha \|[{\bf v}-\boldsymbol{\ell}-\boldsymbol{\omega}\times x]_\tau\|_{L^2(\partial \Omega)}^2 \geq c \mathcal{E}(t)^{\frac{1}{\theta}} \|{\bf u}_0\|_{\tilde{\mathbb X}^r}^{-\frac{2(1-\theta)}{\theta}}.
\]
Substituting this into the energy identity \eqref{gronwall}, we reduce it to the following Gronwall-type differential inequality for $\mathcal{E}(t)$:
\begin{equation}
\label{gronwall1}
\frac{d}{dt}\mathcal{E}(t) + c \mathcal{E}(t)^{\frac{1}{\theta}} \|{\bf u}_0\|_{\tilde{\mathbb X}^r}^{-\frac{2(1-\theta)}{\theta}} \leq 0.
\end{equation}
Solving this differential inequality (note that the exponent is $p = \frac{1}{\theta} > 1$), we obtain $\mathcal{E}(t) \leq C t^{-\frac{1}{p-1}} \|{\bf u}_0\|_{\tilde{\mathbb X}^r}^2$. Since
\[
\frac{1}{p-1} = \frac{\theta}{1-\theta} = \frac{3(2-r)}{6-r} \times \frac{6-r}{2r} = \frac{3(2-r)}{2r} = 3\left(\frac{1}{r}-\frac{1}{2}\right),
\]
we arrive at the desired decay rate for the energy $\mathcal{E}(t) \leq c t^{-3(\frac{1}{r}-\frac{1}{2})} \|{\bf u}_0\|_{\tilde{\mathbb X}^r}^2$. Taking the square root yields:
\begin{equation}
\label{final_decay}
\|{\bf v}(t)\|_{L^2(\Omega)} + |\boldsymbol{\ell}(t)| + |\boldsymbol{\omega}(t)| \leq ct^{-\frac{3}{2}\left(\frac{1}{r}-\frac{1}{2}\right)} \|{\bf u}_0\|_{\tilde{\mathbb X}^r}, \quad t > 1.
\end{equation}
%
%
%

\end{proof}

\subsubsection{\bf Estimates for \({\bf u}\) for other values of \(q\) and \(r\)}

Based on the estimate for \( 1< r \leq 2 \), we can derive estimates for other values of \( r \) and \( q \) as follows:

\begin{itemize}



\item  \textbf{Case 1. 
}

 Let \(1< r \leq 2 \leq q <\infty\). Then,  \(1 < r,q' \leq 2 \).
 From Lemma \ref{lemm.LqL2}, we have
\begin{align*}
\langle e^{-t{\mathbb A}_\alpha }{\bf u}_0, \boldsymbol{\phi}_0 \rangle &= \langle e^{-\frac{t}{2}{\mathbb A}_\alpha }{\bf u}_0, e^{-\frac{t}{2}{\mathbb A}_\alpha }\boldsymbol{\phi}_0 \rangle \leq \| e^{-\frac{t}{2}{\mathbb A}_\alpha }{\bf u}_0 \|_{L^2} \| e^{-\frac{t}{2}{\mathbb A}_\alpha }\boldsymbol{\phi}_0 \|_{L^2}\\
&\leq 
c t^{-\frac{3}{2}(\frac{1}{r}-\frac{1}{q})} \|{\bf u}_0\|_{L^{r}} \|\boldsymbol{\phi}_0\|_{L^{q'}}.
\end{align*}
Hence, by the duality argument, we obtain the estimate:
\begin{equation}
\label{lemm.LqL2_1}
\|e^{-t{\mathbb A}_\alpha }{\bf u}_0\|_{\tilde{\mathbb X}^q} \leq c t^{-\frac{3}{2}(\frac{1}{r}-\frac{1}{q})} \|{\bf u}_0\|_{\tilde{\mathbb X}^r} \quad\mbox{for } 1< r \leq 2 \leq q <\infty.
\end{equation}

\item  \textbf{Case 2.
}
 
 Let \(2 \leq  r \leq  q <\infty\). Then, \(1< q' \leq  r' \leq  2\). There exists \(\theta \in (0,1)\) such that \(\frac{1}{r'} = \frac{1-\theta}{q'} + \frac{\theta}{2}\). From the interpolation property and from Lemma \ref{lemm.LqL2}, we have
\begin{align*}
\langle e^{-t{\mathbb A}_\alpha }{\bf u}_0, \boldsymbol{\phi}_0 \rangle &= \langle {\bf u}_0, e^{-t{\mathbb A}_\alpha }\boldsymbol{\phi}_0 \rangle \leq \|{\bf u}_0\|_{L^r} \| e^{-\frac{t}{2}{\mathbb A}_\alpha }\boldsymbol{\phi}_0 \|_{L^{r'}}\\
&\leq \|{\bf u}_0\|_{L^{r}} \| e^{-\frac{t}{2}{\mathbb A}_\alpha }\boldsymbol{\phi}_0 \|_{L^{q'}}^{1-\theta} \| e^{-\frac{t}{2}{\mathbb A}_\alpha }\boldsymbol{\phi}_0 \|_{L^{2}}^{\theta} \\
&\leq c t^{-\frac{3}{2}(\frac{1}{r}-\frac{1}{q})} \|{\bf u}_0\|_{L^{r}} \|\boldsymbol{\phi}_0\|_{L^{q'}}.
\end{align*}
Hence, a duality argument leads to the estimate:
\begin{equation}
\label{lemm.LqL2_2}
\|e^{-t{\mathbb A}_\alpha}{\bf u}_0\|_{\tilde{\mathbb X}^q} \leq c t^{-\frac{3}{2}(\frac{1}{r}-\frac{1}{q})} \|{\bf u}_0\|_{\tilde{\mathbb X}^r}\mbox{ for } 2 \leq  r \leq q <\infty.
\end{equation}

\item  \textbf{Case 3.
}  

Let \( 1< r \leq  q \leq 2 \). Then, we have \( 2\leq  q' \leq  r' <\infty \).

From the estimate in \eqref{lemm.LqL2_2}, we obtain:
\[
\|e^{-t{\mathbb A}_\alpha }\boldsymbol{\phi}_0\|_{\tilde{\mathbb X}^{r'}} \leq ct^{-\frac{3}{2}\left(\frac{1}{q'} - \frac{1}{r'}\right)} \|\boldsymbol{\phi}_0\|_{\tilde{\mathbb X}^{q'}}.
\]
Hence, we can write:
\begin{align*}
\langle e^{-t{\mathbb A}_\alpha }{\bf u}_0, \boldsymbol{\phi}_0 \rangle& = \langle {\bf u}_0, e^{-t{\mathbb A}_\alpha }\boldsymbol{\phi}_0 \rangle \leq \|{\bf u}_0\|_{L^r} \|e^{-\frac{t}{2}{\mathbb A}_\alpha }\boldsymbol{\phi}_0\|_{L^{r'}}\\
& \leq ct^{-\frac{3}{2}\left(\frac{1}{r} - \frac{1}{q}\right)} \|{\bf u}_0\|_{L^r} \|\boldsymbol{\phi}_0\|_{L^{q'}}.
\end{align*}
%

The duality argument leads us to the following estimate:
\begin{equation}
\label{lemm.LqL2_3}
\|e^{-t{\mathbb A}_\alpha }{\bf u}_0\|_{\tilde{\mathbb X}^q} \leq ct^{-\frac{3}{2}\left(\frac{1}{r} - \frac{1}{q}\right)} \|{\bf u}_0\|_{\tilde{\mathbb X}^r}, \quad 1<  r \leq q \leq  2.
\end{equation}
%

\end{itemize}

%
%
%
%

%

By combining the estimates \eqref{lemm.LqL20}–\eqref{lemm.LqL2_3}, we obtain the first estimate stated in Theorem \ref{largetime.thm} for \( t > 2 \).

%


\begin{coro}
\label{coro.LqLr.larget}
Let $1 < r \leq q < \infty$ and $t > 1$. Then for ${\bf u}(t) := e^{-t{\mathbb A}_{\alpha,q}}{\bf u}_0$, the following decay estimates hold:
\begin{align}
|\boldsymbol{\ell}(t)| + |\boldsymbol{\omega}(t)| &\leq C t^{-\frac{3}{2}\left(\frac{1}{r}-\frac{1}{q}\right)} \|{\bf u}_0\|_{\tilde{\mathbb{X}}^r}, 
\\
\|{\bf u}(t)\|_{L^q(\Omega)} &\leq C t^{-\frac{3}{2}\left(\frac{1}{r}-\frac{1}{q}\right)} \|{\bf u}_0\|_{\tilde{\mathbb{X}}^r}.  
\end{align}
\end{coro}

\begin{rem}
Combining the first estimate stated in Theorem \ref{largetime.thm} for  \( q=r \)
 and \( q=m \), we obtain the estimate:
\begin{align}
\label{LinftyU}
|\boldsymbol{\ell}(t)| + |\boldsymbol{\omega}(t)| \leq c(m) (1+t)^{-\frac{3}{2}\left(\frac{1}{r} - \frac{1}{m} \right)} \|{\bf u}_0\|_{\tilde{\mathbb X}^r}.
\end{align}
Here, $m$ can be taken to be arbitrarily large integer.

\end{rem}

\subsubsection{ \bf Estimate for $\partial_t {\bf u}$  }

Since
\[
\|\partial_t e^{-t{\mathbb A}_\alpha }{\bf u}_0\|_{L^q}
\leq ct^{-1}\|e^{-\frac{t}{2}{\mathbb A}_\alpha}{\bf u}_0\|_{L^q},
\]
we arrive at the following estimate, which corresponds to the second assertion of Theorem \ref{largetime.thm} for \( t > 2 \):
\begin{align}
\label{lemm.LqL2_12}
\|\partial_t {\bf u}\|_{\tilde{\mathbb X}^q}=\|\partial_t e^{-t{\mathbb A}_\alpha}{\bf u}_0\|_{\tilde{\mathbb X}^q}
\leq ct^{-1 - \frac{3}{2}\left(\frac{1}{r} - \frac{1}{q}\right)}\|{\bf u}_0\|_{\tilde{\mathbb X}^r}.
\end{align}

For ${\bf u} \simeq ({\bf v}, \boldsymbol{\ell}, \boldsymbol{\omega})$, we recall the norm equivalence $\|{\bf u}\|_{\tilde{\mathbb{X}}^r} \approx \|{\bf v}\|_{L^r(\Omega)} + |\boldsymbol{\ell}| + |\boldsymbol{\omega}|$. The same relation applies to the time derivative:

\begin{equation} \label{time.equiv}
\|\partial_t {\bf u}\|_{\tilde{\mathbb{X}}^q} \approx \|\partial_t {\bf v}\|_{L^q(\Omega)} + |\dot{\boldsymbol{\ell}}| + |\dot{\boldsymbol{\omega}}|.
\end{equation}
Note that this equivalence implies 
\[\|\partial_t {\bf v}\|_{L^q(\Omega)} + |\dot{\boldsymbol{\ell}}| + |\dot{\boldsymbol{\omega}}| \leq  ct^{-1 - \frac{3}{2}\left(\frac{1}{r} - \frac{1}{q}\right)}\|{\bf u}_0\|_{\tilde{\mathbb X}^r}.\]

\subsubsection {\bf   Estimate of $\nabla^2{\bf v}, \nabla \pi$ }



Let \({\bf v}_B := \boldsymbol{\ell}+ \boldsymbol{\omega}\times x \) and \( \tilde{\bf v}_B =-\frac{1}{2} \nabla \times \Big(\zeta(x)[ x\times \boldsymbol{\ell} + \boldsymbol{\omega} |x|^2]\Big) \) for a smooth cut-off function \( \zeta \) with \( \zeta = 1 \) near \( \partial \Omega \). 
Then, we have \( \text{div} \,\tilde{\bf v}_B = 0 \) in \( \Omega \) and \( \tilde{\bf v}_B = {\bf v}_B \) on \( \partial \Omega \), and $({\bf v}-\tilde{\bf v}_B, \pi)$ satisfies the equations:
\begin{align*}
\begin{cases}
&\begin{array}{l}
-\mu \Delta ({\bf v}-\tilde{\bf v}_B)+\nabla \pi=-\partial_t{\bf v}+\mu \Delta \tilde{\bf v}_B\\
\operatorname{div}({\bf v}-\tilde{\bf v}_B)=0
\end{array}\mbox{ in }\Omega,\\
&\begin{array}{l}
({\bf v}-\tilde{\bf v}_B)\cdot {\bf n}=0\\
2\mu[{\mathbb  D}({\bf v}-\tilde{\bf v}_B){\bf n}]_\tau+\alpha[{\bf v}-\tilde{\bf v}_B]_\tau={\bf 0}
\end{array}\mbox{ on }\partial \Omega.
\end{cases}
\end{align*}
Hence, applying  Lemma \ref{prop.ogawa}  to  $ ({\bf v}-\tilde{\bf v}_B, \pi)   $,
 first estimate  in Theorem \ref{largetime.thm} and \eqref{lemm.LqL2_12} yield the following bound:
\begin{align}
\label{time.final1}
\notag
\|D^2_x {\bf v}(t)\|_{L^q(\Omega)} + \|\nabla \pi(t)\|_{L^q} 
&\leq c\|\partial_t {\bf v}(t)\|_{L^q(\Omega)} + c\|{\bf v}(t)\|_{L^{m}(\Omega_{2R})}
+ c|\boldsymbol{\ell}(t)| + c|\boldsymbol{\omega}(t)| \\
\notag
&\leq c(m)\left( t^{-1 - \frac{3}{2}\left(\frac{1}{r} - \frac{1}{q}\right)} 
+ t^{-\frac{3}{2}\left(\frac{1}{r} - \frac{1}{m} \right)} \right)
\|{\bf u}_0\|_{\tilde{\mathbb X}^r}.
\end{align}
Here $m$ can be arbitraily large integer.

\subsubsection{\bf Estimate of  $\nabla {\bf  v}$}
\label{gradient.decay}

\noindent{\bf $\bullet$ Estimate near the boundary.}

By the interpolation theorem for bounded domains (see \cite{gagliardo}), we observe that
\begin{align*}
\|\nabla {\bf v}\|_{L^q(\Omega_{2R})} 
&\leq c \|\nabla^2 {\bf v}\|_{L^q(\Omega_{2R})} + c \|{\bf v}\|_{L^{m}(\Omega_R)} \\
&\leq c(m)\left( t^{-1 - \frac{3}{2}\left(\frac{1}{r} - \frac{1}{q}\right)} 
+ t^{-\frac{3}{2}\left(\frac{1}{r} - \frac{1}{m}\right)} \right)\|{\bf u}_0\|_{\tilde{\mathbb{X}}^r}, \quad t \geq 2.
\end{align*}
Here, $m$ can be an arbitrarily large integer.

\medskip
\noindent{\bf $\bullet$ Estimate at large distance.}

Assume that $\Omega^c \subset B_{\frac{R}{2}}$ and $\Omega \cap B_R^c = B_R^c$.
Let $\zeta \in C^\infty({\mathbb R}^3)$ be a smooth cut-off function such that
\[
\zeta(x) = 
\begin{cases}
1, & \text{for } |x| \geq R, \\
0, & \text{for } |x| \leq \frac{R}{2}.
\end{cases}
\]
As in the proof of Lemma~\ref{prop.theorem.strong.Lp}, the compatibility condition $\int_{B_R \setminus B_{R/2}} {\bf v} \cdot \nabla \zeta \, dx = 0$ holds.
Hence, by Lemma~\ref{lem.bogovskii_cutoff}, there exists an auxiliary vector field ${\bf w}_1$ compactly supported in $B_R \setminus B_{R/2}$ satisfying
\[
\operatorname{div} {\bf w}_1 = {\bf v} \cdot \nabla \zeta \quad \text{in } {\mathbb R}^3.
\]
Combining \eqref{grad1} with the continuous embedding $L^q(\Omega_R) \hookrightarrow W^{-1,q}(\Omega_R)$, H\"older's inequality (for any arbitrarily large $m \geq q$), and the near-boundary estimate above, we obtain for $t \geq 2$:
\begin{equation}
\label{grad11}
\begin{aligned}
\|\nabla {\bf w}_1\|_{L^q(\Omega)}
&\leq c\|{\bf v}\|_{L^{m}(\Omega_R)} \leq c(m) t^{-\frac{3}{2} \left( \frac{1}{r} - \frac{1}{m} \right)}\|{\bf u}_0\|_{\tilde{\mathbb{X}}^r}, \\
\|\Delta {\bf w}_1\|_{L^q(\Omega)}
&\leq  c\|{\bf v}\|_{L^q(\Omega_R)} + c\|\nabla{\bf v}\|_{L^q(\Omega_R)},\\
&\leq c(m)\left( t^{-1 - \frac{3}{2}\left(\frac{1}{r} - \frac{1}{q}\right)} + t^{-\frac{3}{2}\left(\frac{1}{r} - \frac{1}{m}\right)} \right)\|{\bf u}_0\|_{\tilde{\mathbb{X}}^r}, \\
\|\partial_t {\bf w}_1\|_{L^q(\Omega)}
&\leq c\|\partial_t{\bf v}\|_{L^q(\Omega_R)} \leq c t^{-1-\frac{3}{2} \left( \frac{1}{r} - \frac{1}{q} \right)}\|{\bf u}_0\|_{\tilde{\mathbb{X}}^r}.
\end{aligned}
\end{equation}
Let ${\bf w} = \zeta {\bf v} - {\bf w}_1$ and $q = \zeta\pi$. Then, $({\bf w}, q)$ satisfies the following system in $\mathbb{R}^3$:
\[
\begin{cases}
\text{div} \, {\bf w} = 0 \quad \text{in } {\mathbb R}^3, \\
\partial_t {\bf w }-\mu \Delta {\bf w} + \nabla q = -2\mu \nabla {\bf v} \cdot \nabla \zeta - \mu {\bf v} \Delta \zeta + \pi \nabla \zeta - \partial_t {\bf w}_1 + \mu \Delta {\bf w}_1 := \tilde{\bf f} \quad \text{in } {\mathbb R}^3, \\
{\bf w}|_{t=0} = \zeta {\bf v}_0 - \nabla N * ({\bf v}_0 \cdot \nabla) \zeta := \tilde{\bf w}_0.
\end{cases}
\]

The solution ${\bf w}(t)$ to the system in $\mathbb{R}^3$ can be expressed using the heat semigroup and the Helmholtz projection operator $\mathbb{Q}$:
\[
{\bf w}(t) = e^{\mu t\Delta} \tilde{\bf w}_0 + \int_0^t e^{\mu (t-s)\Delta}\mathbb{Q} \tilde{{\bf f}}(s) \,ds,
\]
where the components of $\mathbb{Q}$ are given by the Riesz transforms as $\mathbb{Q}_{ij} = \delta_{ij} + R_i R_j$. The heat semigroup $e^{\mu t\Delta}$ is defined by convolution with the standard heat kernel $\Gamma(x,t) = (4\mu \pi t)^{-3/2} \exp\left( -|x|^2 / (4\mu t) \right)$.

Using the well-known smoothing estimate for the heat operator, we obtain:
\begin{equation}
\label{grad2}
\|\nabla e^{\mu t\Delta} \tilde{\bf w}_0\|_{L^q({\mathbb R}^3)}
 \leq c t^{-\frac{1}{2} - \frac{3}{2} \left( \frac{1}{r} - \frac{1}{q} \right)}
  \|{\bf u}_0\|_{\tilde{\mathbb X}^r}.
\end{equation}

Next, we estimate the $L^q(B_R^c)$ norm of the integral term by splitting it into two parts, explicitly indicating the spatial and temporal dependence:
\begin{align*}
I(x,t) &= \int_1^t \nabla e^{\mu(t-s)\Delta} \mathbb{Q}\tilde{\bf f}(x,s) \,ds, \\
II(x,t) &= \int_0^1 \nabla e^{\mu(t-s)\Delta}\mathbb{Q}\tilde{\bf f}(x,s) \,ds.
\end{align*}

Note that $\tilde{\bf f}$ is compactly supported in $\Omega_R$, and we generally have the bound
\begin{align*}
\|\tilde{\bf f}(s)\|_{L^p} 
&\leq c \|\nabla {\bf v}(s)\|_{L^p(\Omega_R)} + c \|{\bf v}(s)\|_{L^p(\Omega_R)} + c \|\pi(s)\|_{L^p(\Omega_R)}.
\end{align*}
According to Lemma \ref{lem:pressure_accel_decomp}, the local pressure is bounded by the fractional Sobolev norm of the velocity gradient:
\begin{align}
\label{est:pressure_accel1}
\|\tilde{\bf f}(s)\|_{L^p} 
\leq c \|\nabla {\bf v}(s)\|_{W^{a,p}(\Omega_{2R})} + c \|{\bf v}(s)\|_{L^p(\Omega_R)}, \quad \text{for } \frac{1}{p} < a < 1.
\end{align}

\begin{itemize}
\item[(i)] \textbf{Estimate for large time $s \in [1, t]$ (Term $I$):}

To estimate the $L^q(B_R^c)$ norm of $I$, we follow the approach in Lemma 8.1 of \cite{maremonti-solonnikov0}. We observe that the combined gradient heat-Riesz kernel corresponding to the operator $\nabla e^{\mu (t-s)\Delta}{\mathbb Q}$ satisfies the well-known pointwise bound bounded by $c \left( |x - y|^2 +(t-s)\right)^{-2}$. 

For $|x| \geq 2R$ and $|y| \leq R$, we have $|x-y| \approx |x|$, which leads to the pointwise integral estimate:
\begin{align*}
|I(x,t)| 
&\leq c \int_1^t \int_{\frac{R}{2} \leq |y| \leq R} \left( |x|^2 +(t-s)\right)^{-2} |\tilde{\bf f}(y,s)| \, dy \,ds \\
&\leq c \int_1^t \left( |x|^2 +(t-s)\right)^{-2} \|\tilde{\bf f}(s)\|_{L^1(\Omega_R)} \,ds.
\end{align*}
(Here, we used the fact that $\tilde{\bf f}(y,s)$ is supported in $\Omega_R$, so the $dy$-integral yields the $L^1$-norm of $\tilde{\bf f}$.)

Now, taking the $L^q$-norm with respect to the spatial variable $x$ over $B_{2R}^c$, and applying Minkowski's integral inequality to pass the $L^q$-norm inside the time integral, we obtain:
\begin{align} 
\label{grad3} 
\|I(\cdot,t)\|_{L^q(B_{2R}^c)}  
&\leq c \int_1^t \left\| \left( |\cdot|^2 +(t-s)\right)^{-2} \right\|_{L^q(B_{2R}^c)} \|\tilde{\bf f}(s)\|_{L^1(\Omega_R)} \,ds \notag \\
&\leq c \int_1^t \left( \int_{B_{2R}^c} \left( |x|^2 +(t-s)\right)^{-2q} dx \right)^{\frac{1}{q}} \|\tilde{\bf f}(s)\|_{L^q(\Omega_R)} \,ds,
\end{align}
where in the last line we also used the embedding $L^q(\Omega_R) \hookrightarrow L^1(\Omega_R)$ on the bounded domain $\Omega_R$.

First, we evaluate the spatial $L^q$-norm of the kernel via polar coordinates:
\[
\left( \int_{B_{2R}^c} \left( |x|^2 +(t-s)\right)^{-2q} dx \right)^{\frac{1}{q}} \leq c \left( R^2 +(t-s)\right)^{-\frac{1}{2} \left( 4 - \frac{3}{q} \right)}.
\]

For large time $s > 1$, we do not need the fractional exponent trick and can directly use the full $W^{1,q}$ bound. By interpolation, we have:
\begin{align*}
\|\tilde{\bf f}(s)\|_{L^q(\Omega_R)} 
&\leq c \|\nabla^2{\bf v}(s)\|_{L^q(\Omega_R)} + c \|{\bf v}(s)\|_{L^{m}(\Omega_R)} \\
&\leq c(m)\left( s^{-1 - \frac{3}{2}\left(\frac{1}{r} - \frac{1}{q}\right)} 
+ s^{-\frac{3}{2}\left(\frac{1}{r} - \frac{1}{m}\right)} \right)\|{\bf u}_0\|_{\tilde{\mathbb X}^r}, \quad s > 1.
\end{align*}

Substituting these two estimates back into \eqref{grad3}, we conclude:
\begin{align*} 
\|I(\cdot,t)\|_{L^q(B_{2R}^c)}  
&\leq c(m) \int_1^t \left( R^2 +(t-s)\right)^{-\frac{1}{2} \left( 4 - \frac{3}{q} \right)} \left( s^{-1 - \frac{3}{2}\left(\frac{1}{r} - \frac{1}{q}\right)}  +s^{-\frac{3}{2}\left(\frac{1}{r} - \frac{1}{m}\right)} \right) ds \, \|{\bf u}_0\|_{\tilde{\mathbb X}^r} \notag \\ 
&\leq c(m)\left( t^{-\frac{3}{4} - \frac{3}{2}\left(\frac{1}{r} - \frac{1}{q}\right)}  + t^{-\frac{3}{2}\left(\frac{1}{r} - \frac{1}{m}\right)} \right) \|{\bf u}_0\|_{\tilde{\mathbb{X}}^r}. 
\end{align*}

\item[(ii)] \textbf{Estimate for short time $s \in [0, 1]$ (Term $II$):}

Next, we estimate the $L^q(B_R^c)$ norm of $II$. Here, we utilize the $L^r$--$L^q$ smoothing property of the heat operator. To prevent a non-integrable singularity at $s=0$, we rely on the fractional bound \eqref{est:pressure_accel1} with $p=r$. By applying the short-time estimates from Section \ref{smallt} and interpolating, we deduce:
\begin{align*}
\|\tilde{\bf f}(s)\|_{L^r}  
&\leq c \|\nabla {\bf v}(s)\|_{W^{a,r}(\Omega_{2R})} + c \|{\bf v}(s)\|_{L^r(\Omega_R)}, \quad \text{for } \frac{1}{r} < a < 1 \\
&\leq c\|\nabla^2{\bf v}(s)\|_{L^r(\Omega_{2R})}^{a}\|{\bf v}(s)\|_{L^r(\Omega_{2R})}^{1-a} + c \|{\bf v}(s)\|_{L^r(\Omega_R)} \\
&\leq c \big( s^{-1} \big)^a \big( s^{-\frac{1}{2}} \big)^{1-a} \|{\bf u}_0\|_{\tilde{\mathbb X}^r} + c s^{-\frac{1}{2}} \|{\bf u}_0\|_{\tilde{\mathbb X}^r} \\
&\leq c s^{-\frac{a+1}{2}} \|{\bf u}_0\|_{\tilde{\mathbb X}^r}.
\end{align*}
Since $a < 1$, the exponent satisfies $\frac{a+1}{2} < 1$, making the singularity integrable at $s=0$. Consequently, we obtain:
\begin{align}
\label{grad4}
\notag \|II(\cdot, t)\|_{L^q(B_{2R}^c)}
&\leq \int_0^1 (t-s)^{-\frac{1}{2} - \frac{3}{2} \left( \frac{1}{r} - \frac{1}{q} \right)} \|\tilde{\bf f}(s)\|_{L^r} \,ds \\
\notag&\leq c \int_0^1 (t-s)^{-\frac{1}{2} - \frac{3}{2} \left( \frac{1}{r} - \frac{1}{q} \right)} s^{-\frac{a+1}{2}} \,ds \, \|{\bf u}_0\|_{\tilde{\mathbb X}^r} \\
&\leq ct^{-\frac{1}{2} - \frac{3}{2} \left( \frac{1}{r} - \frac{1}{q} \right)} \|{\bf u}_0\|_{\tilde{\mathbb X}^r}.
\end{align}
\end{itemize}

Combining the estimates \eqref{grad11}, \eqref{grad2}, \eqref{grad3}, and \eqref{grad4}, we obtain the final global estimate for $\|\nabla {\bf v}\|_{L^q(B_{2R}^c)}$:
\begin{align*}
\|\nabla {\bf v}\|_{L^q(B_{2R}^c)} 
\leq c(m)\left( t^{-\frac{1}{2} - \frac{3}{2}\left(\frac{1}{r} - \frac{1}{q}\right)} 
+ t^{-\frac{3}{2}\left(\frac{1}{r} - \frac{1}{m} \right)} \right) \|{\bf u}_0\|_{\tilde{\mathbb X}^r}, \quad t \geq 2.
\end{align*}


\section{Proof of Theorem \ref{maximal.thm}: Uniform-in-time maximal regularity for $1<q<\frac{3}{2}$}
\label{section.maximal.global.thm}

This section is devoted to completing the proof of Theorem~\ref{maximal.thm}. Specifically, we establish the maximal $L^q$-regularity of the solution on the infinite time interval $(0, \infty)$ for the restricted range $1 < q < \frac{3}{2}$.

As in Section \ref{section.maximal.thm}, by a standard density argument it suffices to derive the uniform-in-time $L^q$-estimate (with a constant independent of $T > 0$) for smooth solutions corresponding to smooth, compactly supported data satisfying \eqref{n66}.

%
%

We write ${\bf u} \simeq ({\bf v}, \boldsymbol{\ell}, \boldsymbol{\omega})$.
To derive the uniform-in-time estimates in the $L^q$ setting, we begin with the \textit{a priori} estimate \eqref{w8} established in Lemma~\ref{prop.theorem.strong.Lp}.  
It suffices to show that 
\begin{align}\label{local_Lq_integral}
\int_0^T \|{\bf v}(t)\|_{L^q(\Omega_{2R})}^q \, dt \leq c \|{\bf u}_0\|_{L^q}^q
\end{align}
for some constant $c > 0$ independent of $T$. This reduction is justified by the standard embedding for the initial data, which yields $\|{\bf u}_0\|_{L^q} \leq c \|{\bf u}_0\|_{{\mathbb B}^{2-\frac{2}{q}}_q}$.

We decompose the solution as ${\bf u} = {\bf u}_1 + {\bf u}_2$, where ${\bf u}_1$ is the initial data part (namely, the solution to the homogeneous problem ${\bf f}={\bf 0}$ with initial data ${\bf u}_0$), and ${\bf u}_2$ is the forcing part (namely, the solution corresponding to zero initial data ${\bf u}_2(0)={\bf 0}$ with external force ${\bf f}$). 

Observe that for the initial part ${\bf u}_1$, we have
\begin{align}\label{u1_local_1}
\int_0^1 \|{\bf u}_1(t)\|_{L^q(\Omega_{2R})}^q \, dt 
\leq c \int_0^1 \|{\bf u}_0\|_{L^q}^q \, dt 
\leq c\|{\bf u}_0\|_{L^q}^q.
\end{align}
For $t > 1$, applying the local embedding $L^m(\Omega_{2R}) \hookrightarrow L^q(\Omega_{2R})$ and the $L^q$-$L^m$ decay estimate of the semigroup, we deduce
\begin{align}\label{u1_local_2}
\int_1^T \|{\bf u}_1(t)\|_{L^q(\Omega_{2R})}^q \, dt 
&\leq c \int_1^T \|{\bf u}_1(t)\|_{L^m(\Omega_{2R})}^q \, dt \notag\\
&\leq c \int_1^T t^{-\frac{3q}{2}(\frac{1}{q}-\frac{1}{m})} \, dt \, \|{\bf u}_0\|_{L^q}^q \notag\\
&\leq c \|{\bf u}_0\|_{L^q}^q.
\end{align}
Here, we note that the time integral over $(1, \infty)$ converges for any $1 < q < \infty$ by simply choosing the exponent $m > 3q$, which ensures $-\frac{3q}{2}(\frac{1}{q}-\frac{1}{m}) < -1$.

Next, we estimate the forcing part ${\bf u}_2$. For the initial short-time interval $0 \leq t \leq 2$, the standard $L^q \to L^q$ boundedness of the semigroup yields
\[
\|{\bf u}_2(t)\|_{L^q(\Omega_{2R})} \leq c \int_0^t \|{\bf f}(s)\|_{L^q(\Omega)} \, ds.
\]
Applying Young's convolution inequality (or Hölder's inequality on the bounded interval $[0,2]$), we immediately obtain
\begin{align}\label{u2_short_time}
\int_0^2 \|{\bf u}_2(t)\|_{L^q(\Omega_{2R})}^q \, dt \leq c \int_0^2 \|{\bf f}(t)\|_{L^q(\Omega)}^q \, dt.
\end{align}

Now, let $2 \leq t \leq T$. We split the time integral into a local part and a tail part to isolate the singularity near $t=s$:
\begin{align*}
\|{\bf u}_2(t)\|_{L^q(\Omega_{2R})} 
&\leq \int_{t-1}^t \| e^{-(t-s){\mathbb A}_{\alpha,q}} {\mathbb P} {\bf f}(s) \|_{L^q(\Omega_{2R})} \, ds 
 + c \int_0^{t-1} \| e^{-(t-s){\mathbb A}_{\alpha,q}} {\mathbb P} {\bf f}(s) \|_{L^m(\Omega_{2R})} \, ds \notag\\
&\leq c \int_{t-1}^t \|{\bf f}(s)\|_{L^q(\Omega)} \, ds 
 + c \int_0^{t-1} (t-s)^{-\frac{3}{2}\left(\frac{1}{q}-\frac{1}{m}\right)} \|{\bf f}(s)\|_{L^q(\Omega)} \, ds \notag\\
&:= I(t) + II(t).
\end{align*}

For $I(t)$, applying Hölder's inequality on the interval of length $1$ yields 
$I(t)^q \leq c \int_{t-1}^t \|{\bf f}(s)\|_{L^q(\Omega)}^q \, ds$. 
Integrating over $(2, T)$ and using Fubini's theorem, we have
\begin{align}\label{u2_I_est}
\int_2^T I(t)^q \, dt 
\leq c \int_2^T \left( \int_{t-1}^t \|{\bf f}(s)\|_{L^q(\Omega)}^q \, ds \right) dt 
\leq c \int_0^T \|{\bf f}(s)\|_{L^q(\Omega)}^q \, ds.
\end{align}

For $II(t)$, we recognize it as a convolution of $\|{\bf f}(\cdot)\|_{L^q(\Omega)}$ with the kernel $K(t) = t^{-\frac{3}{2}(\frac{1}{q}-\frac{1}{m})} {\bf 1}_{(1,\infty)}(t)$. 
By applying Young's convolution inequality, we arrive at:
\begin{align}\label{u2_II_est}
\left( \int_2^T II(t)^q \, dt \right)^{\frac{1}{q}} 
&\leq \|K\|_{L^1(1,\infty)} \left( \int_0^T \|{\bf f}(t)\|_{L^q(\Omega)}^q \, dt \right)^{\frac{1}{q}} \notag\\
&\leq c \left( \int_0^T \|{\bf f}(t)\|_{L^q(\Omega)}^q \, dt \right)^{\frac{1}{q}}.
\end{align}
Here, to ensure that the kernel $K$ is in $L^1(1,\infty)$, we require
\[
\frac{3}{2} \left( \frac{1}{q} - \frac{1}{m} \right) > 1 \quad \implies \quad \frac{1}{q} > \frac{2}{3} + \frac{1}{m}.
\]
Since $m \leq \infty$ can be chosen arbitrarily large, this strict inequality holds if and only if $q < \frac{3}{2}$. 

Combining \eqref{u2_short_time}, \eqref{u2_I_est}, and \eqref{u2_II_est}, we conclude the uniform-in-time estimate for ${\bf u}_2$:
\begin{align*}
\int_0^T \|{\bf u}_2(t)\|_{L^q(\Omega_{2R})}^q \, dt \leq c \int_0^T \|{\bf f}(t)\|_{L^q(\Omega)}^q \, dt.
\end{align*}

\appendix

\numberwithin{equation}{section}
\numberwithin{theo}{section}
\setcounter{equation}{0}
\section{  Proof of Lemma \ref{prop.ogawa}: Strong \(L^p\) Estimate}
\label{appendix.prop.ogawa}


The following propositions are classical results and will be used in the proof of Lemma~\ref{prop.ogawa}.
\begin{prop}[{See 
 \cite[Theorem IV.2.1]{galdi}  
}]
\label{appen.lemm0}

Let $1 < r < \infty$. Suppose that ${\bf f }\in {\bf L}^r(\mathbb{R}^3)$, $g \in W^{1,r}(\mathbb{R}^3)$, and that $({\bf u} , p)$ is a generalized solution of
\[
-\mu\Delta {\bf u}  + \nabla p = {\bf  f},\quad \text{div}\,{\bf u}  = { g} \quad \text{in } \mathbb{R}^3,
\]
satisfying $D^{2}{\bf u} , \nabla p \in {\bf L}^r(\mathbb{R}^3)$. Then, there exists a constant $c = c(r,\mu) > 0$ such that
\[
\|D^{2} {\bf u} \|_{L^r(\mathbb{R}^3)} + \|\nabla^{} p\|_{L^r(\mathbb{R}^3)} \leq c\left( \|{\bf f}\|_{L^r(\mathbb{R}^3)} + \|\nabla g\|_{L^r(\mathbb{R}^3)} \right).
\]
\end{prop}

\begin{prop}[ {See \cite[Theorem IV.6.1]{galdi}
} ]

\label{appen.lemm1}

Let $1 < r < \infty$ and let $\mathcal{O}\subset \mathbb{R}^3$ be a bounded domain of class $C^{2}$ (or class $C^{1,1}$). Suppose that ${\bf f} \in {\bf L}^r(\mathcal{O})$ and $g \in W^{1,r}(\mathcal{O})$ satisfy the compatibility condition $\int_{\mathcal{O}} g(x)\,dx = 0$. Let $({\bf u} , p) \in {\bf W}^{2,r}(\mathcal{O}) \times W^{1,r}(\mathcal{O})$ be the generalized solution of$$\begin{cases}
-\mu\Delta {\bf u}  + \nabla p = {\bf f}, \quad \mathrm{div}\,{\bf u}  = g  \text{ in } \mathcal{O}, \\
{\bf u}  = 0  \text{ on } \partial \mathcal{O}.
\end{cases}$$Then, there exists a constant $c = c(r, \mu, \mathcal{O}) > 0$ such that$$\|{\bf u} \|_{W^{2,r}(\mathcal{O})} + \|p\|_{W^{1,r}(\mathcal{O})/\mathbb{R}} \leq c \left( \|{\bf f}\|_{L^r(\mathcal{O})} + \|g\|_{W^{1,r}(\mathcal{O})} \right).$$

%
\end{prop}
%

%
\begin{prop}[{See \cite[Theorem 5.10]{Acevedo2021}}]
\label{appen.lemm2}
Let $\mathcal{O} \subset \mathbb{R}^3$ be a bounded domain of class $C^{2}$ and let $1 < r < \infty$. Assume that the friction coefficient $\alpha \in C^{0,1}(\partial\mathcal{O})$ satisfies $\alpha \geq 0$ with $\alpha \not\equiv 0$ on $\partial\mathcal{O}$ (in particular, a constant $\alpha > 0$). Suppose that ${\bf f} \in {\bf L}^r(\mathcal{O})$ and $g \in W^{1,r}(\mathcal{O})$ satisfy $\int_{\mathcal{O}} g(x)\,dx = 0$. If $({\bf u}, p) \in {\bf W}^{2,r}(\mathcal{O}) \times W^{1,r}(\mathcal{O})$ is a solution to the Stokes problem
\[
\begin{cases}
-\mu\Delta{\bf u} + \nabla p = {\bf f}, \quad \mathrm{div}\,{\bf u} = g & \text{in } \mathcal{O}, \\
{\bf u} \cdot {\bf n} = 0, \quad [2\mu\mathbb{D}({\bf u}){\bf n}]_\tau + \alpha{\bf u}_\tau ={\bf 0} & \text{on } \partial\mathcal{O},
\end{cases}
\]
then there exists a constant $c = c(r, \mu, \mathcal{O}, \alpha) > 0$ such that
\[
\|D^2{\bf u}\|_{L^r(\mathcal{O})} + \|\nabla p\|_{L^r(\mathcal{O})} \leq c\left( \|{\bf f}\|_{L^r(\mathcal{O})} + \|\nabla g\|_{L^r(\mathcal{O})} \right).
\]
\end{prop}
\begin{rem}
Although the results cited in Propositions~\ref{appen.lemm0}--\ref{appen.lemm2} are stated in \cite{galdi,Acevedo2021} for $g = 0$, the inhomogeneous case $g \neq 0$ follows immediately by standard lifting arguments (via the Riesz transforms in $\mathbb{R}^3$, and trace lifting combined with the Bogovski\u{\i} operator \cite[Theorem III.3.3]{galdi} in $\mathcal{O}$).
\end{rem}

The following local estimate for the pressure is essential for our analysis. A key feature of this estimate is that it relies solely on the interior structure of the Stokes equations and the local integrability of the velocity gradient, making it independent of the specific boundary conditions prescribed on $\partial\Omega$.
\begin{prop}
\label{appen.lemm-1}
Let $\Omega$ be an exterior domain such that $\Omega^c \subset B_{R_0}$ for some $R_0 > 0$. As in Lemma \ref{prop.ogawa}, for $R > 3R_0$, let $\Omega_R = \Omega \cap B_R$.

Suppose that $({\bf u}, p)$ is a solution to \eqref{stokes} in $\Omega$ with $\nabla {\bf u} \in {\bf L}^r(\Omega)$ and $p \in L^r_{\text{loc}}(\Omega)$ for $1 < r < \infty$. Assuming $\int_{\Omega_R} p \, dx = 0$, there exists a constant $c$, depending only on $r, R$, and the geometry of $B$, such that
$$
\|p\|_{L^r(\Omega_R)} \leq c \left( \|{\bf f}\|_{L^r(\Omega_R)} + \|\nabla {\bf u}\|_{L^r(\Omega_R)} \right).
$$
\end{prop}

\begin{proof}
Since this estimate is fundamental but its independence from boundary conditions is sometimes overlooked, we provide a brief justification for the reader's convenience.

For the bounded domain $\Omega_R = \Omega \cap B_R$, where $\Omega^c \subset B_R$, the $L^r$-norm of $p$ is characterized by the duality identity
$$\|p\|_{L^r(\Omega_R)} = \sup_{0 \neq \phi \in {C}^\infty_{0}(\Omega_R)} \frac{\langle p, \phi \rangle}{\|\phi\|_{L^{r'}(\Omega_R)}}.$$
Note that $\langle p, \phi \rangle = \langle p, \phi - \frac{1}{|\Omega_R|}\int_{\Omega_R}\phi \, dx\rangle$ since $\int_{\Omega_R} p \, dx = 0$.
For each such $\phi$, the Bogovski\u{\i} operator yields a vector field $\boldsymbol{\psi} \in {\bf W}^{1,r'}_0(\Omega_R)$ satisfying
$$\operatorname{div}\boldsymbol{\psi} = \phi - \frac{1}{|\Omega_R|}\int_{\Omega_R}\phi \, dx \quad \text{in } \Omega_R, \quad \boldsymbol{\psi}|_{\partial\Omega_R} = 0, \quad \|\nabla\boldsymbol{\psi}\|_{L^{r'}(\Omega_R)} \leq c\|\phi\|_{L^{r'}(\Omega_R)}.$$
By testing the Stokes equation with $\boldsymbol{\psi}$, the pressure $p$ is bounded by the local $L^r$-norms of ${\bf f}$ and $\nabla {\bf u}$. Because $\boldsymbol{\psi}$ vanishes on the entire boundary $\partial\Omega_R$, this argument depends only on the local structure of the operator and the geometry of $\Omega_R$. Consequently, the estimate remains valid regardless of the specific boundary conditions on $\partial\Omega$, following the standard approach in Galdi \cite[Lemma~IV.1.1]{galdi}.

\end{proof}
%

Using the four Propositions above, we obtain the following estimate.

\begin{lemm}
\label{prop.ogawa1}
Let $\Omega$ be the exterior domain and $R$ be the radius as in Proposition \ref{appen.lemm-1}.
Suppose that ${\bf f} \in L^r(\Omega)$ and that $({\bf u} , p)$ is a solution of \eqref{stokes} with ${\bf u}  \in {\bf L}^r_{\text{loc}}(\Omega)$. Then, $D^2 {\bf u} , \nabla p \in{\bf  L}^r(\Omega)$, and there exists a constant $c > 0$ such that
\begin{align}
\label{num.prop.ogawa}
\|D^2 {\bf u} \|_{L^r(\Omega)} + \|\nabla p\|_{L^r(\Omega)} + \|\nabla {\bf u} \|_{L^r(\Omega_R)} + \|p\|_{L^r(\Omega_R)/\mathbb{R}}  \\[1em]
\notag \qquad\qquad \leq C\left( \|{\bf f}\|_{L^r(\Omega)} + \|{\bf u} \|_{L^r(\Omega_R)} \right).
\end{align}
Here,
\[
\|p\|_{L^r(\Omega_R)/\mathbb{R}} := \inf_{k \in \mathbb{R}} \|p + k\|_{L^r(\Omega_R)}.
\]
\end{lemm}


\begin{proof}
Take a cut-off function $\zeta,\xi\in C^\infty_0({\mathbb R}^3)$ with
$\zeta=1$ on $B_{\frac{R}{2}}, $ $\zeta=0$ on $B_R^c$, $\xi=1$ on $B_{\frac{R}{3}}$ and $\xi=0$ on $B_{\frac{2R}{3}}$.
Observe that
${\bf u} ={\bf u} \xi+{\bf u} \zeta (1-\xi)+{\bf u} (1-\zeta)(1-\xi):={\bf u} _1+{\bf u} _2+{\bf u} _3$,
 and $p=p\xi+p\zeta (1-\xi)+p(1-\zeta)(1-\xi):=p_1+p_2+p_3$.
Then, $({\bf u} _1,p_1)$ satisfies 
\begin{align}
\left\{\begin{array}{l}
-\mu\Delta{\bf u}_1+\nabla p_1=  {\bf f}_1,   
\operatorname{div}{\bf u} _1=g_1,\mbox{ in }\Om_R,\\[1em]
{\bf u} _1\cdot {\bf  n}=0, 2\mu  [{\mathbb D}({\bf u} _1){\bf n}]_\tau+\alpha [{\bf u} _1]_\tau=0\mbox{ on }\pa \Om_R,
\end{array}\right. 
\end{align}
$({\bf u} _2,p_2)$ satisfies 
\begin{align}
\left\{\begin{array}{l}
-\mu\Delta {\bf u} _2+\nabla p_2= {\bf f}_2,\
\operatorname{div}{\bf u} _2=g_2,\mbox{ in }\Om_R,\\
{\bf u} _2=0,\mbox{ on }\pa \Om_R,
\end{array}\right. 
\end{align}
and $({\bf u} _3,p_3)$ satisfies \begin{align}
-\mu\Delta {\bf u} _3+\nabla p_3={\bf f}_3,  
\operatorname{div}{\bf u} _3=g_3,\mbox{ in }{\mathbb R}^3.
\end{align}
where ${\bf f}_1{\bf f}_2,{\bf f}_3$ are defined by 
\begin{align*}{\bf f}_1&=\xi {\bf f}-2\nabla {\bf u} \cdot \nabla \xi-{\bf u}  \Delta  \xi+(p+k)\nabla \xi,\\
{\bf f}_2&=\zeta (1-\xi){\bf f}
-2\nabla {\bf  u} \cdot \nabla(\zeta (1-\xi))-{\bf u}  \Delta (\zeta (1-\xi))+(p+k)\nabla(\zeta (1-\xi)),\\
{\bf f}_3&=(1-\zeta) (1-\xi){\bf f}-2\nabla {\bf u} \cdot  \nabla \Big((1-\xi)(1-\zeta)\Big)-{\bf u}  \Delta \Big(  (1-\xi) (1-\zeta) \Big)
\\
&\qquad+(p+k)\nabla \Big(  (1-\xi)(1-\zeta)\Big).
\end{align*}
and $g_1,g_2, g_3$ are defined by 
\begin{align*}
g_1={\bf u} \cdot \nabla \xi,\ 
g_2={\bf u} \cdot \nabla(\zeta (1-\xi)),\ g_3={\bf u} \cdot \nabla\Big(  (1-\xi)(1-\zeta)\Big).
\end{align*}
Here, $k$ is any constant.

According to Proposition \ref{appen.lemm0}-Proposition \ref{appen.lemm2} we have
\begin{align*}
&\|\nabla^2{\bf u}\|_{L^p(\Om_R)}+\|\nabla p\|_{L^p(\Om_R)}\\
&\leq c\|{\bf f}\|_{L^p(\Om)} +c\|\nabla {\bf u} \|_{L^p(\Om_R)}+c\|{\bf u} \|_{L^p(\Om_R)}
+c\|p\|_{L^p(\Om_R)/{\mathbb R}}.
\end{align*}

If we choose $k=-\int_{\Omega_R}pdx$, then $\int_{\Omega_R}p+k dx=0$, hence, according to Lemma \ref{appen.lemm-1}, 
\[
\|p\|_{L^r(\Omega_R)/{\mathbb R}}\leq c\|{\bf f}\|_{L^r(\Omega)}+c\|\nabla {\bf u}\|_{L^r(\Omega_R)}.
\]

By the well-known Gagliardo–Nirenberg inequality in bounded domains (see \cite{gagliardo}), we have
\[
\|\nabla {\bf u} \|_{L^r(\Omega_R)} \leq c \|D^2{\bf  u}\|_{L^r(\Omega_R)}^{1/2} \|{\bf u} \|_{L^r(\Omega_R)}^{1/2} + \|{\bf u} \|_{L^r(\Omega_R)}.
\]
For any $\epsilon > 0$, there exists a constant $C_\epsilon > 0$ such that
\[
\|\nabla {\bf u} \|_{L^r(\Omega_R)} \leq \epsilon \|D^2 {\bf u} \|_{L^r(\Omega_R)} + C_\epsilon \|{\bf u} \|_{L^r(\Omega_R)}.
\]

Combining the estimates above and taking $\epsilon > 0$ sufficiently small, we obtain the inequality \eqref{num.prop.ogawa} stated in Lemma ~\ref{prop.ogawa1}.
\end{proof}

%
%
\begin{lemm}
\label{poincare}
Let $\Omega \subset \mathbb{R}^3$ be an exterior domain associated with a bounded rigid body, such that $\Omega^c \subset B_R$. For any ${\bf u} \in C^\infty_0(\overline{\Omega})$ satisfying the boundary condition ${\bf u} \cdot {\bf n} = 0$ on $\partial \Omega$, there exists a constant $C > 0$ such that
\[
\|{\bf u}\|_{L^q(\Omega_R)} \leq C(R) \|\nabla {\bf u}\|_{L^q(\Omega_R)}.
\]
\end{lemm}

\begin{proof}
We proceed by contradiction. Suppose there exists a sequence $\{{\bf u}_k\} \subset {\bf C}^\infty_0(\overline{\Omega}_R)$ such that
\[
\|{\bf u}_k\|_{L^q(\Omega_R)} = 1, \quad \|\nabla{\bf u}_k\|_{L^q(\Omega)} \leq \frac{1}{k}, \quad \text{and} \quad{\bf u}_k \cdot {\bf n} = 0 \text{ on } \partial \Omega.
\]
By the Rellich-Kondrachov compactness theorem, there exists a subsequence (still denoted by $\{{\bf u}_k\}$) and a vector field ${\bf u} \in W^{1,q}(\Omega_R)$ such that
\[
{\bf u}_k \rightarrow {\bf u} \text{ strongly in } L^q(\Omega_R) \quad \text{and} \quad \nabla {\bf u}_k \rightharpoonup \nabla {\bf u} \text{ weakly in } L^q(\Omega_R).
\]
By the lower semi-continuity of the norm, we have 
\[\|\nabla {\bf u}\|_{L^q(\Omega_R)} \leq \liminf_{k \to \infty} \|\nabla {\bf u}_k\|_{L^q(\Omega_R)} = 0,\]
which implies that ${\bf u}$ is a constant vector, say ${\bf u} \equiv {\bf c}$, on each connected component of $\Omega_R$.
Furthermore, the strong convergence implies $\|{\bf u}\|_{L^q(\Omega_R)} = 1$, hence ${\bf c} \neq {\bf 0}$. 

By the trace theorem, the boundary condition ${\bf u}_k \cdot {\bf n} = 0$ is preserved in the limit, yielding ${\bf c} \cdot {\bf n}(x) = 0$ for all $x \in \partial \Omega$. 
We recall that $\Omega$ is an exterior domain whose boundary $\partial \Omega$ is the surface of a bounded rigid body, which forms a compact, closed surface in $\mathbb{R}^3$. Because $\partial \Omega$ is a compact closed surface
, the Gauss map is surjective onto the unit sphere. Thus, there must exist at least one point $x_0 \in \partial \Omega$ where the unit normal vector ${\bf n}(x_0)$ is parallel to ${\bf c}$ (i.e., ${\bf n}(x_0) = \pm {\bf c}/|{\bf c}|$). At this point, we have $|{\bf c} \cdot {\bf n}(x_0)| = |{\bf c}|$, which implies ${\bf c} = {\bf 0}$. This contradicts the fact that $\|{\bf u}\|_{L^q(\Omega_R)} = 1$, completing the proof.
\end{proof}

\begin{rem}

Lemma \ref{poincare} implies that the local $L^q$-norm of ${\bf u}$  on the right-hand side of (ii) in Lemma \ref{prop.ogawa} can be controlled by the local $L^q$-norm of $\nabla {\bf u}$

%


\end{rem}

We are now ready to prove the estimate in Lemma \ref{prop.ogawa}.  
If ${\bf u}  \in L^r(\Omega)$ and q  $\in [r, \infty],$ then
$\|{\bf u} \|_{L^r(\Omega_R)} \leq C(R) \|{\bf u} \|_{L^q(\Omega_R)}.$
Applying this inequality to \eqref{num.prop.ogawa} in Lemma~\ref{prop.ogawa1} directly leads to estimate (ii) in Lemma~\ref{prop.ogawa}.

To complete the proof of Lemma \ref{prop.ogawa}, it remains to establish estimate (i) in the case $r < \frac{3}{2}$.  
For this purpose, we make use of the following uniqueness result:

\begin{theo}
\label{prop.unique}
Let $\Omega$ be an exterior domain in $\mathbb{R}^3$, and let ${\bf u} $ be a weak solution to the Stokes system
\begin{align}
\label{stokes11}
\left\{
\begin{array}{ll}
-\mu \Delta{\bf u} + \nabla p = 0 & \text{in } \Omega, \\
\mathrm{div}\, {\bf u}  = 0 & \text{in } \Omega, \\
{\bf u}  \cdot {\bf n} = 0, \quad  2\mu  [{\mathbb D}({\bf u} ){\bf n}]_\tau+\alpha [{\bf u}  ]_\tau={\bf 0}& \text{on } \partial \Omega, \\
{\bf u}  \in L^r(\Omega) \text{ for some }1  \leq r < \infty,
\end{array}
\right.
\end{align}
with $D^2{\bf u} \in{\bf  L}^q(\Omega)$ for some $q \geq  1$.
Then, ${\bf u}  \equiv 0$ in $\Omega$.
\end{theo}
\begin{proof}
According to Theorem V.3.3 in \cite{galdi}, if ${\bf u} $ satisfies $(\nabla{\bf u}, \nabla \boldsymbol{\varphi})=[{\bf f},\boldsymbol{\varphi}]$ for all $\boldsymbol{\varphi}\in {\bf C}^\infty_{0,\sigma}(\Omega)$ with $D^2{\bf  u}\in L^q(\Omega)$ and compactly supported ${\bf f}\in L^q(\Omega)$ for some $1<q<\infty,$ then
\[
{\bf u} (x)={\bf u}_0+{\mathbb U}_0\cdot x+{\bf u}^{(1)}(x),\quad  p(x)=p_0+p^{(1)}(x),
\]
where $\mbox{\rm trace}({\mathbb U}_0)=0$, ${\bf u}_0, p_0$ are constants, and
\[
D^m {\bf u}^{(1)}(x)=O(|x|^{-1-m}), \quad D^m p^{(1)}(x)=O(|x|^{-2-m}), \quad m\geq 0.
\]
From the condition ${\bf u} \in L^r(\Omega)$ for some $r<\infty$, we deduce that ${\bf u}_0=0$ and ${\mathbb U}_0=0$. Hence $D^m {\bf u} (x)=O(|x|^{-1-m})$. 
Furthermore, since the pressure is unique up to an additive constant in the Stokes problem, we can uniquely fix the constant $p_0 = 0$ by imposing the condition $p(x) \to 0$ as $|x| \to \infty$. This implies that the pressure itself decays as $p(x) = O(|x|^{-2})$.

Now, multiplying ${\bf u}$ to the momentum equation $\eqref{stokes11}_1$ and integrating by parts over $\Omega_R$, we deduce that
\begin{align*}
2\mu \int_{\Omega_R}|{\mathbb D}({\bf u} )|^2 dx+\alpha \int_{\partial \Omega}|[{\bf u}]_\tau|^2 dS_x
&=\int_{\partial \Omega_R} {\bf u} \cdot {\mathbb S}({\bf u} ,p){\bf n} dS_x \\
&\leq c \int_{\partial \Omega_R} |{\bf u}| \big( |\nabla {\bf u}| + |p| \big) dS_x \\
&\leq c\int_{\partial \Omega_R} |x|^{-1} \big( |x|^{-2} + |x|^{-2} \big) dS_x \\
&\leq c\int_{\partial \Omega_R} |x|^{-3} dS_x \rightarrow 0 \quad \mbox{as } R\rightarrow \infty.
\end{align*}
This leads to the conclusion that ${\bf u} \equiv {\bf 0}$.
\end{proof}

We claim that
\[
\|D^2 {\bf u} \|_{L^p(\Omega)} + \|\nabla p\|_{L^p(\Omega)} \leq c \|{\bf f}\|_{L^p(\Omega)}, \quad 1<p<\frac{3}{2}.
\]

We employ a standard contradiction argument (cf.~\cite{kozono-sohr, giga-sohr}).
Assume, for contradiction, that the estimate fails. Then, there exists a sequence \( \{{\bf f}_m\} \subset L^p(\Omega) \) with \( \|{\bf f}_m\|_{L^p(\Omega)} \leq \frac{1}{m} \), and corresponding solutions \( ({\bf u} _m, p_m) \) to the system \eqref{stokes}, such that
\[
\|D^2{\bf u}_m\|_{L^p(\Omega)} + \|\nabla p_m\|_{L^p(\Omega)} = 1.
\]
This normalization is introduced in order to apply Lemma~\ref{prop.ogawa1} and obtain a uniform lower bound for $\|{\bf u} _m\|_{L^p(\Omega_R)}$. Indeed, Lemma~\ref{prop.ogawa1} yields
\[
1 \le \frac{C}{m} + c\|{\bf u} _m\|_{L^p(\Omega_R)},
\]
so $\|{\bf u} _m\|_{L^p(\Omega_R)}$ cannot converge to zero. This nontrivial lower bound will later lead to a contradiction with the compactness argument.

To apply compactness, we first note that the combined normalization $\|D^2 {\bf u}_m\|_{L^p(\Omega)} + \|\nabla p_m\|_{L^p(\Omega)} = 1$ guarantees, via the Sobolev embedding, that the sequence $\{{\bf u}_m\}$ is locally bounded in $L^{p^*}(\Omega_R)$, where 
\[
p^* = \frac{3p}{3 - 2p} \quad \text{if } p < \frac{3}{2}, \quad \text{and } p^* = \infty \quad \text{otherwise.}
\]
This provides the necessary local velocity bound for the compactness argument.

By compactness, there exists a subsequence (still indexed by \( m \)) and a limit pair \( ({\bf u}, p) \) such that
\begin{align*}
D^2 {\bf u} _m &\rightharpoonup D^2 {\bf u}  \quad \text{weakly in } L^p(\Omega), \\
{\bf u} _m &\to {\bf u}  \quad \text{strongly in } L^r(K), \quad r < p^{*}, \\
\nabla p_m &\rightharpoonup \nabla p \quad \text{weakly in } L^p(\Omega),
\end{align*}
for any compact \( K \subset \Omega \), where $ p^{*} = \frac{3p}{3 - 2p} $ if $p < \frac{3}{2}, $ and $p^{*} = \infty$ otherwise.

Passing to the limit, the limit pair \( ({\bf u} , p) \) satisfies the Stokes system \eqref{stokes11}. Although the data vanish in the limit, the solution remains nontrivial due to the preserved estimate
\[
1 \leq c \|{\bf u} \|_{L^p(\Omega_R)},
\]
which follows from the strong convergence \( {\bf u} _m \to {\bf u}  \) in \( L^r(K) \) for any compact \( K \subset \Omega \) and \( r < p^{*} \).

Meanwhile, by lower semi-continuity, we have
\[
\|D^2 {\bf u} \|_{L^p(\Omega)} + \|\nabla p\|_{L^p(\Omega)} \leq \liminf_{m \to \infty} \left( \|D^2 {\bf u} _m\|_{L^p(\Omega)} + \|\nabla p_m\|_{L^p(\Omega)} \right) = 1.
\]

If \( p < \frac{3}{2} \), the Sobolev embedding yields
\[
\nabla^2{\bf u}  \in L^p(\Omega) \quad \Rightarrow \quad {\bf u}  \in L^{\frac{3p}{3 - 2p}}(\Omega),
\]
so \( {\bf u}  \in L^r(\Omega) \) for some \( r < \infty \). Then, by the uniqueness result in Theorem~\ref{prop.unique}, we conclude \( {\bf u}  \equiv 0 \), contradicting the estimate above.

Therefore, the original estimate must hold.

\section{Proof of Lemma \ref{lemma.uniqueness}}
\label{appendix.uniqueness}


Consider the fundamental solutions
\[
N(x) = \frac{1}{4\pi|x|} \quad \text{and} \quad E_\lambda(x) = \frac{e^{-\sqrt{\lambda}|x|}}{4\pi|x|} \quad (\mathrm{Re}\,\sqrt{\lambda} > 0)
\]
of the operators $-\Delta$ and $\lambda - \Delta$ in $\mathbb{R}^3$, respectively.
The fundamental solution of the resolvent Stokes equations is represented by
\[
{\bf G}^i_\lambda := \nabla \times \nabla \times (N * E_\lambda {\bf e}^i) = E_\lambda {\bf e}^i + \nabla \left( \frac{\partial}{\partial x_i} N * E_\lambda \right), \quad
Q^i := -\partial_{x_i} N(x),
\]
which satisfies
\[
\lambda {\bf G}^i_\lambda - \Delta {\bf G}^i_\lambda + \nabla Q^i = \delta(x){\bf e}^i, \quad \text{div} {\bf G}^i_\lambda = 0.
\]
(The construction is completely analogous to the parabolic case discussed in Chapter 4 of \cite{lady}.)

Let $\zeta_R$ be a cut-off function such that $\zeta_R = 1$ on $B_R$ and $\zeta_R = 0$ on $B_{2R}$. Let ${\bf v}=(v_1,v_2,v_3)$. Then, each component $v_i$, $i=1,2,3$ is represented by the following integral equation:
\begin{align}
\label{integral_stokes_resolvent}
\begin{split}
v_i(x)\zeta_R(x) &= -\int_{\partial \Omega} {\bf G}^i_\lambda(x-y) \cdot \mathbb{S}({\bf v},\pi)(y){\bf n} \\
&\qquad \qquad- {\bf v}(y) \cdot \mathbb{S}({\bf G}^i_\lambda,Q^i)(x-y) {\bf n} \, dS_y \\
&\quad + \int_{\Omega} {\bf v}(y) \cdot \left[ (\nabla \zeta_R(y)) \partial_{y_i} N(x-y) - (\nabla \zeta_R(y)) \times \nabla \times (N(x-y){\bf e}^i) \right] dy \\
&\quad + \int_{\Omega} \sum_{j=1}^2 {\bf v}(y) \cdot R^i_{j,\lambda}(x,y) \, dy \\
&= v^{(1)}_i + v^{(2)}_i + v^{(3)}_i,
\end{split}
\end{align}
where
\begin{align*}
R^i_{1,\lambda}(x,y) &= -2 \sum_{k=1}^3 \nabla \partial_{y_k} \zeta_R(y) \left[ \nabla_y \times (\partial_{y_k} N * E_\lambda(x-y) {\bf e}^i) \right] \\
&\qquad \qquad- \nabla \Delta \zeta_R(y) \times \left[ \nabla_y N * E_\lambda(x-y) \times {\bf e}^i \right], \\
R^i_{2,\lambda}(x,y) &= -2 (\nabla \zeta_R(y) \cdot \nabla_y) {\bf G}^i_\lambda(x-y) - \Delta \zeta_R(y) {\bf G}^i_\lambda(x-y).
\end{align*}

Direct computation shows the asymptotic behavior for the fundamental solutions in the resolvent case:
\[
|\nabla^k (N * E_\lambda)(x-y)| \leq c(\lambda) |x-y|^{-1-k}.
\]

Let $R \geq 2|x|$. Then, we estimate the second term:
\begin{align*}
|\nabla^2{v}^{(2)}_i(x)| &\leq c R^{-1} \int_{R \leq |y| \leq 2R} |{\bf v}(y)| |x-y|^{-4} \, dy \\
&\leq c \|{\bf v}\|_{{\bf L}^s(\Omega)} R^{-2-\frac{3}{s}} \to 0 \quad \text{as } R \to \infty,
\end{align*}
and similarly for the third term:
\begin{align*}
|\nabla^2v^{(3)}_i(x)| &\leq c(\lambda) \sum_{k=1}^3 R^{-3+k} \int_{R \leq |y| \leq 2R} |\nabla^2{\bf v}(y)| |x-y|^{-2-k} \, dy \\
&\leq c(\lambda) \|\nabla^2{\bf v}\|_{{\bf L}^q(\Omega)} R^{-\frac{3}{q}-2} \to 0 \quad \text{as } R \to \infty.
\end{align*}

Let ${\bf v}^{(j)}=(v^{(j)}_1,v^{(j)}_2,v^{(j)}_3)$ for $j=1,2,3$.
Since \( {\bf v} - {\bf v}^{(1)} = {\bf v}^{(2)} + {\bf v}^{(3)} \) and the left-hand side does not depend on \( R \), it follows that
\[
\nabla^2 {\bf v}^{(2)} + \nabla^2 {\bf v}^{(3)} \equiv 0.
\]
Therefore,
\[
{\bf v}(x) = {\bf a} + {\mathbb B} \cdot x + {\bf v}^{(1)}(x),
\]
where \({\bf a} \in \mathbb{C}^3 \) and \( {\mathbb B} \in \mathbb{C}^{3 \times 3} \) are constant coefficients.

The assumption \( {\bf v} \in {\bf L}^s(\Omega) \) for some \( s < \infty \) implies that both \({\bf a} = {\bf 0} \) and \( {\mathbb B} = {\bf 0} \).

This leads to the following representation for the pressure \( \pi \):
\[
\pi(x) = p_0 + \int_{\partial \Omega} {\bf Q}(x-y)\cdot {\mathbb S}({\bf v}(y),\pi(y)){\bf n} - {\bf v}(y)\cdot \nabla {\bf Q}(x-y) {\bf n}\, dS_y.
\]
Here, ${\bf Q}=(Q^i)_{i=1}^3$ and $p_0 \in \mathbb{C}$ is a constant.

Now, let $|x| \geq R$ with $\Omega^c \subset B_{R/2}$. Then, direct computation shows the spatial decay rates:
\begin{align*}
|\nabla^k {\bf v}^{(1)}(x)| &\leq c(\lambda) |x|^{-3-k} \left( \|\nabla^2 {\bf v}\|_{{\bf L}^q(\Omega)} + \|\nabla \pi\|_{{\bf L}^q(\Omega)} \right), \\
|\nabla^k (\pi(x) - p_0)| &\leq c |x|^{-1-k} \left( \|\nabla^2 {\bf v}\|_{{\bf L}^q(\Omega)} + \|\nabla \pi\|_{{\bf L}^q(\Omega)} \right).
\end{align*}

Note that the divergence-free condition $\text{div} \overline{\bf v} = 0$ and the boundary condition on $\partial \Omega$ imply $\int_{\partial B_R} {\bf n} \cdot \overline{\bf v} \, dS_x = \int_{\partial \Omega} {\bf n} \cdot \overline{\bf v} \, dS_x = 0$. Thus, we can replace $\pi$ with $\pi - p_0$ in the boundary integral over $\partial B_R$.
Multiplying the homogeneous resolvent system by $\overline{\bf v}$ (the complex conjugate of ${\bf v}$), integrating over $\Omega_R = \Omega \cap B_R$, and taking the real part yields:
\begin{align*}
&\text{Re}(\lambda) \left( \int_{\Omega_R} |{\bf v}|^2 dx + m|\boldsymbol{\ell}|^2 + |\mathbb{J}_0||\boldsymbol{\omega}|^2 \right) \\
&+ 2\mu \int_{\Omega_R} |\mathbb{D}({\bf v})|^2 dx
+ \alpha \int_{\partial \Omega} |[{\bf v} -\boldsymbol{ \ell} -\boldsymbol{ \omega }\times x]_\tau|^2 dS \\
&= \text{Re} \left( \int_{\partial B_R} (2\mu\mathbb{D}({\bf v}){\bf n} - (\pi - p_0){\bf n}) \cdot \overline{\bf v} \, dS_x \right) \\
&\leq c \int_{\partial B_R} |{\bf v}| (|\nabla {\bf v}| + |\pi - p_0|) \, dS_x \leq c R^{-1} \to 0 \quad \text{as } R \to \infty.
\end{align*}

Letting $R \to \infty$, for any $\lambda \in \mathbb{C}$ with $\text{Re}(\lambda) \geq 0$, we obtain the exact energy identity:
\begin{align*}
&\text{Re}(\lambda) \left( \int_{\Omega} |{\bf v}|^2 dx + m|\boldsymbol{\ell}|^2 + |\mathbb{J}_0||\boldsymbol{\omega}|^2 \right) \\
&+ 2\mu \int_{\Omega} |\mathbb{D}({\bf v})|^2 dx
+ \alpha \int_{\partial \Omega} |[{\bf v} -\boldsymbol{ \ell} -\boldsymbol{ \omega }\times x]_\tau|^2 dS = 0.
\end{align*}

This immediately implies that $\mathbb{D}({\bf v}) \equiv 0$, meaning ${\bf v}$ must be a rigid body motion. Since ${\bf v} \in {\bf L}^s(\Omega)$ in the exterior domain $\Omega$, it follows that ${\bf v} \equiv {\bf 0}$. The boundary conditions then yield $\boldsymbol{\ell} + \boldsymbol{\omega} \times x = {\bf 0}$ on $\partial \Omega$, which implies $\boldsymbol{\ell} = {\bf 0}$ and $\boldsymbol{\omega} = {\bf 0}$. Thus, we conclude that
\[
 ({\bf v}, \boldsymbol{\ell},   \boldsymbol{\omega}) \equiv ({\bf 0}, {\bf 0}, {\bf 0}).
\]

\bigskip
\noindent


\end{document}